\documentclass[a4paper,11pt,leqno]{article}

\usepackage[english]{babel}
\usepackage{amsmath,amsthm,amssymb,thmtools,enumitem,mathrsfs,mathtools,url,doi,titlesec,bm}
\usepackage[left=2.5cm,right=2.5cm,top=2.5cm,bottom=2.5cm,bindingoffset=0cm]{geometry}

\setlist[enumerate]{topsep=0pt,label=\textup{(\arabic*)},leftmargin=\parindent,labelsep=.5em}
\setlist{noitemsep}

\usepackage{quoting}
\quotingsetup{vskip=0pt}

\usepackage[T1]{fontenc}

\usepackage{tikz}
\usetikzlibrary{arrows}

\declaretheoremstyle[
  spaceabove=\topsep, spacebelow=6pt,
  headfont=\normalfont\bfseries,
  notefont=\mdseries, notebraces={(}{)},
  bodyfont=\normalfont\itshape,
  postheadspace=1em,
  qed=\qedsymbol
]{mystyle}

\declaretheoremstyle[
  spaceabove=\topsep, spacebelow=6pt,
  headfont=\normalfont\bfseries,
  notefont=\mdseries, notebraces={(}{)},
bodyfont=\normalfont,
  postheadspace=1em,
  qed=\qedsymbol
]{mydefstyle}

\swapnumbers
\theoremstyle{mystyle}
\declaretheorem[numberlike=subsection]{proposition}
\declaretheorem[numberlike=subsection]{theorem}
\declaretheorem[numberlike=subsection]{corollary}
\declaretheorem[numberlike=subsection]{lemma}

\declaretheorem[numbered=no,name=Main Theorem]{MainThm}

\theoremstyle{mydefstyle}
\declaretheorem[numberlike=subsection]{definition}
\declaretheorem[numberlike=subsection]{example}
\declaretheorem[numberlike=subsection]{remark}
\declaretheorem[numberlike=subsection]{remarks}

\numberwithin{equation}{subsection}

\titleformat{\section}[block]
  {\filcenter\normalfont\large\bfseries}{\thesection.}{1em}{}
\titleformat{\subsection}[runin]
  {\normalfont\bfseries}{\thesubsection}{.5em}{}

\titlespacing*{\section} {0pt}{6ex plus 1ex minus .2ex}{3 ex plus .2ex}
\titlespacing*{\subsection} {0pt}{\topsep}{1em}

\input{BMmacsLatex.sty}

\begin{document}

\begin{center}
\textbf{\Large ON THE TATE AND MUMFORD-TATE CONJECTURES}
\medskip

\textbf{\Large IN CODIMENSION ONE FOR VARIETIES WITH $h^{2,0} =1$}
\bigskip

\textit{by}
\bigskip

{\Large Ben Moonen}
\end{center}
\vspace{8mm}

{\small 

\noindent
\begin{quoting}
\textbf{Abstract.} We prove the Tate conjecture for divisor classes and the Mumford-Tate conjecture for the cohomology in degree~$2$ for varieties with $h^{2,0}=1$ over a finitely generated field of characteristic~$0$, under a mild assumption on their moduli. As an application of this general result, we prove the Tate and Mumford-Tate conjectures for several classes of algebraic surfaces with $p_g=1$.
\medskip

\end{quoting}

} 
\vspace{8mm}


\section*{Introduction}

\subsection{}
In this paper we study the Tate conjecture for divisor classes on varieties over a finitely generated field of characteristic zero, henceforth simply referred to as ``the Tate conjecture''. Whereas, from a modern perspective, the Hodge-theoretic analogue---the Lefschetz theorem on divisor classes---is quite easy to prove, it is an uncomfortable fact that the Tate conjecture is known only for some rather special classes of varieties. For abelian varieties, Faltings proved it in 1983, alongside with the Mordell conjecture and the Shafarevich conjecture. For K3 surfaces, the Tate conjecture was proven independently by Andr\'e and Tankeev. For Hilbert modular surfaces the Tate conjecture is known by work of Harder, Langlands and Rapoport, completed by results of, independently, Klingenberg and Kumar Murty and Ramakrishnan. In general, however, the Tate conjecture remains widely open.

In view of the Lefschetz theorem on divisor classes, the Tate conjecture is implied by the Mumford-Tate conjecture for cohomology in degree~$2$. This conjecture is not even known for abelian varieties, though it is known for K3 surfaces, again by Andr\'e and Tankeev.

Our main contribution in this paper is a proof of the Mumford-Tate conjecture for the cohomology in degree~$2$, and hence the Tate conjecture for divisor classes, for varieties with $h^{2,0}=1$, under a mild assumption on their moduli:

\begin{MainThm}
\label{MainThmIntro}
Let $X$ be a non-singular complete variety over~$\mC$ with $h^{2,0}(X)=1$. Assume there exists a smooth projective family $f \colon \cX \to S$ over a non-singular irreducible base variety~$S$ such that $X \cong \cX_\xi$ for some $\xi \in S(\mC)$, and such that the variation of Hodge structure $R^2f_* \mQ_\cX$ is not isotrivial. Then the Tate Conjecture for divisor classes on~$X$ is true and the Mumford-Tate conjecture for the cohomology in degree~$2$ is true.
\end{MainThm}

\noindent
We in fact prove a more general result, that applies to submotives of an~$H^2$; for this we refer to Theorem~\ref{MainThm}.

\subsection{}
In the last section we apply our main theorem to algebraic surfaces with $p_g=1$. Given some irreducible component of the moduli space of such surfaces, the challenge is to show that the Hodge structure on the~$H^2$ is not constant over it. If this holds, the Mumford-Tate conjecture is true for all surfaces in this moduli component. In some cases, the fact that the $H^2$ is non-constant is contained in the literature, but we also treat some cases where additional geometric arguments are needed. Theorem~\ref{MainThmSurf} gives the list of cases we have worked out thus far; with minor exceptions this includes all moduli components for which a good description is available in the literature. We hope that the experts in this area can supply such a description in many more cases, and we expect that our main theorem can serve as a standard tool to then deduce the Mumford-Tate conjecture.

\subsection{}
Let us now sketch some of the main ideas involved in the proof of the Main Theorem.

Assume a situation as in the statement of the theorem. Possibly after replacing~$S$ with a finite cover we have a decomposition $R^2f_* \mQ_\cX(1) = \mQ_S^\rho \oplus \mV$, where $\rho$ is the generic Picard number in the family and $\mV$ is a variation of Hodge structure over~$S$ such that on a very general fibre~$\mV_s$ there are no non-zero Hodge classes. On~$\mV$ we have a symmetric polarization form~$\phi$.

A central role in the paper is taken by the endomorphism algebra 
$E = \End_{\QVHS_S}(\mV)$. By result of Zarhin, $E$ is a field that is either totally real or a CM-field. Even for a concretely given family, it is usually very hard to determine~$E$. Rather than attempting to do so, we leave $E$ as an unknown, and we see how far we can get. The proof of the main theorem then splits up in three cases that each require a different approach:
\begin{enumerate}
\item $E$ totally real, $\rank_E(\mV) \neq 4$;
\item $E$ a CM-field;
\item $E$ totally real, $\rank_E(\mV) = 4$.
\end{enumerate}

\subsection{}
The proof in case~(1) may be viewed as an extension of the arguments due to Deligne~\cite{DelK3}, later refined by Andr\'e in~\cite{YAK3}, who used this to prove the Tate conjecture for  hyperk\"ahler varieties. Our main new contributions in this case are twofold:
\begin{enumerate}[label=\textup{(\alph*)}]
\item We develop a variant of the Kuga-Satake construction so as to take into account a non-trivial totally real field of endomorphisms.
\item We introduce the systematic use of norm functors (or: corestrictions); this is a technique of independent interest.
\end{enumerate}

In order to explain this in more detail, let us briefly go back to the arguments of Andr\'e in the case of K3 surfaces. We take for $\cX \to S$ a universal family of (polarized) K3's over the moduli space. Consider the even Clifford algebra $C^+(\mV,\phi)$, which is an algebra in the category~$\QVHS_S$ of variations of Hodge structure over~$S$. The Kuga-Satake construction produces an abelian scheme $\pi\colon A\to S$ equipped with an action by (an order in) a semisimple algebra~$D$ such that we have an isomorphism of algebras in $\QVHS_S$,
$$
u\colon C^+(\mV,\phi) \isomarrow \ul\End_D(R^1\pi_*\mQ_A)\, .
$$
In particular, if we write $V=\mV_\xi$ and $H= H^1(A_\xi,\mQ)$, we have an isomorphism $u_\xi\colon C^+(V,\phi) \isomarrow \ul\End_D(H)$. However, the whole point of doing the Kuga-Satake construction in a family is that on the source and target of~$u_\xi$ we now have an action of $\pi_1(S,\xi)$.

In the case of K3 surfaces, it is known that the monodromy action on $H^2(X,\mQ)\bigl(1\bigr)$ is ``big'', which means that the algebraic monodromy group is the full group $\SO(V,\phi)$. (For $\cX/S$ a universal family of polarized K3's we have $\rho=1$ and $V$ is a subspace of $H^2(X,\mQ)\bigl(1\bigr)$ of codimension~$1$.) A rather spectacular consequence of this is that the above isomorphism~$u_\xi$ is the Hodge realization of an isomorphism of motives (in the sense of Andr\'e~\cite{YAPour}) $\motu_\xi\colon C^+(\motV_\xi,\phi) \isomarrow \ul\End_D\bigl(\motH^1(A_\xi)\bigr)$. In a somewhat implicit form this idea is already present in Deligne's paper~\cite{DelK3}; it relies on the characterization of~$u_\xi$ as the {\it unique\/} $\pi_1$-equivariant algebra isomorphism. It is made explicit by Andr\'e in \cite{YAK3}, Section~6.2, and is essentially a consequence of the Theorem of the fixed part.

Once we have the motivic isomorphism~$\motu_\xi$, we can take $\ell$-adic realizations and apply Faltings's result for abelian varieties. To conclude the Tate conjecture for $X = \cX_\xi$ we still have to ``extract''~$\motV_\xi$ from the even Clifford algebra $C^+(\motV_\xi,\phi)$. If the dimension of~$\motV_\xi$ is odd, this is relatively easy; if $\dim(\motV_\xi)$ is even, a further trick is needed. Andr\'e's proof of the Mumford-Tate conjecture then still requires further ideas, which we shall not review here.

\subsection{} 
Once we leave the realm of hyperk\"ahler varieties, we no longer dispose of a ``big monodromy'' result of the sort used in the above argument. This is where the field $E$ comes in.

Consider a situation as in the Main Theorem, and assume we are in case~(1). The generic Mumford-Tate group of the variation~$\mV$ is then the group $\SO_E(V,\phi)$ of $E$-linear isometries with determinant~$1$. We prove that the monodromy of~$\mV$ is ``maximal'', in the sense that the algebraic monodromy group $G_\mono(\mV)$ equals the generic Mumford-Tate group. A natural idea, then, is to imitate the above argument, using a Kuga-Satake construction ``relative to the field~$E$''.

In order to make this work, we need some new techniques. The point is that the Kuga-Satake construction is highly non-linear. (Step one: form the even Clifford algebra.) To overcome this, we make systematic use of norm functors. In brief: whenever we are in a Tannakian category~$\cC$ (Hodge structures, Galois representations, motives,...) there is a norm functor from the category~$\cC_{(E)}$ of $E$-modules in~$\cC$ to $\cC$ itself. This is an extremely natural and useful construction that appears to be not so widely known. In Section~\ref{Norms} we explain this, building upon the work of Ferrand~\cite{Ferrand}.

Once we have norms at our disposal, the correct replacement for the even Clifford algebra $C^+(\motV_\xi,\phi)$ is not (still with $E$ totally real) simply the even Clifford algebra of $\motV_\xi$ over~$E$ (which is an $E$-algebra in the category of motives) but rather its norm $\FNorm_{E/\mQ} C^+_E(\motV_\xi,\phi)$. Once we have this working, we are back on the trail paved for us by Andr\'e. The new, ``relative'', version of the Kuga-Satake construction produces an abelian scheme $\pi\colon A\to S$ with an action by a semisimple algebra~$D$ and an isomorphism of algebras in $\QVHS_S$,
$$
u\colon \FNorm_{E/\mQ} C^+_E(\mV,\phi) \isomarrow \ul\End_D\bigl(\mH^1(A)\bigr)\, ,
$$
where we write $\mH^1(A)=R^1\pi_*\mQ_A$. The maximality of the monodromy allows us to lift this to a motivic level, pass to $\ell$-adic realizations, and use Faltings's result. Though some technical details get more involved than in the case $E=\mQ$, from this point on everything works as expected.

\subsection{}
In case~(2), when $E$ is a CM-field, we use a different approach. This should not come as a surprise: the Kuga-Satake construction is based on the idea that we can lift a Hodge structure from a special orthogonal group to a spin group, whereas in the CM case we are dealing with unitary groups. Instead, we establish a direct relation between the motive $\motH^2(X)\bigl(1\bigr)$ and a motive of the form $\ul\Hom_E\bigl(\motH^1(A),\motH^1(B)\bigr)$, where $A$ and~$B$ are abelian varieties with $E$-action. In fact, $A$ is a fixed abelian variety of CM-type, depending only on~$E$ and the choice of a CM-type, and if we vary~$X$, only $B$ varies.

On the level of Hodge realizations, we show that we can find $A$ and~$B$ such that $H^2(X)\bigl(1\bigr) \cong \ul\Hom_E\bigl(H^1(A),H^1(B)\bigr)$. This gives a new interpretation of what van Geemen~\cite{vGeemen} calls a ``half-twist''. The construction can be done in families, and using a monodromy argument we prove that $\motH^2(X)\bigl(1\bigr)$ is isomorphic to $\ul\Hom_E\bigl(\motH^1(A),\motH^1(B)\bigr)$, up to twisting by a $1$-dimensional motive~$\motU$. We expect $\motU$ to be trivial, but we are unable to prove this. (This is one of the main reasons why, in our main theorem, we are not yet able to prove the ``motivic Mumford-Tate conjecture'', which is the additional statement that the Mumford-Tate group equals the motivic Galois group.) We can, however, show that $\motU$ has trivial Hodge and $\ell$-adic realizations; this allows to apply the results of Faltings and to deduce the main theorem.

\subsection{}
What remains is case~(3), when $E$ is totally real and the local system~$\mV$ has rank~$4$ over~$E$. In this case, the algebraic monodromy group~$G_\mono$ may be strictly smaller than the generic Mumford-Tate group $\SO_E(V,\phi)$. As the proof in case~(1) crucially relies on the maximality of the monodromy, this argument breaks down in an essential way. (If $G_\mono = \SO_E(V,\phi)$, the argument of case~(1) works, so case~(3) is really about the situation where the monodromy is non-maximal.)

We are able to deal with the case of non-maximal monodromy by combining ideas from the proofs of cases (1) and~(2). The algebraic monodromy group is the group of norm~$1$ elements in a quaternion algebra~$\Delta$ over~$E$. With $D = \FNorm_{E/\mQ}(\Delta)$, we construct a complex abelian variety~$A$ and an abelian scheme $B \to S$, both with an action of~$D^\opp$, such that we have an isomorphism $u \colon \ul\Hom_D\bigl(\mH^1(A_S),\mH^1(B)\bigr) \isomarrow \FNorm_{E/\mQ}(\mV)$. After a rather minute analysis of all groups involved, and using the information that we still have about the monodromy, we are able to show that the fibre~$u_\xi$ of~$u$ at the point~$\xi$ is the Hodge realization of a motivic isomorphism $\motu_\xi \colon \ul\Hom_D\bigl(\motH^1(A),\motH^1(B_\xi)\bigr) \isomarrow \FNorm_{E/\mQ}(\motV_\xi)$. This is of course the crucial step, as now we may again invoke the results of Faltings. We reduce the proof of the Mumford-Tate conjecture for $X = \cX_\xi$ to the MTC for the abelian variety $A\times B_\xi$, and while this is not a case already contained in the literature, there is enough information available to prove this by fairly direct arguments.

\subsection{}
Let us now give a brief overview of the individual sections. In Section~\ref{Conjectures} we review the Tate and Mumford-Tate conjectures. Working systematically over finitely generated fields (rather than only number fields) has the advantage that these conjectures can be stated for any variety over~$\mC$, but apart from choices in the presentation we do not claim any originality here.

In Section~\ref{Zarhlad} we review the results of Zarhin on Mumford-Tate groups of Hodge structures of K3 type, which are crucial for everything that follows, and we prove an $\ell$-adic analogue of this, using a result of Pink. This fills what seems to be a gap in Andr\'e's paper~\cite{YAK3}; see Remark~\ref{YAGap}.

As already mentioned, in Section~\ref{Norms} we discuss norm functors. This can be read independently of the rest of the paper and is of interest in a much more general setting.

Sections \ref{KugaSat} and~\ref{RealMult} contain the constructions needed to deal with case~(1). These two sections are closest to the work of Deligne and Andr\'e, the main point being that we develop a variant of the Kuga-Satake construction relative to a totally real endomorphism field. In Section~\ref{Mono=>MTC} we then state the main result in its general form, and we prove some preliminary results. At the end of this section we complete the proof in case~(1). In Section~\ref{Mono=>MTC2} and the somewhat long Section~\ref{Mono=>MTC3} we deal with cases (2) and~(3), respectively.

In Section~\ref{AlgSurf}, finally, we prove the MTC for several classes of algebraic surfaces with $p_g=1$.

\subsection{}
\textit{Acknowledgements.} I am much indebted to Y.~Andr\'e and P.~Deligne, who have greatly influenced my understanding of the notions that play a central role in this paper. Further, I thank J.~Commelin, W.~Goldring, C.~Peters, R.~Pignatelli and Q.~Yin for inspiring discussions and helpful comments.

\subsection{}
\label{NotConvIntro}
\textit{Notation and conventions.} (a) By a Hodge structure of K3 type we mean a polarizable $\mQ$-Hodge structure of type $(-1,1) + (0,0) + (1,-1)$ with Hodge numbers $1,n,1$ for some~$n$. By a VHS of K3 type over some base variety~$S$ we mean a polarizable variation of Hodge structure whose fibers are of K3 type.

(b) In the first eight sections, we always view abelian schemes over a base scheme~$S$ as objects of the category $\QAV_S$ of abelian schemes up to isogeny.

(c) Let $k$ be a field of characteristic~$0$ and $E$ a finite \'etale $k$-algebra. If $V$ is an $E$-module of finite rank equipped with a nondegenerate $E$-bilinear form $\tilde\phi \colon V \times V \to E$ then $\phi = \trace_{E/k} \circ \tilde\phi$ is a nondegenerate $k$-bilinear form on~$V$ with the property that $\phi(ev,w) = \phi(v,ew)$ for all $v,w \in V$ and $e \in E$. (Terminology: $\phi$ is the transfer of~$\tilde\phi$.) Conversely, given a nondegenerate $k$-bilinear form $\phi \colon V \times V \to k$ with $\phi(ev,w) = \phi(v,ew)$, there is a unique $E$-bilinear form $\tilde\phi$ on~$V$ with $\phi = \trace_{E/k} \circ \tilde\phi$, and $\tilde\phi$ is again nondegenerate. We refer to~$\tilde\phi$ as the $E$-bilinear lift of~$\phi$. The uniqueness implies that if $\phi$ is symmetric or alternating, so is~$\tilde\phi$.

More generally, if $E$ comes equipped with an involution $e \mapsto \bar{e}$ and $\phi$ satisfies $\phi(ev,w) = \phi(v,\bar{e}w)$ for all $v,w \in V$ and $e \in E$ then there is a unique hermitian form $\tilde\phi \colon V \times V \to E$ such that $\phi = \trace_{E/k} \circ \tilde\phi$. In this setting we refer to~$\tilde\phi$ as the $E$-valued hermitian lift of~$\phi$.

(d) We shall often consider algebraic groups that are obtained via a restriction of scalars, and it will be convenient to simplify the notation for such groups. As a typical example, in the situation described in~(c) we have an orthogonal group $\OO(V,\tilde\phi)$ over~$E$ and we denote by $\OO_{E/k}(V,\phi)$ the algebraic group over~$k$ obtained from it by restriction of scalars. Similarly, if $E$ comes equipped with an involution and $\psi$ is a hermitian form with respect to this involution, we denote by $\UU_{E/k}(V,\psi)$ the corresponding unitary group, viewed as an algebraic group over~$k$ through restriction of scalars. (Note that in this case the restriction of scalars goes from the fixed algebra $E_0 \subset E$ of the involution to~$k$.) 

(e) Let $k$ be a field and $E$ a finite \'etale $k$-algebra. We denote the torus $\Res_{E/k} \mG_{\mult,E}$ by~$T_E$. In particular, $T_k = \mG_{\mult,k}$. The norm map defines a homomorphism $T_E \to T_k$, whose kernel we denote by~$T_E^1$.


\section{Review of some cycle conjectures}
\label{Conjectures}

\subsection{}
\label{MotivesYA}
Let $K \subset \mC$ be a subfield that is finitely generated over~$\mQ$. We denote by $\Mot_K$ the category of motives over~$K$ as defined by Y.~Andr\'e in~\cite{YAPour}. (As ``base pieces'' we take all projective smooth $K$-schemes.) This is a semisimple Tannakian category whose identity object is denoted by~$\unitmot$.

We use bold letters ($\motV$, $\motW$, ...) for motives. Their Hodge realizations and $\ell$-adic realizations are denoted by the corresponding characters of regular weight (in the sense of typography!) with a subscript~``$\Betti$'' or~``$\ell$'' ($V_\Betti$, $W_\Betti$, ..., respectively $V_\ell$, $W_\ell$,~...). If $\motW$ is a motive with Hodge realization~$W_\Betti$, we usually simply write~$W$ for the underlying $\mQ$-vector space and we write 
\[
G_\Betti(\motW) \subset G_\mot(\motW) \subset \GL(W)
\]
for the Mumford-Tate group and the motivic Galois group. We identify the $\Ql$-vector space underlying the $\ell$-adic realization with $W_\ell = W \otimes \Ql$ via the comparison isomorphism between Betti and \'etale cohomology. The $\ell$-adic realization is then a Galois representation $\rho_{\motW,\ell} \colon \Gal(\Kbar/K) \to \GL(W_\ell)$, and we denote by $G_\ell(\motW)$ the Zariski closure of the image of~$\rho_{\motW,\ell}$. We have
\[
G_\ell^0(\motW) \subset G_\mot(\motW) \otimes \Ql \subset \GL(W) \otimes \Ql = \GL(W_\ell)\, .
\]

\subsection{}
\label{MTConjecture}
Let $\motV$ be a motive over~$K$. The Mumford-Tate conjecture for $\motV$ is the assertion
\[
\hbox{We have $G_\ell^0(\motV) =  G_\Betti(\motV_\mC) \otimes \Ql$ as subgroups of $\GL(V_\ell)$.} \leqno{\hbox{MTC}(\motV):}
\]

Note that, a priori, this conjecture depends on the chosen embedding $K \hookrightarrow \mC$. It also depends on the choice of the prime number~$\ell$. In the rest of the paper, we fix~$\ell$ and whenever we refer to the Mumford-Tate conjecture it is with reference to this prime number. Our results are valid for all~$\ell$.

\begin{proposition}
\label{FieldExtProp}
Let $K \subset L$ be subfields of~$\mC$ that are finitely generated over~$\mQ$. Let $\motV$ be a motive over~$K$, let $\motV_L$ be the motive over~$L$ obtained by extension of scalars, and write $V_{L,\ell}$ for its $\ell$-adic realization.

\textup{(\romannumeral1)} The isomorphism $\GL(V_\ell\bigr) \isomarrow \GL(V_{L,\ell})$ induced by the canonical isomorphism $V_\ell \isomarrow V_{L,\ell}$ restricts to an isomorphism $G_\ell^0(\motV) \isomarrow G_\ell^0(\motV_L)$.

\textup{(\romannumeral2)} With respect to the chosen embeddings $K \hookrightarrow L \hookrightarrow \mC$ we have $\hbox{\rm MTC}(\motV_L) \Leftrightarrow \hbox{\rm MTC}(\motV)$.
\end{proposition}

\begin{proof}
Let $K^\prime$ be the algebraic closure of~$K$ in~$L$. Then $K \subset K^\prime$ is a finite extension and the natural homomorphism $r \colon \Gal(\Lbar/L) \to \Gal(\Kbar/K)$ has as its image the subgroup $\Gal(\Kbar/K^\prime) \subset \Gal(\Kbar/K)$. Further, the diagram
\[
\begin{matrix}
\Gal(\Lbar/L) & \mapright{\rho_{\ell,\motV_L}} & \GL(V_{L,\ell}) \cr
\mapdownl{r} && \mapdownr{\wr} \cr
\Gal(\Kbar/K) & \sizedmapright{\rho_{\ell,\motV_L}}{\rho_{\ell,\motV}} & \GL(V_\ell)
\end{matrix}
\]
is commutative. This gives (\romannumeral1), and (\romannumeral2) is an immediate consequence.
\end{proof}

\subsection{}
\label{TateHodge}
Let $X$ be a complete non-singular variety over~$K$. For some integer $i \geq 0$, consider the motive $\motH = \motH^{2i}(X)\bigl(i\bigr) = (X,\pi_{2i},i)$, with $\pi_{2i}$ the K\"unneth projector in degree~$2i$.

An element $\xi \in H_\ell = H^{2i}\bigl(X_\Kbar,\Ql(i)\bigr)$ is called a Tate class if it is invariant under some open subgroup of $\Gal(\Kbar/K)$; this is equivalent to the requirement that $\xi$ is invariant under $G_\ell^0(\motH)$. Let $\cT^i(X) \subset H_\ell$ be the subspace of Tate classes. We have a cycle class map
\begin{equation}
\label{eq:CycClassMap}
\class\colon \CH^i(X_\Kbar) \otimes \Ql \to \cT^i(X)\, . 
\end{equation}
The Tate conjecture for cycles of codimension~$i$ on~$X$ is the assertion
\[
\hbox{The group $G_\ell^0(\motH)$ is reductive and the map \eqref{eq:CycClassMap} is surjective.} \leqno{\hbox{TC}^i(X):}
\]
The reductivity of $G_\ell^0(\motH)$ is equivalent to the condition that the representation~$\rho_{\motH,\ell}$ is completely reducible. 

Similarly, an element $\xi \in H_\Betti = H^{2i}\bigl(X_\mC,\mQ(i)\bigr)$ is called a Hodge class if $\xi$ is purely of type $(0,0)$ in the Hodge decomposition, which is equivalent to the condition that $\xi$ is invariant under $G_\Betti(\motH)$. Writing $\cB^i(X) \subset H_\Betti$ for the subspace of Hodge classes we have a cycle class map $\class\colon \CH^i(X_\mC)  \to \cB^i(X)$ and the Hodge conjecture asserts that this map is surjective.

\begin{proposition}
\label{FieldExtProp2}
Let $K \subset L$ be a finitely generated field extension. Let $X$ be a complete non-singular variety over~$K$. Then $\hbox{\rm TC}^i(X_L) \Leftrightarrow \hbox{\rm TC}^i(X)$.
\end{proposition}

\begin{proof}
For the complete reducibility of the Galois representation this is immediate from Proposition~\ref{FieldExtProp}(\romannumeral1), as this assertion only depends on $G_\ell^0(\motH)$ and $G_\ell^0(\motH_L)$. For the surjectivity of the cycle class map, the implication~``$\Leftarrow$'' is clear. For the converse we may assume $K \subset L$ is a principal field extension. If $K \subset L$ is a finite extension, it is clear that $\hbox{TC}^i(X_L)$ implies $\hbox{TC}^i(X)$. Hence it suffices to prove the assertion in the situation that $L$ is the function field of a curve~$C$ over~$K$ with a $K$-rational point $t \in C(K)$. The assertion now follows from the compatibility of the cycle class map with specialization. (Use \cite{SGA45}, Cycle, Th\'eor\`eme~2.3.8.) 
\end{proof}

\subsection{}
The above results allow us to formulate the Mumford-Tate conjecture and the Tate conjecture for motives over~$\mC$. 

If $\motV$ is a motive over~$\mC$, choose a subfield $K \subset \mC$ that is finitely generated over~$\mQ$ and a motive~$\motW$ over~$K$ with $\motW_\mC \cong \motV$. Define $G_\ell^0(\motV)$ to be $G_\ell^0(\motW)$, viewed as an algebraic subgroup of $\GL(V_\ell)$ via the comparison isomorphism $W_\ell \isomarrow V_\ell$. By Proposition~\ref{FieldExtProp}, this is independent of the choice of~$K$ and~$\motW$. The Mumford-Tate conjecture for~$\motV$ is then the assertion that $G_\Betti(\motV) \otimes \Ql = G_\ell^0(\motV)$ as subgroups of $\GL(V_\ell)$. For any choice of $K$ and~$\motW$ as above this is equivalent to the Mumford-Tate conjecture for~$\motW$.

Next let $X$ be a complete non-singular variety over~$\mC$. For $i\geq 0$, let $\motH = \motH^{2i}(X)\bigl(i\bigr)$, and let $\cT^i(X) \subset H_\ell$ be the subspace of $G_\ell^0(\motH)$-invariants. Then $\hbox{TC}^i(X)$ is the assertion that $G_\ell^0(\motH)$ is reductive and that $\class\colon \CH^i(X) \otimes \Ql \to \cT^i(X)$ is surjective. For any form~$X_K$ of~$X$ over a finitely generated field~$K$, this is again equivalent to the Tate conjecture on cycles of codimension~$i$ for~$X_K$.

\subsection{}
\label{TCMotives}
As long as we do not know that motivated cycles are algebraic, it does not make sense to formulate the Tate conjecture for arbitrary motives. In this paper we shall make one exception to this. Namely, suppose we are given a submotive $\motV \subset \motH^2(X)\bigl(1\bigr)$ for some complete non-singular complex variety~$X$. As before, we define the space of Tate classes $\cT(\motV) \subset V_\ell$ to be the space of $G_\ell^0(\motV)$-invariants in~$V_\ell$. Then by the Tate conjecture for~$\motV$ we mean the assertion that $G_\ell^0(\motV)$ is reductive and that the composition 
$$
\pr \circ \class \colon \CH^i(X) \otimes \Ql \to \cT^1(X) \twoheadrightarrow \cT(\motV)
$$
is surjective. By the Lefschetz theorem on divisor classes, surjectivity of this map is equivalent to the assertion that $\cB(\motV) \otimes \Ql = \cT(\motV)$ as subspaces of~$V_\ell$, where $\cB(\motV) = V^{G_\Betti(\motV)}$ is the space of Hodge classes. 

If $\motV$ is given as a submotive of~$\motH^2(X)$ then by the Tate conjecture for~$\motV$ we mean the Tate conjecture for~$\motV(1)$.

\begin{remarks}
\label{MTCforV(n)}
(\romannumeral1) If the Mumford-Tate conjecture is true for some motive~$\motV$, it is also true for any submotive $\motV^\prime \subset \motV$. Indeed, we can decompose $\motV = \motV^\prime \oplus \motV^\pprime$; then $G_\Betti(\motV) \subset G_\Betti(\motV^\prime) \times G_\Betti(\motV^\pprime)$ and $G_\ell^0(\motV) \subset G_\ell^0(\motV^\prime) \times G_\ell^0(\motV^\pprime)$, and in both cases the projection to the first factor is surjective.

(\romannumeral2) If the Mumford-Tate conjecture is true for~$\motV$, it is also true for any Tate twist $\motV(n)$. To see this, we first note that the Mumford-Tate conjecture can also be phrased as the conjectural equality $\mathfrak{g}_\Betti(\motV) \otimes \Ql = \mathfrak{g}_\ell(\motV)$ of Lie subalgebras of $\End(V_\ell)$. 

View $G_\Betti\bigl(\motV \oplus \unitmot(1)\bigr)$ as an algebraic subgroup of $G_\Betti(\motV) \times \mG_\mult$. If $\motV$ has weight zero then $G_\Betti\bigl(\motV \oplus \unitmot(1)\bigr) = G_\Betti(\motV) \times \mG_\mult$; if the weight is not zero, the first projection $G_\Betti\bigl(\motV \oplus \unitmot(1)\bigr) \to G_\Betti(\motV)$ is an isogeny and hence it gives an isomorphism on Lie algebras. For the groups~$G_\ell^0$ the analogous assertions are true. Further,
$$
G_\Betti\bigl(\motV \oplus \unitmot(1)\bigr) \isomarrow G_\Betti\bigl(\motV(1) \oplus \unitmot(1)\bigr)
\quad\hbox{and}\quad
G_\ell^0\bigl(\motV \oplus \unitmot(1)\bigr) \isomarrow G_\ell^0\bigl(\motV(1) \oplus \unitmot(1)\bigr)
$$
where in both cases the isomorphism is the one induced from the isomorphism $\GL(V) \times \mG_\mult \isomarrow \GL\bigl(V(1)\bigr) \times \mG_\mult$ given by $(A,c) \mapsto (cA,c)$. Combining these remarks we see that $\hbox{\rm MTC}(\motV)$ is equivalent to $\hbox{MTC}\bigl(\motV(n)\bigr)$.
\end{remarks}

\subsection{}
\label{StrictSubmot}
Let $X$ be a non-singular projective variety of dimension~$d$ over a field~$K$ of characteristic~$0$. Let $\pi_j$ denote the K\"unneth projector that cuts out the motive $\motH^j(X)$. For $i \geq 0$ and $n \in \mZ$ we have maps $\CH^d(X\times X) \otimes \mQ \to \End_{\Mot_L}\bigl(\motH^{2i}(X)(n)\bigr)$ given by $\gamma \mapsto \pi_{2i} \circ \class(\gamma) \circ \pi_{2i}$. We say that a submotive $\motV \subset \motH^{2i}(X)\bigl(n\bigr)$ is cut out by an an algebraic cycle if it is given by a projector in $\End_{\Mot_L}\bigl(\motH^{2i}(X)(n)\bigr)$ that lies in the image of this map.

The property of being cut out by an algebraic cycle is invariant under extension of scalars: if $K \subset L$ is a field extension then $\motV \subset \motH^{2i}(X)\bigl(n\bigr)$ is cut out by an algebraic cycle if and only if the same is true for $\motV_L \subset \motH^{2i}(X_L)\bigl(n\bigr)$. If $K \subset L$ is algebraic, this follows from an easy Galois argument; the general case then follows from the fact that for an extension $K \subset L$ of algebraically closed fields, the map $\CH(X)_\mQ/{\sim_{\hom}} \to \CH(X_L)_\mQ/{\sim_{\hom}}$ is bijective.

\subsection{}
\label{AlgTrans}
Let $\motV$ be a motive over a field $K \subset \mC$. As $\Mot_K$ is a semisimple category, $\motV$ canonically decomposes as a direct sum of isotypic components. The motivic Galois group $G_\mot(\motV)$ sits in a short exact sequence $1 \tto G_\mot(\motV_\mC) \tto G_\mot(\motV) \tto \Gal(L/K) \tto 1$ for some finite Galois extension $L/K$. (Cf.\ \cite{YAPour}, the end of Section~4, and note that $G_\mot(\motV_\Kbar) \cong G_\mot(\motV_\mC)$, as follows from the Scolie in ibid., 2.5.) Hence there is a unique isotypic component~$\motV^\alg$ of~$\motV$ (possibly zero) such that $(\motV^\alg)_\mC \cong \unitmot^{\oplus n}$ for some~$n$. We define $\motV^\trans$ to be the direct sum of all other isotypic components of~$\motV$; this gives us a canonical decomposition $\motV = \motV^\alg \oplus \motV^\trans$.

If $\motV \subset \motH^2(X)\bigl(1\bigr)$ for some smooth projective variety~$X$, it follows from the Lefschetz theorem on divisor classes that there are no non-zero Hodge classes in the Betti realization of~$\motV^\trans$; hence on Hodge realizations, $V_\Betti = V_\Betti^\alg \oplus V_\Betti^\trans$ is just the decomposition of~$V_\Betti$ into its algebraic and transcendental parts. Further, in this situation $\motV \subset \motH^2(X)\bigl(1\bigr)$ is cut out by an algebraic cycle if and only if the same is true for~$\motV^\trans$.


\section{An \texorpdfstring{$\ell$-adic}{l-adic} analogue of a result of Zarhin}
\label{Zarhlad}

\subsection{}
\label{ZarhResult}
We start by reviewing some results of Zarhin in~\cite{ZarhK3}. Let $(H,\phi)$ be a polarized Hodge structure of K3-type. We assume $H$ has trivial algebraic part, by which we mean that $H \cap H_\mC^{0,0} = (0)$. By \cite{ZarhK3}, Theorem~1.5.1, the endomorphism algebra $E = \End_{\QHS}(H)$ is a field which is either totally real or a CM-field. 

If $E$ is totally real, let $\tilde\phi \colon H \times H \to E$  be the $E$-bilinear lift of~$\phi$. (See \ref{NotConvIntro}(c).) In this case, \cite{ZarhK3}, Theorem~2.2.1 gives that the Mumford-Tate group of~$H$ is the group $\SO_{E/\mQ}(H,\tilde\phi)$, with notation as in~\ref{NotConvIntro}(d).

If $E$ is a CM-field, let $E_0 \subset E$ be the totally real subfield and $e \mapsto \bar{e}$ the complex conjugation on~$E$. Let $\tilde\phi \colon H \times H \to E$ be the $E$-valued hermitian lift of~$\phi$. (Again see \ref{NotConvIntro}(c).) In this case, Zarhin's result, \cite{ZarhK3}, Theorem~2.3.1, is that the Mumford-Tate group of~$H$ is the unitary group $\UU_{E/\mQ}(H,\tilde\phi)$.

\subsection{}
\label{ladicSetup}
The goal of this section is to prove an $\ell$-adic analogue of Zarhin's results. This is based on Pink's results in~\cite{Pinkladic}. For later purposes it will be convenient to first generalize these results to the case of a finitely generated ground field. 

Let $K\subset \mC$ be as in~\ref{MotivesYA} and consider a projective non-singular variety~$Y$ over~$K$. Let $\motH$ be a submotive of $\motH^{2i}(Y)\bigl(i\bigr)$ for some $i \geq 0$ that is cut out by an algebraic cycle (see~\ref{StrictSubmot}). We denote by~$G_\ell^\red(\motH)$ the reductive quotient of~$G_\ell^0(\motH)$ and let $V_\ell$ be the semi-simplification of~$H_\ell$ as a representation of~$G_\ell^0(\motH)$. We have $G_\ell^\red(\motH) \subset \GL(V_\ell)$. A cocharacter $\mu \colon \mG_{\mult,\Qlbar} \to G^\red_\ell(\motH)_{\Qlbar}$ gives rise to a grading 
\[
V_{\ell,\Qlbar} = \bigoplus_{j \in \mZ}\, V_{\ell,\Qlbar}^j
\]
where, by convention, the summand $V_{\ell,\Qlbar}^j$ is the subspace of~$V_{\ell,\Qlbar}$ on which $\mG_{\mult,\Qlbar}$ acts through the character $z \mapsto z^{-j}$.

The following notion was introduced in \cite{Pinkladic}, Def.~3.17.

\begin{definition}
We call a cocharacter~$\mu$ of~$G^\red_\ell(\motH)$ over~$\Qlbar$ a weak Hodge cocharacter if $\dim(V_{\ell,\Qlbar}^j) = \dim_\mC(H_\Betti^{j,-j})$ for all $j \in \mZ$.
\end{definition}

The following result is a slight generalization of a theorem of Pink in~\cite{Pinkladic}. The result is stated in loc.\ cit.\ only over number fields, and only for motives of the form $\motH^j(Y)$; however, the arguments also apply to submotives $\motH \subset \motH^{2i}(Y)\bigl(i\bigr)$ that are cut out by an algebraic cycle. (This condition is needed to ensure that $\motH$ gives rise to a strictly compatible system of $\ell$-adic Galois representations; cf.\ the comments following \cite{Pinkladic}, Definition~3.1.) As we shall now explain, the result also extends to finitely generated ground fields.

\begin{theorem}
\label{GenbyWHC}
With notation and assumptions as in\/~{\rm \ref{ladicSetup}}, $G^\red_\ell(\motH)_{\Qlbar}$ is generated by the images of the weak Hodge cocharacters.
\end{theorem}

\begin{proof} 
This follows from Pink's results by a specialization argument. If $k$ is the algebraic closure of~$\mQ$ in~$K$, there exist a smooth proper morphism of $k$-varieties $f \colon X \to S$, and an algebraic cycle $Z \subset X$, flat over~$S$, such that:

\begin{enumerate}[label=---]
\item $k(S) \cong K$ and the generic fiber $X_\eta$ of~$f$ is isomorphic to~$Y$;
\item for $s \in S$, the motivated cycle $\pi_{2i} \circ [Z_s] \circ \pi_{2i}$ on $X_s \times X_s$ cuts out a submotive $\motM_s \subset \motH^{2i}(X_s)\bigl(i\bigr)$ such that $\motM_\eta = \motH$. 
\end{enumerate}

Base change via $k \hookrightarrow \mC$ gives a smooth projective family $f_\mC \colon X_\mC \to S_\mC$, and the Hodge realizations of the submotives~$\motM_s$ for $s \in S(\mC)$ are the fibres of a sub-VHS of $R^{2i}f_{\mC,*}\mQ(i)$. In particular, the Hodge numbers of the motives~$\motM_s$ equal the Hodge numbers of the Betti realization of~$\motH$. The $\ell$-adic realizations of the~$\motM_s$ are the fibres of a smooth $\Ql$-subsheaf of $R^{2i}f_* \Ql(i)$. If $t$ is a point of~$S$, write $G_{t,\ell} = G_\ell(\motM_t)$, and let $V_{t,\ell}$ be the semi-simplification of~$M_{t,\ell}$ as a representation of~$G_{t,\ell}^0$.

Choosing geometric points $\bar{s}$ above~$s$ and $\bar\eta$ above~$\eta$, we get a specialization isomorphism $M_{\eta,\ell} \isomarrow M_{s,\ell}$. Taking this as an identification, $G_{s,\ell} \subseteq G_{\eta,\ell}$. By a result of Serre (see \cite{SerreRibet}, Section~1 and \cite{SerreMW}, Section~10.6), there are infinitely many closed points $s \in S$ such that $G_{s,\ell} = G_{\eta,\ell}$. For such points~$s$ we then have an isomorphism $V_{s,\ell} \isomarrow V_{\eta,\ell}$ and $G^\red_{s,\ell} = G^\red_{\eta,\ell}$. The theorem now follows by applying \cite{Pinkladic}, Theorem~3.18 (which is valid for submotives cut out by an algebraic cycle) to~$\motM_s$. 
\end{proof}

\subsection{}
Let $V$ be a finite dimensional vector space over a field~$k$ of characteristic~$0$, equipped with a non-degenerate symmetric bilinear form $\phi \colon V \times V \to k$.

If $\mu \colon \mG_{\mult,\kbar} \to \OO(V_\kbar,\phi)$ is a cocharacter and $V_\kbar = \oplus V^j_\kbar$ is the corresponding grading of~$V_\kbar$ then $V_\kbar^m$ and $V_\kbar^n$ are orthogonal whenever $m+n \neq 0$, and $\phi$ gives rise to non-degenerate pairings $V_\kbar^{-n} \times V_\kbar^n \to \kbar$. We shall say that~$\mu$ is {\it of K3 type\/} if $\dim(V_\kbar^{-1}) = \dim(V_\kbar^1) = 1$ and $V_\kbar^n = (0)$ whenever $|n|>1$.

\begin{theorem}
\label{ZarhK3ell}
\textup{(\romannumeral1)} Let $G \subset \SO(V,\phi)$ be a connected reductive subgroup such that $G_\kbar$ is generated by the images of cocharacters of K3 type. Then $V$, as a representation of~$G$, has a decomposition
\[
V = V_0 \oplus V_1 \oplus \cdots \oplus V_t \oplus V_{t+1} \oplus \cdots \oplus V_{t+u}
\]
such that
\begin{enumerate}
\item $V_0 = V^G$,
\item $V_i$ and~$V_j$ are non-isomorphic if $i \neq j$,
\item the form $\phi_i \colon V_i \times V_i \to k$ obtained by restriction of~$\phi$ is non-degenerate,
\item for $i \in \{1,\ldots,t\}$ the endomorphism algebra $E_i = \End_G(V_i)$ is a field, and $\phi_i(ev,w) = \phi(v,ew)$ for all $e \in E_i$ and $v, w \in V_i$,
\item for $i \in \{t+1,\ldots,t+u\}$ the endomorphism algebra $E_i = \End_G(V_i)$ is an \'etale quadratic extension of a field~$E_{i,0}$, and if $e \mapsto \bar{e}$ is the unique non-trivial automorphism of $E_i$ over~$E_{i,0}$, we have $\phi_i(ev,w) = \phi(v,\bar{e}w)$ for all $e \in E_i$ and $v, w \in V_i$.
\end{enumerate}

\noindent
Up to permutation of the summands $V_1\ldots,V_t$ and $V_{t+1},\ldots,V_{t+u}$, the decomposition is unique.

\textup{(\romannumeral2)} For $i>0$, let $G_i$ be the image of~$G$ in $\OO(V_i,\phi_i)$. Then each $G_{i,\kbar}$ is again generated by cocharacters of K3 type and $G = \prod_{i=1}^{t+u}\, G_i$ as subgroups of $\prod_{i=1}^{t+u}\, \OO(V_i,\phi_i)$.

\textup{(\romannumeral3)} For $i \in \{1,\ldots,t\}$ we have $G_i = \SO_{E_i/k}(V_i,\tilde\phi_i)$, where $\tilde\phi_i \colon V_i \times V_i \to E_i$ is the $E_i$-bilinear lift of~$\phi_i$. For $i \in \{t+1,\ldots,t+u\}$ we have $G_i = \UU_{E_i/k}(V_i,\tilde\phi_i)$, where $\tilde\phi_i \colon V_i \times V_i \to E_i$ is the $E_i$-valued hermitian lift of~$\phi_i$.
\end{theorem}

\begin{proof} (\romannumeral1) Let $W \subset V$ be an isotypic component on which $G$ acts non-trivially, so that $\End_G(W)$ is a simple $k$-algebra. If $U$ is an isotypic component of~$W_\kbar$ as a representation of~$G_\kbar$, there is a cocharacter $\mu \colon \mG_{\mult,\kbar} \to G_\kbar$ of K3 type that gives a non-trivial grading of~$U$. But $\dim(V_\kbar^{-1}) = \dim(V_\kbar^1) = 1$ for the grading of~$V_\kbar$ given by~$\mu$; so $U$ is necessarily an irreducible representation of~$G_\kbar$. Hence $\End_G(W)$ is a field, and we conclude that the non-trivial isotypic components of~$V$ are all irreducible.

If $\phi|_W$ is degenerate, it is the zero form. In this case there is a unique irreducible summand $W^\prime$ such that $\phi$ restricts to a non-degenerate form on $W \oplus W^\prime$. With $E = \End_G(W)$ we have $\End_G(W\oplus W^\prime) \cong E \times E$, such that the involution induced by~$\phi$ is given by $(e_1,e_2) \mapsto (e_2,e_1)$, which is of the second kind. If $\phi|_W$ is non-degenerate, let $e \mapsto \bar{e}$ be the involution of $E = \End_G(W)$ induced by~$\phi$. 

Let $V_1,\ldots,V_t$ be the irreducible summands of~$V$ on which $\phi$ is non-zero and such that the involution on $\End_G(V_i)$ is trivial (i.e., of the first kind). Let $V_{t+1},\ldots,V_{t+u}$ be the remaining irreducible summands on which $\phi$ is non-zero (with involution of the second kind on $\End_G(V_i)$), together with all $W \oplus W^\prime$ as above for which $\phi|_W = \phi|_{W^\prime} = 0$. Together with $V_0 = V^G$ this gives the stated decomposition.

(\romannumeral2) Choose a decomposition as in~(\romannumeral1). We view $G$ as an algebraic subgroup of $\prod_{i=1}^{t+u}\, \OO(V_i,\phi_i)$; in particular this gives us homomorphisms $r_i \colon G \to \OO(V_i,\phi_i)$. 

Let $\cS$ be the set of cocharacters of~$G_\kbar$ that are of K3 type. If $\mu \in \cS$, there is a unique index $i \in \{1,\ldots,t+u\}$ such that the induced action of $\mG_{\mult,\kbar}$ on $V_{i,\kbar}$ is non-trivial. This gives a decomposition $\cS = \coprod_{i=1}^{t+u}\, \cS_i$. Clearly the subsets~$\cS_i$ are stable under the action of $G(\kbar)$ on~$\cS$ by conjugation. Also they are stable under the natural action of $\Gal(\kbar/k)$. Hence we have normal subgroups~$H_i$ of~$G$ (over~$k$) such that $H_{i,\kbar}$ is generated by the cocharacters in~$\cS_i$.

If $\mu \in \cS_i$ then the induced action of $\mG_{\mult,\kbar}$ on $V_{j,\kbar}$ is trivial for all $j \neq i$. It follows that $H_i \subset \OO(V_i,\phi_i)$ (viewed as algebraic subgroup of~$\OO(V,\phi)$), and because $G$ is generated by all images of cocharacters of K3 type, $G = \prod_{i=1}^{t+u}\, H_i$. A fortiori, $H_i$ is the image of~$G$ in $\OO(V_i,\phi_i)$, i.e., $H_i=G_i$. 

(\romannumeral3) As in Zarhin's paper~\cite{ZarhK3}, the argument is based on Kostant's results in~\cite{Kostant}. Kostant states the Corollary to his main theorem (loc.\ cit., p.~107) over~$\mC$; this implies the same result over an arbitrary algebraically closed field of characteristic~$0$, as all objects involved are defined over a subfield that admits an embedding into~$\mC$.

For $i \in \{1,\ldots,t\}$ we have $G_{i,\kbar} \subset \prod_{\sigma \colon E_i \to \kbar}\, \SO(V_{i,\sigma},\tilde\phi_{i,\sigma})$, where $V_{i,\sigma} = V_i \otimes_{E_i,\sigma} \kbar$ and $\tilde\phi_{i,\sigma}$ denotes the bilinear extension of~$\tilde\phi_i$ to a form on~$V_{i,\sigma}$. By the same arguments as in the proof of~(\romannumeral2) we have connected subgroups $H_\sigma \subset \SO(V_{i,\sigma},\tilde\phi_{i,\sigma})$ such that $G_{i,\kbar} = \prod_\sigma\, H_\sigma$. Each~$H_\sigma$ is generated by the images of cocharacters of K3 type, and its representation on~$V_{i,\sigma}$ is irreducible. By Kostant's result, $H_\sigma = \SO(V_{i,\sigma},\tilde\phi_{i,\sigma})$; hence $G_i = \SO_{E_i/k}(V_i,\tilde\phi_i)$.

For $i \in \{t+1,\ldots,t+u\}$, let $\Sigma_0$ be the set of embeddings $E_{i,0} \to \kbar$. For $\tau \in \Sigma_0$ we have a unique $G_\kbar$-stable decomposition of $V_{i,\tau} = V_i \otimes_{(E_{i,0}),\tau} \kbar$ as $V_{i,\tau} = W \oplus W^\prime$ and the bilinear extension of the hermitian form~$\tilde\phi_i$ to~$V_{i,\tau}$ induces a duality $W^\prime \cong W^\vee$. Then $\UU(V_{i,\tau},\tilde\phi_{i,\tau}) \cong \GL(W)$, with $\GL(W)$ acting on~$W$ and~$W^\prime$ through the tautological representation and its contragredient, respectively. As before we have connected subgroups $H_\tau \subset \UU(V_{i,\tau},\tilde\phi_{i,\tau})$ such that $G_{i,\kbar} = \prod_{\tau\in \Sigma_0}\, H_\tau$ and each~$H_\tau$ is generated by  cocharacters of K3 type. If $\mu$ is a cocharacter of K3 type of~$G_{i,\tau}$ that has a non-trivial projection to~$H_\tau$, this induces a non-trivial grading of~$W$ with $W^j = (0)$ if $|j|>1$ and either $\dim(W^{-1}) = 1$ and $\dim(W^1)=0$ or $\dim(W^{-1}) = 0$ and $\dim(W^1)=1$. By Kostant's result it follows that $H_\tau = \GL(W)$; hence $G_i = \UU_{E_i/k}(V_i,\tilde\phi_i)$.
\end{proof}

\begin{corollary}
\label{HowtogetG=G'}
Let $G \subset G^\prime \subset \SO(V,\phi)$ be connected reductive subgroups such that $G_\kbar$ and~$G^\prime_\kbar$ are both generated by images of cocharacters of K3 type. If $\End_{G^\prime}(V) \isomarrow \End_G(V)$ then $G=G^\prime$.
\end{corollary}

\begin{proof}
The proposition implies that $G$ is the identity component of the commutant of $\End_G(V)$ inside $\SO(V,\phi)$; likewise for~$G^\prime$. 
\end{proof}

\begin{remark} 
Let $k \subset L$ be an extension of fields of characteristic~$0$. If $G \subset \SO(V,\phi)$ is a connected reductive subgroup such that $G_\kbar$ is generated by images of cocharacters of K3 type then the same is true for $G_\Lbar \subset \SO(V_\Lbar,\phi)$: use that $G_\kbar$ is Zariski dense in~$G_\Lbar$.
\end{remark}

\begin{remark}
\label{YAGap}
Theorem~\ref{ZarhK3ell} and its Corollary~\ref{HowtogetG=G'} fill what appears to be a gap in Andr\'e's paper~\cite{YAK3}. Specifically, the last sentence of loc.\ cit., Section~7.4, is correct but only refers to the setting of Hodge-Tate modules, i.e., local Galois representations. The above results provide the needed analogue for {\it global\/} Galois representations. 
\end{remark}


\section{Norm functors (a.k.a.\ corestrictions)}
\label{Norms}

\subsection{}
\label{NormBasics}
We shall need some basic results about the norm, or ``corestriction'', of algebraic structures. A basic reference for this is Ferrand's paper~\cite{Ferrand}; see also~\cite{Riehm}. We only need these notions for an extension $k \to E$ where $k$ is a field of characteristic~$0$ and $E$ is a finite \'etale $k$-algebra, i.e., a finite product of finite field extensions of~$k$, and for our purposes in this paper it suffices to consider only neutral Tannakian categories. 

Let
\begin{equation}
\label{eq:NormFunctMod}
\FNorm_{E/k} \colon \Mod_E \to \Mod_k 
\end{equation}
be the norm functor defined in \cite{Ferrand}. If $M$ is an $E$-module, $\FNorm_{E/k}(M)$ is a $k$-form of $\otimes_\sigma\, M_\sigma$, where the tensor product is taken over the set of $k$-homomorphisms $\sigma\colon E \to \kbar$ and $M_\sigma = M \otimes_{E,\sigma} \kbar$. By definition of the norm functor we have a polynomial map $\nu_M \colon M \to \FNorm_{E/k}(M)$ such that $\nu_M(em) = \Norm_{E/k}(e) \cdot \nu_M(m)$ for all $e\in E$ and $m\in M$.

The norm functor is a $\otimes$-functor (non-additive, unless $E=k$). It has the property that $\FNorm_{E/k}\bigl(\Hom_E(M_1,M_2)\bigr) = \Hom_k\bigl(\FNorm_{E/k}(M_1),\FNorm_{E/k}(M_2)\bigr)$.

As shown in \cite{Ferrand}, if $A$ is an $E$-algebra, $\FNorm_{E/k}(A)$ has a natural structure of a $k$-algebra; this gives a functor 
\[
\FNorm_{E/k} \colon \Alg_E \to \Alg_k
\]
that on the underlying modules is the norm functor~\eqref{eq:NormFunctMod}. The polynomial map $\nu_A \colon A \to \FNorm_{E/k}(A)$ is multiplicative; see \cite{Ferrand}, Prop.~3.2.5. If $A$ is a central simple $E$-algebra, $\FNorm_{E/k}(A)$ is what is classically called the corestriction of~$A$ to~$k$, which is a central simple $k$-algebra.

Let $G$ be an affine $E$-group scheme with affine algebra $A = \Gamma(G,\cO_G)$. Then $\FNorm_{E/k}(A)$ has a natural structure of a commutative Hopf algebra over~$k$ and this is the affine algebra of the $k$-group scheme $\Res_{E/k} G$. (See \cite{Ferrand}, Prop.~6.2.2.)

Let $V$ be an $E$-module of finite type, and write $N(V) = \FNorm_{E/k}(V)$. We have a natural homomorphism $\eta\colon \Res_{E/k} \GL(V) \to \GL\bigl(N(V)\bigr)$. On $k$-valued points it is the homomorphism that sends an $E$-linear automorphism $f \in \GL(V)$ to $\FNorm_{E/k}(f)$. If $V$ is a faithful $E$-module, $\Ker(\eta) = T_E^1 \subset \Res_{E/k} \GL(V)$, with notation as in~\ref{NotConvIntro}(e). If $V$ is not a faithful $E$-module, $N(V) = 0$ and $\eta$ is trivial.

\subsection{}
\label{ForgetFun}
Let $\cG$ be an affine group scheme over~$k$ and consider the neutral Tannakian category $\calC = \Rep_k(\cG)$. Let $\calC_{(E)}$ be the category of $E$-modules in~$\calC$. Then $\calC_{(E)}$ is $\Rep_E(\cG_E)$. 

Let $V$ be an $E$-module of finite type. Denote the underlying $k$-vector space by~$V_{(k)}$. We assume given a representation $\tilde\rho \colon \cG_E \to \GL(V)$, making $V$ into an object of~$\calC_{(E)}$. Note that to give~$\tilde\rho$ is equivalent to giving a homomorphism $\cG \to \Res_{E/k} \GL(V)$. This, in turn, is equivalent to giving $\rho \colon \cG \to \GL(V_{(k)})$ such that the $E$-action on~$V$ commutes with the $\cG$-action. The homomorphism~$\rho$ makes $V_{(k)}$ into an object of~$\calC$, and $V \mapsto V_{(k)}$ (or better: $\tilde\rho \mapsto \rho$) is the forgetful functor $\calC_{(E)} \to \calC$.

\subsection{}
\label{GEGEbar}
Keeping the notation of~\ref{ForgetFun}, the Tannakian subcategory $\langle V_{(k)}\rangle^\otimes \subset \calC$ generated by~$V_{(k)}$ is equivalent to $\Rep_k(G)$, where $G := \Image(\rho) \subset \Res_{E/k} \GL(V) \subset \GL(V_{(k)})$. The Tannakian subcategory $\langle V\rangle^\otimes \subset \calC_{(E)}$ generated by~$V$ is equivalent to $\Rep_k(\ol{G}_E)$, where $\ol{G}_E$ is the image of the natural homomorphism $G_E \to \GL(V)$. Note that the quotient map $G_E \twoheadrightarrow \ol{G}_E$ is not injective, in general; this corresponds to the fact that not every $E$-module in $\langle V_{(k)}\rangle^\otimes$ comes from an object of $\langle V\rangle^\otimes$. As an example, suppose $G = \Res_{E/k} H$ for some algebraic subgroup $H \subset \GL(V)$. In this case $\ol{G}_E = H$ and $G_E \twoheadrightarrow \ol{G}_E$ is the canonical quotient map $G_E \to H$.

\subsection{}
\label{N(V)Tannak}
With $V$ as in~\ref{ForgetFun} we can also form the $k$-vector space $N(V) = \FNorm_{E/k}(V)$ and consider the representation
\[
G \hookrightarrow \Res_{E/k} \GL(V) \mapright{\eta} \GL\bigl(N(V)\bigr)\, ,  
\]
where $\eta$ is the homomorphism defined in~\ref{NormBasics}. This makes $N(V)$ into an object of $\langle V_{(k)}\rangle^\otimes \subset \calC$. If $V$ is a faithful $E$-module, the Tannakian subcategory $\langle N(V)\rangle^\otimes \subset \calC$ is equivalent to $\Rep(G_1)$, where $G_1 \cong G/(G\cap T_E^1)$ is the image of~$G$ in $\GL\bigl(N(V)\bigr)$.

Associating $N(V) = \FNorm_{E/k}(V)$ to~$V$ defines a $\otimes$-functor $N\colon \calC_{(E)} \to \calC$, which is non-additive, unless $k=E$. We again call this~$N$ a norm functor.

For $V$ as above we have a natural map $\End_{\calC_{(E)}}(V) \to \End_{\calC}\bigl(N(V)\bigr)$. This is a ``normic law'' in the sense of~\cite{Ferrand}; by definition of the norm functor it therefore factors as
\begin{equation}
\label{eq:Homalpha}
\End_{\calC_{(E)}}(V) \mapright{\nu} \FNorm_{E/k}\bigl(\End_{\calC_{(E)}}(V)\bigr) \mapright{\alpha} \End_{\calC}\bigl(N(V)\bigr)  
\end{equation}
with $\alpha$ a homomorphism of $k$-algebras.

\begin{lemma}
\label{ResHInvarLem}
With notation as above, suppose $G = \Res_{E/k} H$ for an algebraic subgroup $H \subset \GL(V)$. 

\textup{(\romannumeral1)} The homomorphism~$\alpha$ in \textup{\eqref{eq:Homalpha}} is an isomorphism.

\textup{(\romannumeral2)} The natural map $N(V^H) \to N(V)^G$ is an isomorphism. 
\end{lemma}

\begin{proof}
It suffices to prove the assertions after extension of scalars to~$\kbar$. Let $\Sigma(E)$ be the set of $k$-homomorphisms $E \to \kbar$, and for $\sigma \in \Sigma(E)$ let a subscript~``$\sigma$'' denote the extension of scalars from~$E$ to~$\kbar$ via~$\sigma$. Then the assertions just say that
\[
\bigotimes_{\sigma \in \Sigma(E)} \, \End(V_\sigma)^{H_\sigma} \isomarrow \End\Bigl(\bigotimes_{\sigma \in \Sigma(E)}\, V_\sigma\Bigr)^{\prod_\sigma H_\sigma}
\quad\hbox{and}\quad
\bigotimes_{\sigma \in \Sigma(E)} \, (V_\sigma)^{H_\sigma} \isomarrow \Bigl(\bigotimes_{\sigma \in \Sigma(E)}\, V_\sigma\Bigr)^{\prod_\sigma H_\sigma}\, ,
\]
which are clear. 
\end{proof}

\subsection{}
\label{IndepOfFib}
{}From the description of the norm functor given in~\ref{N(V)Tannak}, it is not clear whether this functor depends on the choice of a fibre functor on~$\calC$. An alternative description, from which it is clear that it does not depend on such a choice, and that at the same time allows to extend the construction to more general tensor categories (including non-neutral Tannakian categories), is the following. It is important here that $\charact(k) = 0$, so that the symmetric algebra of a $k$-module~$W$ is the same as the divided power algebra~$\Gamma_k(W)$. We have not yet attempted to extend the construction to coefficient fields of positive characteristic.

Let $d = [E:K]$. As explained in \cite{Ferrand}, Section~2.4, $\Sym^d_k(E)$ has the structure of a $k$-algebra. The norm map $\Norm_{E/k} E \to k$, which is a homogeneous polynomial map of degree~$d$, induces an augmentation $\Sym^d_k(E) \to k$ (i.e., a homomorphism of $k$-algebras that is a section of the structural homomorphism $k \to \Sym^d_k(E)$). The norm functor $N\colon \calC_{(E)} \to \calC$ associates to an object~$M$ of~$\calC_{(E)}$ the object
\[
\Sym^d_k\bigl(M_{(k)}\bigr) \otimes_{\Sym^d_k(E)} k\, .
\]
(Cf.\ \cite{Ferrand}, the construction in the proof of Th\'eor\`eme~3.2.3.) For $f \colon M \to M^\prime$ a morphism in~$\calC_{(E)}$, the induced $N(f) \colon N(M) \to N(M^\prime)$ is given by $\Sym^d(f) \otimes \id$.

\begin{example}
\label{NormQHS}
For $E$ a number field, consider the norm functor $N \colon \QHS_{(E)} \to \QHS$. Let $\Sigma(E)$ be the set of complex embeddings of~$E$. If $V$ is a $\mQ$-Hodge structure of weight~$n$ on which $E$ acts, we have a decomposition $V_\mC = \oplus_{\sigma \in \Sigma(E)}\, V_\mC(\sigma)$ and for each~$\sigma$ a decomposition $V_\mC(\sigma) = \oplus_{p+q=n}\, V_\mC(\sigma)^{p,q}$, such that complex conjugation on~$V_\mC$ restricts to bijections $V_\mC(\sigma)^{p,q} \isomarrow V_\mC(\bar\sigma)^{q,p}$. If $\tilde{E}$ is the normal closure of~$E$ in~$\mC$ we have a decomposition $V_{\tilde{E}} = \oplus_{\sigma \in \Sigma(E)}\, V_{\tilde{E}}(\sigma)$, and the $\mQ$-vector space $N(V) = \FNorm_{E/\mQ}(V)$ is the space of $\Gal(\tilde{E}/\mQ)$-invariants in $\otimes_{\sigma \in \Sigma(E)}\, V_{\tilde{E}}(\sigma)$. It has $\mQ$-dimension $\dim_E(V)^{[E:\mQ]}$. The Hodge decomposition of $N(V)_\mC = \otimes_{\sigma \in \Sigma(E)}\, V_\mC(\sigma)$ is given by
\[
N(V)_\mC^{i,j} = \sum_{(p,q)}\; \bigotimes_{\sigma\in \Sigma(E)}\, V_\mC(\sigma)^{p(\sigma),q(\sigma)}
\]
where the sum is taken over all functions $(p,q)\colon \Sigma(E) \to \mZ^2$ with $\sum_\sigma\, p(\sigma) = i$ and $\sum_\sigma\, q(\sigma) = j$. In particular, $N(V)$ is a pure Hodge structure of weight $n \cdot [E:\mQ]$.
\end{example}

\begin{remark}
\label{NmNotAdjoint}
To avoid confusion, let us note that the norm functor $N \colon \calC_{(E)} \to \calC$ is not additive, in general; in particular, it is not adjoint to the extension of scalars functor $\calC \to \calC_{(E)}$. Also note that in the previous example $V$ and~$N(V)$ in general have different weight (and can even have weights of different parity); this shows that in general there are no non-zero morphisms between $V_{(k)}$ and~$N(V)$.
\end{remark}

\subsection{}
\label{MTCNormsIntro}
Let $E$ be a finite \'etale $\mQ$-algebra, i.e., a finite product of number fields. Let $\motV$ be a motive over~$\mC$ with a faithful action of~$E$ by endomorphisms, and write $N(\motV) = \FNorm_{E/\mQ}(\motV)$. The following result gives some relations between the Tate and Mumford-Tate conjectures for~$\motV$ and those for~$N(\motV)$. As always, $\ell$ is some fixed prime number.

\begin{proposition}
\label{MTCNorms}
\textup{(\romannumeral1)} The group $G_\ell^0(\motV)$ is reductive if and only if $G_\ell^0\bigl(N(\motV)\bigr)$ is reductive.

\textup{(\romannumeral2)} If the Mumford-Tate conjecture for~$\motV$ is true then also the Mumford-Tate conjecture for~$N(\motV)$ is true.

\textup{(\romannumeral3)} Suppose there exists a non-degenerate symmetric $E$-bilinear form $\tilde\phi \colon V \times V \to E$ such that $G_\mot(\motV) \subseteq \OO_{E/\mQ}(V,\tilde\phi)$. Then the Mumford-Tate conjecture for~$N(\motV)$ implies the Mumford-Tate conjecture for~$\motV$. 

\textup{(\romannumeral4)} If $\motW_1$ and~$\motW_2$ are motives, the Mumford-Tate conjecture for $\motW_1 \oplus \motW_2$ implies the Mumford-Tate conjecture for $\motW_1 \otimes \motW_2$. 
\end{proposition}

\begin{proof}
As a special case of what was discussed in~\ref{N(V)Tannak}, we have isomorphisms
\begin{equation}
\label{eq:G/Z=G}
G_\Betti(\motV)/Z \isomarrow G_\Betti\bigl(N(\motV)\bigr)
\qquad\hbox{and}\qquad
G_\ell^0(\motV)/Z_\ell \isomarrow G_\ell^0\bigl(N(\motV)\bigr)\, , 
\end{equation}
where $Z = G_\Betti(\motV) \cap T_E^1$ and $Z_\ell =  G_\ell^0(\motV) \cap (T_E^1 \otimes \Ql)$. Part~(\romannumeral2) follows, and (\romannumeral4) is the special case of~(\romannumeral2) where $E = \mQ \times \mQ$.

By \cite{BorelLAG}, Corollary~14.11, if we abbreviate $G_\ell^0(\motV)$ to~$G_\ell$ and if $\pi \colon G_\ell \to \ol{G}_\ell =  G_\ell/Z_\ell$ is the quotient map, $\pi\bigl(R_{\rm u}(G_\ell)\bigr) = R_{\rm u}(\ol{G}_\ell)$. On the other hand, as $T_E^1$ is a torus, $R_{\rm u}(G_\ell) \cap (T_E^1 \otimes \Ql) = \{1\}$. Hence $R_{\rm u}(G_\ell) \isomarrow R_{\rm u}(\ol{G}_\ell)$, which proves~(\romannumeral1).

(\romannumeral3) Because $G_\mot(\motV) \subseteq \OO_{E/\mQ}(V,\tilde\phi)$, the above group schemes $Z$ and~$Z_\ell$ are finite. As all groups in question are connected, it follows from~\eqref{eq:G/Z=G} that $\hbox{MTC}\bigl(N(\motV)\bigr)$ implies $\hbox{MTC}(\motV)$.
\end{proof}


\section{The Kuga-Satake construction in the presence of nontrivial endomorphisms}
\label{KugaSat}

\subsection{}
\label{NotConvNorm}
\textit{Notation and conventions.\/} Throughout this section, $k$ is a field of characteristic~$0$ and $k \subset E$ is a finite \'etale extension. Choose an algebraic closure $k \subset \kbar$. We denote by $\Sigma(E)$ the $\Gal(\kbar/k)$-set of $k$-algebra homomorphisms $E \to \kbar$.

It will be convenient to denote restrictions of scalars and norms (=corestrictions) by a subscript~``$E/k$''. (Cf.\ \ref{NotConvIntro}.) For instance, if $V$ is an $E$-module equipped with a symmetric bilinear form $\tilde\phi$ and $C^+(V,\tilde\phi)$ is the even Clifford algebra, we write $\SO_{E/k}(V,\tilde\phi)$ for $\Res_{E/k} \SO(V,\tilde\phi)$ and $C^+_{E/k}(V,\tilde\phi)$ for the $k$-algebra $\FNorm_{E/k} C^+(V,\tilde\phi)$.

A subscript~``$(k)$'' indicates that we forget the $E$-structure on an object. For instance, with $V$ as above, $V_{(k)}$ denotes the underlying $k$-vector space.

For $\sigma \in \Sigma(E)$, a subscript~``$\sigma$'' denotes an extension of scalars via~$\sigma$. Thus, for instance, $V_\sigma = V \otimes_{E,\sigma} \kbar$ and $\tilde\phi_\sigma$ denotes the extension of~$\tilde\phi$ to a $\kbar$-bilinear form on~$V_\sigma$.

\subsection{}
\label{CSpins}
Let $V$ be a nonzero free $E$-module of finite rank, equipped with a non-degenerate symmetric $E$-bilinear form $\tilde\phi \colon V \times V \to E$. We denote by $\phi = \trace_{E/k} \circ \tilde\phi \colon V_{(k)} \times V_{(k)} \to k$ its transfer. We then have an isomorphism of Clifford algebras $C_{E/k}(V,\tilde\phi) \isomarrow C(V_{(k)},\phi)$, inducing an injective homomorphism
\begin{equation}
\label{eq:C+E/kinC+}
C^+_{E/k}(V,\tilde\phi) \hookrightarrow C^+(V_{(k)},\phi)\, . 
\end{equation}

Left multiplication in the even Clifford algebra gives rise to a representation
\[
\rho_\spin \colon \CSpin_{E/k}(V,\tilde\phi) \to \GL\bigl(C^+_{E/k}(V,\tilde\phi)\bigr)\, .
\]
The kernel of this representation is the subtorus $T_E^1 \subset \CSpin_{E/k}(V,\tilde\phi)$, with notation as in~\ref{NotConvIntro}(e). We denote the quotient group $\CSpin_{E/k}(V,\tilde\phi)/T_E^1$ by $\CSpinbar_{E/k}(V,\tilde\phi)$.
 
We also have a representation
\[
\rho_\ad \colon \CSpin_{E/k}(V,\tilde\phi) \twoheadrightarrow \SO_{E/k}(V,\tilde\phi) \to \GL\bigl(C^+_{E/k}(V,\tilde\phi)\bigr)\, ,
\]
obtained from the action of $\SO(V,\tilde\phi)$ on $C^+(V,\tilde\phi)$ by transport of structure.

Let
\[
D = C^+_{E/k}(V,\tilde\phi)\, ,
\]
which is a semisimple $k$-algebra. We shall use the notation~$D$ when it appears in its role as algebra, and write $C^+_{E/k}(V,\tilde\phi)$ for the underlying $k$-vector space, which enters the discussion in a different role. The right multiplication of~$D$ on $C^+_{E/k}(V,\tilde\phi)$ commutes with the action of $\CSpin_{E/k}(V,\tilde\phi)$ via~$\rho_\spin$. We have an isomorphism of $\CSpin_{E/k}(V,\tilde\phi)$-representations
\begin{equation}
\label{eq:rhoad=EndD}
C^+_{E/k}(V,\tilde\phi)_\ad \cong \ul\End_D\bigl(C^+_{E/k}(V,\tilde\phi)_\spin\bigr)\, , 
\end{equation}
where the subscripts~``$\ad$'' and~``$\spin$'' indicate through which representation $\CSpin_{E/k}(V,\tilde\phi)$ acts.

\subsection{}
\label{CSpinbar}
Through the homomorphism \eqref{eq:C+E/kinC+}, we have a homomorphism
\[
\CSpin_{E/k}(V,\tilde\phi) \twoheadrightarrow \CSpinbar_{E/k}(V,\tilde\phi) \hookrightarrow \CSpin(V_{(k)},\phi)\, .
\]
The algebraic subgroup $\CSpinbar_{E/k}(V,\tilde\phi) \subset \CSpin(V_{(k)},\phi)$ is the inverse image of $\SO_{E/k}(V,\tilde\phi) \subset \SO(V_{(k)},\phi)$ under the natural homomorphism $\CSpin(V_{(k)},\phi)\to \SO(V_{(k)},\phi)$; so we have a diagram
\[
\begin{matrix}
\CSpin_{E/k}(V,\tilde\phi) && \\[2pt]
\mapdownl{} && \\
\CSpinbar_{E/k}(V,\tilde\phi) & \subset & \CSpin(V_{(k)},\phi) \\[2pt]
\mapdownl{} && \mapdownr{} \\
\SO_{E/k}(V,\tilde\phi) & \subset & \SO(V_{(k)},\phi) 
\end{matrix}
\]
in which the square is cartesian. 

Consider the representation 
\[
R_\spin \colon \CSpin_{E/k}(V,\tilde\phi)\to \CSpin(V_{(k)},\phi) \mapright{r_\spin} \GL\bigl(C^+(V_{(k)},\phi)\bigr)\, ,
\] 
where $r_\spin$ is the spin representation of $\CSpin(V_{(k)},\phi)$. By construction, $\rho_\spin$ is a direct summand of $R_\spin$.

\begin{lemma}
\label{AnalogDel3.5}
Let $\alpha$ be an algebra automorphism of $D = C^+_{E/k}(V,\tilde\phi)$ that commutes with the action of $\Spin_{E/k}(V,\tilde\phi)$ via~$\rho_\ad$. Then $\alpha = \id_D$.
\end{lemma}

\begin{proof}
We have
\[
C^+_{E/k}(V,\tilde\phi) \otimes_k \Qbar = \bigotimes_{\sigma \in \Sigma(E)}\, C^+(V_\sigma,\tilde\phi_\sigma)
\]
and
\[
\Spin_{E/k}(V,\tilde\phi) \otimes_k \Qbar = \prod_{\sigma \in \Sigma(E)}\, \Spin(V_\sigma,\tilde\phi_\sigma)\, .
\]
In this description the representation~$\rho_\ad$ is the exterior tensor product of the representations~$\rho_\ad$ of the factors.

Consider a $\kbar$-linear automorphism of $\otimes\, C^+(V_\sigma,\tilde\phi_\sigma)$ that commute with the adjoint action of the group $\prod\, \Spin(V_\sigma,\tilde\phi_\sigma)$. Any such automorphism is of the form $\alpha = \otimes \alpha_\sigma$, where $\alpha_\sigma$ is a linear automorphism of $C^+(V_\sigma,\tilde\phi_\sigma)$ that commutes with the adjoint action of $\Spin(V_\sigma,\tilde\phi_\sigma)$. If additionally $\alpha$ is an algebra automorphism, $\alpha(1) = 1$ implies that we may rescale the~$\alpha_\sigma$ such that $\alpha_\sigma(1) = 1$ for all~$\sigma$. Then all~$\alpha_\sigma$ are algebra automorphisms and the assertion now follows from \cite{DelK3}, Proposition~3.5.
\end{proof}

\subsection{}
\label{KS/pt}
In the rest of this section we specialize the above to the case that $k=\mQ$ and $E$ is a totally real number field. 

Let $(V,\phi)$ be a polarized Hodge structure of K3-type. (See \ref{NotConvIntro}.) Suppose $E$ acts on~$V$ such that $\phi(ev,w) = \phi(v,ew)$, and let $\tilde\phi \colon V \times V \to E$ be the $E$-bilinear lift of~$\phi$. The Hodge structure is then described by a homomorphism $h \colon \mS \to \SO_{E/\mQ}(V,\tilde\phi)_\mR \subset \SO(V_{(\mQ)},\phi)$. As in \cite{DelK3}, 4.2, there is a unique lifting of~$h$ to a homomorphism $\tilde{h} \colon \mS \to \CSpin(V_{(\mQ)},\phi)_\mR$ of weight~$1$, and by what was discussed in~\ref{CSpinbar}, $\tilde{h}$~factors through $\CSpinbar_{E/\mQ}(V,\tilde\phi)_\mR$.

View $\rho_\spin$ and $R_\spin$ as representations of $\CSpinbar_{E/\mQ}(V,\tilde\phi)$. The representation $\rho_\spin \circ \tilde{h}$ defines a $\mQ$-Hodge structure on $C^+_{E/\mQ}(V,\tilde\phi)$. Similarly, $R_\spin \circ \tilde{h}$ defines a $\mQ$-Hodge structure on $C^+(V_{(\mQ)},\phi)$, which by \cite{DelK3}, Section~4, is polarizable and of type $(1,0) + (0,1)$. As $\rho_\spin$ is a direct summand of~$R_\spin$, it follows that the Hodge structure $C^+_{E/\mQ}(V,\tilde\phi)$ is polarizable and of type $(1,0) + (0,1)$, too. Moreover, as the algebra $D = C^+_{E/\mQ}(V,\tilde\phi)$ acts on $C^+_{E/\mQ}(V,\tilde\phi)$ (from the right) by Hodge-endomorphisms, this defines a complex abelian variety~$A$ (up to isogeny; see our conventions in~\ref{NotConvIntro}) with multiplication by~$D$. We refer to~$A$ with its $D$-action as the Kuga-Satake variety associated with~$(V,\tilde\phi)$. It follows from the construction together with~\eqref{eq:rhoad=EndD} that we have an isomorphism of $\mQ$-Hodge structures
\[
u_{(V,\tilde\phi)}\colon C^+_{E/\mQ}(V,\tilde\phi) \cong \ul\End_D\bigl(H^1(A,\mQ)\bigr)\, . 
\]

\subsection{}
\label{KS/S}
The construction of the Kuga-Satake abelian variety also works in families. Let $S$ be a non-singular complex algebraic variety and $(\mV,\phi)$ a polarized VHS of K3 type over~$S$ with multiplication by a totally real field~$E$. (This means we are given a homomorphism $E \to \End_{\QVHS_S}(\mV)$.) Fix a base point $t \in S$ and write $V = \mV(t)$ for the fiber at~$t$. We assume that the VHS comes from a $\mZ$-VHS over~$S$; this just means that there exists a lattice in~$V$ that is stable under the action of $\pi_1(S,t)$. In this situation there exist a finite \'etale cover $\pi\colon S^\prime \to S$, an abelian scheme $g \colon A \to S^\prime$ (up to isogeny) with multiplication by the algebra $D = C^+_{E/k}(V,\tilde\phi)$, and an isomorphism 
\[
u_{(\mV,\tilde\phi)}\colon \pi^* C^+_{E/\mQ}(\mV,\tilde\phi) \isomarrow \ul\End_D(R^1 g_*\mQ_A)\, . 
\]
of algebras in the category $\QVHS_{S^\prime}$. It is of course understood here that the fiber of $A \to S^\prime$ over a point $s^\prime \in S^\prime$ is the Kuga-Satake abelian variety associated with the fiber of~$\mV$ over $\pi(s^\prime) \in S$. The construction of $A/S^\prime$ is an easy variation on what is explained in \cite{DelK3}, \S~5 and in \cite{YAK3},~\S~5.


\section{Motives with real multiplication}
\label{RealMult}

\subsection{}
\label{WpsiPrep}
Let $X$ be a non-singular complex projective variety. Let $\motM \subset \motH^2(X)$ be a submotive that is cut out by an algebraic cycle (see~\ref{StrictSubmot}), with $\dim_\mC(M_{\Betti,\mC}^{2,0})=1$. Suppose we have a decomposition of the motive $\motM(1)$ as $\motM(1) = \motV \oplus \motT$, where $\motT$ is spanned by divisor classes on~$X$. (So $\motT \cong \unitmot^{\oplus r}$ for some $r\geq 0$.) Suppose further that the motive~$\motV$ has multiplication by a totally real field~$E$.

Let $V$ be the $E$-vector space underlying the Hodge realization of~$\motV$ and denote by $V_{(\mQ)}$ the underlying $\mQ$-vector space. Through the choice of an ample bundle on~$X$, we get a non-degenerate symmetric pairing  $\tilde\phi \colon \motV \otimes_E \motV \to \unitmot_E$. (The ample bundle gives a bilinear form $\phi\colon V_{(\mQ)} \otimes_\mQ V_{(\mQ)} \to \mQ$ and $\tilde\phi$ corresponds to the $E$-bilinear lift of~$\phi$ as in~\ref{NotConvIntro}(c).)

The goal of this section is to prove the following result, which in the next section will be used to prove some cases of our main theorem.

\begin{proposition}
\label{MTCDimOdd}
Notation and assumptions as in\/~{\rm \ref{WpsiPrep}}. Assume that $\dim_E(V) = 2m+1$ is odd. Further assume there exists a complex abelian variety~$A$ with multiplication by the even Clifford algebra $D = C^+_{E/\mQ}(V,\tilde\phi)$ and an isomorphism 
\begin{equation}
\label{eq:motuassumpt}
\motu\colon C^+_{E/\mQ}(\motV,\tilde\phi) \isomarrow \ul\End_D\bigl(\motH^1(A)\bigr) 
\end{equation}
of algebras in the category~$\Mot_\mC$. Then the Tate Conjecture and the Mumford-Tate conjecture for the submotive $\motM \subset \motH^2(X)$ are true.
\end{proposition}

To prove the Tate and Mumford-Tate conjectures for~$\motM$, it suffices to prove these conjectures for~$\motV$.

\subsection{}
\label{descending}
With notation and assumptions as above, there exist
\begin{enumerate}[label=---]
\item a variety $Y$ over a finitely generated field $K \subset \mC$ and an isomorphism $\alpha \colon Y_\mC \isomarrow X$;

\item a submotive $\motW \subset \motH^2(Y)\bigl(1\bigr)$ that is cut out by an algebraic cycle, an action of~$E$ on~$\motW$, and a form $\tilde\psi \colon \motW \otimes_E \motW \to \unitmot_E$ in $\Mot_{K,(E)}$ such that the isomorphism $\alpha^*\colon \motH^2(X)\bigl(1\bigr) \isomarrow \motH^2(Y_\mC)\bigl(1\bigr)$ restricts to an $E$-equivariant isomorphism $\motV \isomarrow \motW_\mC$ via which $\tilde\phi$ corresponds with~$\tilde\psi_\mC$;

\item an abelian variety $B/K$ with multiplication by~$D$, and a $D$-equivariant isogeny $B_\mC \isomarrow A$;

\item an isomorphism $C^+_{E/\mQ}(\motW,\tilde\psi) \cong \ul\End_D\bigl(\motH^1(B)\bigr)$ of algebras in the category~$\Mot_K$ that after extension of scalars to~$\mC$ gives back the isomorphism~\eqref{eq:motuassumpt}.
\end{enumerate}

\noindent
In what follows we fix these data. The assertion we want to prove is that the Tate conjecture and Mumford-Tate conjecture for~$\motW$ are true. As explained in Section~\ref{Conjectures}, we may replace the ground field~$K$ by a finite extension. In particular, we may, and will, further assume that
\begin{enumerate}[label=---]
\item the group $G_\ell(\motW)$ is connected;
\item the $3$-torsion of~$B$ is $K$-rational;
\item the algebraic part~$\motW^\alg$ of~$\motW$ (see~\ref{AlgTrans}) is isomorphic to~$\unitmot^{\oplus h}$ for some~$h$.
\end{enumerate}

\subsection{}
\label{PosetQ}
We retain the notation and the conventions introduced in~\ref{NotConvNorm}, taking $k=\mQ$. Let $\gS = \gS\bigl(\Sigma(E)\bigr)$ be the symmetric group on the set~$\Sigma(E)$. Let $\cQ$ be the set of $\gS$-orbits in $\mN^{\Sigma(E)}$. We give $\cQ$ a poset structure by the rule that for orbits $q$ and~$q^\prime$ we have $q \leq q^\prime$ if there exist representatives $\epsilon$, $\epsilon^\prime \in \mN^{\Sigma(E)}$ such that $\epsilon(\sigma) \leq \epsilon^\prime(\sigma)$ for all $\sigma \in \Sigma(E)$. (Less canonically, the elements of~$\cQ$ correspond to Young diagrams with $[E:\mQ]$ rows, and $q \leq q^\prime$ if and only if the Young diagram of~$q^\prime$ contains the diagram of~$q$.)

Consider the motive $C^+_{E/\mQ}(\motW,\tilde\psi)$, which is an object of the category~$\Mot_K$. We have an ascending filtration $\motF_\gdot$ indexed by~$\cQ$ on  this motive that can be described as follows. We think of $G_\mot(\motW)$ as an algebraic subgroup of $\OO_{E/\mQ}(W,\tilde\psi)$. The motive $C^+_{E/\mQ}(\motW,\tilde\psi)$ then corresponds to the $G_\mot(\motW)$-representation $C^+_{E/\mQ}(W,\tilde\psi)$.

Recall that $\dim_E(W) = 2m+1$. We have
\[
C^+_{E/\mQ}(W,\tilde\psi) \otimes_\mQ \Qbar = \bigotimes_{\sigma \in \Sigma(E)} \, C^+(W_\sigma,\tilde\psi_\sigma)\, .
\]
For $i \in \mN$ and $\sigma \in \Sigma(E)$, let $F_{\sigma,i} \subset C^+(W_\sigma,\tilde\psi_\sigma)$ be the image of $\oplus_{j \leq 2i}\, W_\sigma^{\otimes j}$ in $C^+(W_\sigma,\tilde\psi_\sigma)$. Note that $F_{\sigma,m} = C^+(W_\sigma,\tilde\psi_\sigma)$ and that $F_{\sigma,m}/F_{\sigma,m-1} \cong \wedge^{2m} W_\sigma \cong W_\sigma^\vee \otimes \det(W_\sigma)$.

For $\epsilon \in \mN^{\Sigma(E)}$, let $\tilde{F}_\epsilon = \otimes_\sigma\, F_{\sigma,\epsilon(\sigma)}$. For $q \in \cQ$ we then define $\tilde{F}_q \subset \otimes_\sigma \, C^+(W_\sigma,\tilde\psi_\sigma)$ to be the linear span of all $\tilde{F}_\epsilon$ for $\epsilon \in q$. The $\Gal(\Qbar/\mQ)$-action on $\otimes_\sigma \, C^+(W_\sigma,\tilde\psi_\sigma)$ preserves the subspaces~$\tilde{F}_q$. The space $F_q = (\tilde{F}_q)^{\Gal(\Qbar/\mQ)} \subset C^+_{E/\mQ}(W,\tilde\psi)$ of Galois-invariants in~$\tilde{F}_q$ is stable under the action of $\OO_{E/\mQ}(W,\tilde\psi)$, and we define $\motF_q \subset C^+_{E/\mQ}(\motW,\tilde\psi)$ to be the corresponding submotive.

Let $q_2 \in \cQ$ be the orbit of $(m-1,m,\ldots,m)$ and let $q_1 = \bigl\{(m,\ldots,m)\bigr\} \in \cQ$. Then $q_2 \leq q_1$ and
\begin{align*}
\motF_{q_1}/\motF_{q_2} &\cong \FNorm_{E/\mQ}\bigl(\wedge_E^{2m} \motW\bigr)\\
&\cong \FNorm_{E/\mQ}\bigl(\motW^\vee \otimes_E \det\nolimits_E(\motW)\bigr)\\
&\cong \FNorm_{E/\mQ}\bigl(\motW \otimes_E \det\nolimits_E(\motW)\bigr)\\
&\cong \FNorm_{E/\mQ}(\motW) \otimes_\mQ \det(\motW_{(\mQ)})\, ,
\end{align*}
where we use the isomorphism $\motW^\vee \cong \motW$ given by the form~$\tilde\psi$. Because the category $\Mot_K$ is semisimple, it follows that $\FNorm_{E/\mQ}(\motW) \otimes_\mQ \det(\motW_{(\mQ)})$ is (non-canonically) isomorphic to a submotive of $\ul\End_D\bigl(\motH^1(B)\bigr)$.

\subsection{}
\label{GellinGB}
Let $E_\ell = E \otimes \Ql$. By our assumptions, $G_\ell(\motW) \subset \OO_{E_\ell/\Ql}(W_\ell,\tilde\psi_\ell)$ is connected; hence $G_\ell(\motW)$ is in fact a subgroup of $\SO_{E_\ell/\Ql}(W_\ell,\tilde\psi_\ell)$. By what was discussed in~\ref{N(V)Tannak}, the image of the $\ell$-adic representation associated with the motive $\FNorm_{E/\mQ}(\motW)$ is isomorphic to $G_\ell(\motW)/Z_\ell$, where $Z_\ell = G_\ell(\motW) \cap T_{E_\ell}^1$, which is a finite central subgroup scheme of~$G_\ell(\motW)$. As the $\ell$-adic realization of $\det(\motW_{(\mQ)})$ is trivial, the conclusion of~\ref{PosetQ} together with Faltings's results in \cite{Faltings} imply that $G_\ell(\motW)/Z_\ell$, and hence also~$G_\ell(\motW)$, is reductive. This implies that the $\ell$-adic representation $\rho_{\motW,\ell}$ is completely reducible.

For the Mumford-Tate groups we have a similar situation. Let $Z = G_\Betti(\motW) \cap T_E^1$, which is a finite central subgroup scheme of~$G_\Betti(\motW)$; then
\[
G_\Betti(\motW)/Z \isomarrow G_\Betti\bigl(\FNorm_{E/\mQ}(\motW)\bigr) = G_\Betti\bigl(\FNorm_{E/\mQ}(\motW) \otimes \det(\motW_{(\mQ)})\bigr)\, .
\]
As $\FNorm_{E/\mQ}(\motW) \otimes \det(\motW_{(\mQ)})$ lies in the category of abelian motives, it follows from Deligne's results in~\cite{DelAbsHodge} that
\[
G_\ell(\motW)/Z \subseteq \bigl(G_\Betti(\motW)/Z_\ell\bigr) \otimes \Ql
\]
as algebraic subgroups of $\GL\bigl(\FNorm_{E/\mQ}(W) \otimes \det(W_{(\mQ)}\bigr) \otimes \Ql$. As $G_\ell(\motW)$ and $G_\Betti(\motW)$ are connected groups, this implies that 
\begin{equation}
\label{eq:GlinGB}
G_\ell(\motW) \subseteq G_\Betti(\motW) \otimes \Ql 
\end{equation}
as algebraic subgroups of $\GL(W_{(\mQ)}) \otimes \Ql$.

\subsection{}
\label{TateClassesOK}
By Zarhin's result \cite{ZarhK3}, Theorem~2.2.1, and its $\ell$-adic analogue Theorem~\ref{ZarhK3ell}, there exist algebraic subgroups $H_\Betti \subset \GL(W)$ (over~$E$) and $H_\ell \subset \GL(W_\ell)$ (over~$E_\ell$) such that $G_\Betti(\motW) = \Res_{E/\mQ}(H_\Betti)$ and $G_\ell(\motW) = \Res_{E_\ell/\Ql}(H_\ell)$.

Because $\FNorm_{E/\mQ}(\motW) \otimes_\mQ \det(\motW_{(\mQ)})$ is isomorphic to a submotive of $\ul\End_D\bigl(\motH^1(B)\bigr)$, Faltings's results imply that
\[
\bigl[\FNorm_{E/\mQ}(W_\Betti) \otimes_\mQ \det(W_{\Betti,(\mQ)})\bigr]^{G_\Betti(\motW)} \otimes \Ql \isomarrow  \bigl[\FNorm_{E_\ell/\Ql}(W_\ell) \otimes_\mQ \det(W_{\ell,(\Ql)})\bigr]^{G_\ell(\motW)}
\] 
under the comparison isomorphism between Betti and $\ell$-adic realizations. Because $\det(W_{\Betti,(\mQ)})$  and $\det(W_{\ell,(\Ql)})$ are trivial as representations of~$G_\Betti(\motW)$ and~$G_\ell(\motW)$, respectively, Lemma~\ref{ResHInvarLem}(\romannumeral2) gives that
\[
\FNorm_{E_\ell/\Ql}\bigl(W_\Betti^{H_\Betti} \otimes_E E_\ell\bigr) = \FNorm_{E/\mQ}\bigl(W_\Betti^{H_\Betti}\bigr) \otimes_\mQ \Ql \isomarrow  \FNorm_{E_\ell/\Ql}\bigl(W_\ell^{H_\ell}\bigr)\, .
\]
On the other hand, $W_\Betti^{H_\Betti} = W_{\Betti,(\mQ)}^{G_\Betti(\motW)}$ and $W_\ell^{H_\ell} = W_{\ell,(\Ql)}^{G_\ell(\motW)}$, so it follows from~\eqref{eq:GlinGB} that $W_\Betti^{H_\Betti} \otimes_E E_\ell \hookrightarrow W_\ell^{H_\ell}$ under the comparison isomorphism $W_\Betti \otimes_E E_\ell \isomarrow W_\ell$. Combining these remarks we find that
\[
\bigl(W_{\Betti,(\mQ)}\bigr)^{G_\Betti(\motW)} \otimes_\mQ \Ql \isomarrow  \bigl(W_{\ell,(\Ql)}\bigr)^{G_\ell(\motW)}\, .
\]
Because $\motW$ is a submotive of $\motH^2(Y)\bigl(1\bigr)$, all classes in $(W_{\Betti,(\mQ)})^{G_\Betti(\motW)}$ are algebraic by the Lefschetz theorem on divisor classes. Hence all classes in $(W_{\ell,(\Ql)})^{G_\ell(\motW)}$ are algebraic as well, and together with the conclusions obtained in~\ref{GellinGB} this proves the Tate conjecture for divisor classes on~$\motW$ (and hence on $\motM \subset \motH^2(X)$).

\subsection{}
\label{NotatVaria}
\textit{Notation.\/} We fix an algebraic closure $\Ql \subset \Qlbar$ containing~$\Qbar$. This gives an identification of $\Sigma(E)$ with the set of $\Ql$-algebra homomorphisms $E_\ell \to \Qlbar$. We shall use the notation $\Sigma(E)$ in both meanings. A subscript~``$\sigma$'' will denote an extension of scalars from $E$ to~$\Qbar$ or from~$E_\ell$ to~$\Qlbar$, as will be clear from the context. Let
\[
\ol{E}_\ell = E_\ell \otimes_{\Ql} \Qlbar = \prod_{\sigma \in \Sigma(E)}\, \Qlbar\, .
\]
Similarly, we let
\[
\ol{W}_\ell = W_\ell \otimes_{\Ql} \Qlbar = \bigoplus_{\sigma \in \Sigma(E)}\, W_{\ell,\sigma}
\qquad\hbox{and}\qquad
\ol{H}^1_\ell(B) = H^1_\ell(B) \otimes_{\Ql} \Qlbar
\]
and, with $D_\ell = D \otimes_\mQ \Ql$,
\[
\ol{D}_\ell = D_\ell \otimes_{\Ql} \Qlbar \cong \bigotimes_{\sigma \in \Sigma(E)}\,  C^+(W_{\ell,\sigma},\tilde\psi_\sigma)\, .
\]

For the proof of the Mumford-Tate conjecture, we shall need to compare the endomorphisms of~$W_\Betti$ as a Hodge structure and of $W_\ell$ as a Galois representation. A key step is the following result.

\begin{proposition}
\label{IsometryProp}
Let $\beta$ be a $\Gal(\Kbar/K)$-equivariant $\ol{E}_\ell$-linear isometry of~$\ol{W}_\ell$. Then $\beta$ is in the image of the map
\[
\End_{\QHS_{(E)}}(W_\Betti) \otimes_E \ol{E}_\ell \longhookrightarrow \End_{\ol{E}_\ell}(\ol{W}_\ell)\, .
\]
\end{proposition}

\begin{proof}
Write $\Gamma_K = \Gal(\Kbar/K)$. For $\sigma \in \Sigma(E)$, let $\beta_\sigma$ be the restriction of~$\beta$ to the summand~$W_{\ell,\sigma}$. Write $C^+(\beta_\sigma)$ for the induced automorphism of $C^+(W_{\ell,\sigma},\tilde\psi_\sigma)$, which is $\Gamma_K$-equivariant. The tensor product $\otimes_{\sigma \in \Sigma(E)} C^+(\beta_\sigma)$ is a Galois-equivariant automorphism of the algebra
\[
\bigotimes_{\sigma \in \Sigma(E)} \, C^+(W_{\ell,\sigma},\tilde\psi_\sigma) = C^+_{E_\ell/\Ql}(W_\ell,\tilde\psi) \otimes_\Ql \Qlbar \cong \End_{\ol{D}_\ell}\bigl(\ol{H}^1_\ell(B)\bigr)\, .
\]
As $\dim_E(W)$ is odd this algebra is a matrix algebra over~$\Qlbar$, so by Skolem-Noether there exists an automorphism $\delta \in \Aut_{\ol{D}_\ell}\bigl(\ol{H}^1_\ell(B)\bigr)$ such that $\otimes C^+(\beta_\sigma) = \Inn(\delta)$ as automorphisms of $\End_{\ol{D}_\ell}\bigl(\ol{H}^1_\ell(B)\bigr)$. Because $\otimes C^+(\beta_\sigma)$ is Galois-equivariant and the centre of $\End_{\ol{D}_\ell}\bigl(\ol{H}^1_\ell(B)\bigr)$ is $\Qlbar \cdot \id$, we find a character $\chi \colon \Gamma_K \to \Qlbar^*$ such that 
\begin{equation}
\label{eq:gammadelta}
{}^\gamma\delta = \chi(\gamma) \cdot \delta 
\end{equation}
for all $\gamma \in \Gamma_K$. (To avoid confusion, note that ${}^\gamma\delta$ is the conjugate of~$\delta$ by the automorphism of $\ol{H}^1_\ell(B)$ given by the action of~$\gamma$.) 

Taking determinants over~$\Qlbar$ we find that $\chi(\gamma) \in \mu_{2g}(\Qlbar)$, where $g = \dim(B)$. On the other hand, the relation~\eqref{eq:gammadelta} means that $\delta$ defines an isomorphism of Galois representations $\ol{H}^1_\ell(B) \otimes \chi \isomarrow \ol{H}^1_\ell(B)$. It follows that $\chi$ lies in the Tannakian subcategory of $\Rep(\Gamma_K,\Qlbar)$ generated by $\ol{H}^1_\ell(B)$. By the connectedness of $G_\ell\bigl(\motH^1(B)\bigr)$, which is a consequence of the assumption (see the beginning of~\ref{PosetQ}) that the $3$-torsion of~$B$ is $K$-rational, it follows that $\chi = 1$; so $\delta$ is a $\Gamma_K$-equivariant $\ol{D}_\ell$-linear automorphism of $\ol{H}^1_\ell(B)$.

By the results of Faltings (see \cite{Faltings}, Thm.~1), $\delta$ and~$\delta^{-1}$ are in the image of the map 
\[
\End_{\QHS,D}\bigl(H^1_\Betti(B)\bigr) \otimes_\mQ \Qlbar \tto \End_{\ol{D}_\ell}\bigl(\ol{H}^1_\ell(B)\bigr)
\]
induced by the comparison isomorphism $H^1_\Betti(B) \otimes \Ql \isomarrow H^1_\ell(B)$. Next we consider the diagram
\[
\begin{matrix}
C^+_{E/\mQ}(W_\Betti,\tilde\psi) \otimes_\mQ \Qlbar & \maprightou{\sim}{u_\Betti \otimes \id} & \End_D\bigl(H^1_\Betti(B)\bigr) \otimes_\mQ \Qlbar \cr
\mapdownl{\wr} && \mapdownr{\wr}\cr
C^+_{E_\ell/\Ql}(W_\ell,\tilde\psi) \otimes_\Ql \Qlbar &\maprightou{\sim}{u_\ell\otimes \id} & \End_{D_\ell}\bigl(H^1_\ell(B)\bigr) \otimes_\Ql \Qlbar
\end{matrix}
\]
where $u_\Betti$ and $u_\ell$ are the realizations of the isomorphism~$\motu$ in Proposition~\ref{MTCDimOdd}, and the vertical maps are given by the comparison isomorphisms. Because this diagram is commutative, we find that there exists $\theta_1,\ldots,\theta_n \in \End_{\QHS}\bigl(C^+_{E/\mQ}(W_\Betti,\tilde\psi)\bigr)$ and $c_1,\ldots,c_n \in \Qlbar$ such that $c_1\theta_1 + \cdots + c_n\theta_n = \otimes C^+(\beta_\sigma)$ as endomorphisms (in fact, automorphisms) of $C^+_{E_\ell/\Ql}(W_\ell,\tilde\psi) \otimes_\Ql \Qlbar$.

Next we note that $\otimes C^+(\beta_\sigma)$ preserves the filtration~$F_\gdot$ introduced in~\ref{PosetQ}. If we denote by $\End^F \subset \End$ the subspaces of endomorphisms that preserve~$F_\gdot$, the squares in the diagram
\[
\begin{matrix}
\End^F_{\QHS}\bigl(C^+_{E/\mQ}(W_\Betti,\tilde\psi)\bigr) \otimes \Qlbar  & \subset & \End_{\QHS}\bigl(C^+_{E/\mQ}(W_\Betti,\tilde\psi)\bigr) \otimes \Qlbar  \cr
\bigcap && \bigcap \cr
\End^F_{\mQ}\bigl(C^+_{E/\mQ}(W_\Betti,\tilde\psi)\bigr) \otimes \Qlbar & \subset & \End_{\mQ}\bigl(C^+_{E/\mQ}(W_\Betti,\tilde\psi)\bigr) \otimes \Qlbar  \cr
\mapdownl{\wr} && \mapdownr{\wr} \cr
\End^F_{\Qlbar}\bigl(C^+_{\ol{E}_\ell/\Qlbar}(\ol{W}_\ell,\tilde\psi)\bigr) & \subset & \End_{\Qlbar}\bigl(C^+_{\ol{E}_\ell/\Qlbar}(\ol{W}_\ell,\tilde\psi)\bigr)
\end{matrix}
\]
are cartesian. Hence we may additionally assume that the endomorphisms~$\theta_i$ preserve~$F_\gdot$. Taking the induced actions on $F_{q_1}/F_{q_2}$ as in~\ref{PosetQ}, we find that 
\[
\otimes \beta_\sigma \in \End\bigl(\otimes_\sigma\; W_{\ell,\sigma}\bigr) = \End\bigl(\FNorm_{E_\ell/\Ql}(W_\ell)\bigr) \otimes_\Ql \Qlbar = \End_\mQ\bigl(\FNorm_{E/\mQ}(W)\bigr) \otimes_\mQ \Qlbar
\]
lies in the subalgebra $\End_{\QHS}\bigl(\FNorm_{E/\mQ}(W_\Betti)\bigr) \otimes_\mQ \Qlbar$. 

Let $H \subset \GL(W)$ be the algebraic subgroup such that $G_\Betti(\motW) = \Res_{E/\mQ}(H)$; see \ref{NotatVaria}. Then the information we get is that $\otimes \beta_\sigma$ commutes with the action of $\prod_\sigma\, H_\sigma$ on $\otimes_\sigma\,  W_{\ell,\sigma}$. This implies that each~$\beta_\sigma$ individually commutes with the action of~$H_\sigma$. Finally,
\[
\begin{matrix}
\End_{\QHS_{(E)}}(W_\Betti) \otimes_E \ol{E}_\ell && \End_{\ol{E}_\ell}(\ol{W}_\ell) \cr
\Big\Vert && \Big\Vert \cr
\prod_\sigma\, \End_{\Qlbar[H_\sigma]}(W_{\ell,\sigma}) & \longhookrightarrow & \prod_\sigma\, \End_{\Qlbar}(W _{\ell,\sigma}) 
\end{matrix}
\]
and the proposition is proven.
\end{proof}

\begin{lemma}
\label{D*Lemma}
Let $(D,*)$ be a semisimple algebra with involution over an algebraically closed field~$F$ of characteristic~$0$. Then $D$ is generated, as an $F$-algebra, by the elements $d \in D$ with $dd^*=1$.
\end{lemma}

\begin{proof}
There is an immediate reduction to the case that $(D,*)$ is simple as an algebra with involution. Then $(D,*)$ is isomorphic to one of the following three (for some~$n$):
\begin{enumerate}
\item $M_n(F)$ with involution $A \mapsto A^t$;
\item $M_n(F)$ with involution $A \mapsto A^*$ (adjoint matrix);
\item $M_n(F) \times M_n(F)$ with involution $(A,B) \mapsto (B^t,A^t)$.
\end{enumerate}

\noindent
In the first two cases the result follows by the double centralizer theorem and the remark that the standard $n$-dimensional representations of $\OO_n(F)$ and~$\SL_n(F)$ are (absolutely) irreducible. Similarly, in the third case we are looking at the representation $\St \oplus \St^\vee$ of~$\GL_n(F)$, with $\St$ the standard representation. In this case the subalgebra~$D^\prime$ generated by the $d \in D$ with $dd^*=1$ is semisimple with centralizer $F\times F$, and by the double centralizer theorem we get $D^\prime = M_n(F) \times M_n(F)$.
\end{proof}

\begin{corollary}
\label{EndE=EndE}
Writing $\Gamma_K = \Gal(\Kbar/K)$, we have
\[
\End_{E_\ell[\Gamma_K]}(W_\ell) = \End_{\QHS_{(E)}}(W_\Betti) \otimes_E E_\ell\, ,
\]
viewing both sides as subalgebras of $\End_{E_\ell}(W_\ell)$.
\end{corollary}

\begin{proof}
Because $G_\ell(\motW) \subseteq G_\Betti(\motW) \otimes \Ql$ (see~\ref{GellinGB}) we have the inclusion~``$\supseteq$''. It suffices to prove the other inclusion after extension of scalars to~$\ol{E}_\ell$. The pairing~$\tilde\psi_\ell$ (the $\ell$-adic realization of the form~$\tilde\psi$ on the motive) induces an involution~$*$ on the semisimple algebra $\End_{\ol{E}_\ell[\Gamma_K]}(\ol{W}_\ell)$, and the elements with $dd^* = 1$ are precisely the $\Gamma_K$-equivariant isometries of~$\ol{W}_\ell$. Now use Proposition~\ref{IsometryProp} and Lemma~\ref{D*Lemma}.
\end{proof}

\subsection{}
\label{PfMTCDimOdd}
We now derive the Mumford-Tate conjecture for~$\motW$, thereby completing the proof of Proposition~\ref{MTCDimOdd}. 

By the assumptions made in~\ref{descending}, we have a decomposition $\motW = \motW^\trans \oplus \unitmot_E^{\oplus \nu}$ in $\Mot_{K,(E)}$ such that the underlying motive~$\motW^\trans_{(\mQ)}$ (forgetting the $E$-module structure) is the transcendental part of~$\motW_{(\mQ)}$. By the Tate conjecture (see~\ref{TateClassesOK}), there are no non-zero Tate classes in~$W_{\ell,(\Ql)}^\trans$. By Pink's theorem~\ref{GenbyWHC}, $G_\ell(\motW) \otimes \Qlbar$ is generated by the images of weak Hodge cocharacters. 

By the results of Zarhin that we have reviewed in~\ref{ZarhResult}, $\End_\QHS(W_{\Betti,(\mQ)}^\trans)$ is a field~$F$ that contains~$E$; hence $\End_{\QHS_{(E)}}(W_\Betti^\trans) \cong \End_\QHS(W_{\Betti,(\mQ)}^\trans)$. Similarly, because there are no non-zero Tate classes in~$W_{\ell,(\Ql)}^\trans$, Theorem~\ref{ZarhK3ell} gives us that $\End_{\Ql[\Gamma_K]}(W_{\ell,(\Ql)}^\trans)$ is a commutative semisimple $\Ql$-algebra containing~$E_\ell$; hence $\End_{E_\ell[\Gamma_K]}(W_\ell^\trans) = \End_{\Ql[\Gamma_K]}(W_{\ell,(\Ql)}^\trans)$. By Corollary~\ref{EndE=EndE} it follows that $\End_\QHS(W_{\Betti,(\mQ)}) \otimes \Ql = \End_{\Ql[\Gamma_K]}(W_{\ell,(\Ql)})$, and the Mumford-Tate conjecture $G_\Betti(\motW) \otimes \Ql = G_\ell(\motW)$ then follows from Corollary~\ref{HowtogetG=G'}.


\section{Monodromy and the Mumford-Tate conjecture: the totally real case}
\label{Mono=>MTC}

\subsection{}
\label{MainThmSetup}
Let $S$ be a connected non-singular complex algebraic variety and $f \colon \cX \to S$ a smooth projective morphism such that the fibres~$\cX_s$ are connected algebraic varieties.

Let $\xi \in S(\mC)$, and let $\motM_\xi \subset \motH^2(\cX_\xi)$ be a submotive that is cut out by a $\pi_1(S,\xi)$-invariant projector~$p_\xi$. Then $p_\xi$ is the value at~$\xi$ of an idempotent section $p \in H^0\bigl(S,\ul\End(R^2f_*\mQ_\cX)\bigr)$, and by \cite{YAPour}, Th\'eor\`eme~0.5, $p_s$ is a motivated cycle for every $s \in S(\mC)$. Denoting by $\motM_s \subset \motH^2(\cX_s)$ the submotive cut out by~$p_s$ we obtain a family of motives $\{\motM_s\}_{s \in S(\mC)}$ parametrized by~$S$. The Hodge realizations of these motives form a direct factor $\mM \subset R^2f_*\mQ_\cX$ in $\QVHS_S$. In what follows we assume that $\dim_\mC(\mM_s^{2,0})=1$ for all (equivalently: some) $s \in S(\mC)$.
\medskip

The main result of this paper is the following.

\begin{theorem}
\label{MainThm}
Let $X$ be a non-singular complex projective variety of dimension~$d$. Let $\motM \subset \motH^2(X)$ be a submotive whose Hodge realization~$M_\Betti$ satisfies $\dim_\mC(M_\Betti^{2,0})=1$. Suppose there exists a smooth projective family $f \colon \cX \to S$ as in\/~{\rm \ref{MainThmSetup}}, a point $\xi \in S(\mC)$ and an isomorphism $X \isomarrow \cX_\xi$, via which we identify~$\motM$ with a submotive $\motM_\xi \subset \motH^2(\cX_\xi)$, such that:
\begin{enumerate}[label=\textup{(\alph*)}]
\item $\motM_\xi$ is cut out by a $\pi_1(S,\xi)$-invariant projector $p_\xi \in \End_{\Mot_\mC}\bigl(\motH^2(\cX_\xi)\bigr)$ that is of the form $p_\xi = \pi_2 \circ \class(\gamma) \circ \pi_2$ for some $\gamma \in \CH^d(X\times X)\otimes\mQ$,
\item the associated variation of Hodge structure $\mM \subset R^2f_*\mQ_\cX$, as in\/~{\rm \ref{MainThmSetup}}, is not isotrivial. 
\end{enumerate}

\noindent
Then the Tate conjecture and the Mumford-Tate conjecture for~$\motM$ are true.
\end{theorem}

The Main Theorem stated in the Introduction is the special case where $\motM = \motH^2(X)$, so that $p_\xi = \id$, in which case condition~(a) is automatically satisfied. At first reading of the proof, it is advisable to keep this case in mind. 

We start with some preparations for the proof. Some remarks to help the reader to navigate through the proof will be made in~\ref{Navigation}.

\subsection{}
\label{SetupContd}
We retain the notation of the theorem, and we assume that condition~(a) is satisfied. The notation of~\ref{MainThmSetup} then applies. Let $S^\gen \subset S$ be the Hodge-generic locus for the $\mQ$-variation of Hodge structure~$\mM$. The choice of a relatively ample bundle on~$\cX$ gives us a polarization form~$\phi$ on this VHS. 

Let $M = M_\xi$ be the $\mQ$-vector space underlying the Hodge realization of~$\motM_\xi$, and let $G_\mono(\mM) \subset \GL(M)$ denote the algebraic monodromy group of the VHS~$\mM$. We assume that $G_\mono(\mM)$ is connected; this can always be achieved by passing to a connected finite \'etale cover of~$S$. We then have an orthogonal decomposition $\mM(1) = \mV \oplus \mQ_S^{\oplus \rho}$ in $\QVHS_S$ such that for $s \in S^\gen$ the fibre $\mV_s$ has no non-zero Hodge classes. We again write~$\phi$ for the restriction of~$\phi$ to~$\mV$.

Let $V = \mV_\xi$. Let $G_\Betti(\mV) \subset \SL(V)$ be the generic Mumford-Tate group of the VHS~$\mV$. The subspace $V \subset M$ is stable under the action of $G_\mono(\mM)$. The homomorphism $G_\mono(\mM) \to \GL(V)$ is injective and identifies $G_\mono(\mM)$ with the algebraic monodromy group $G_\mono(\mV)$ of the variation~$\mV$. By the result of Y.~Andr\'e in \cite{YAMTGps}, Section~5, $G_\mono(\mV)$ is a normal subgroup of $G^\der_\Betti(\mV)$.

The variation~$\mV$ is of type $(-1,1) + (0,0) + (1,-1)$ with Hodge numbers $1,n,1$ for some $n\geq 0$. Let $E = \End_{\QVHS_S}(\mV)$ be the endomorphism algebra. For $s \in S$ we have an injective homomorphism $E \hookrightarrow \End_\QHS\bigl(\mV_s\bigr)$ and for $s \in S^\gen$ this is an isomorphism. By \cite{ZarhK3}, Theorems 1.5.1 and~1.6, it follows that either $E$ is a totally real field and $G_\Betti(\mV) = \SO_{E/\mQ}(V,\tilde{\phi})$, or $E$ is a CM-field and $G_\Betti(\mV) = \UU_{E/\mQ}(V,\tilde{\phi})$. (Cf.\ Section~\ref{ZarhResult}.)

Let $\upsilon \colon \tilde{S} \to S$ be the universal cover of~$S$ for the complex topology, and fix a trivialization $\upsilon^* \mV \isomarrow V \times \tilde{S}$. The variation~$\mV$ then gives rise to a period map $\tilde{S} \to \cD$, where $\cD$ is a domain of type~\romannumeral4 if $E$ is totally real and is a complex ball if $E$ is a CM field. In what follows the following condition will play an important role: 
\[
\hbox{The VHS~$\mV$ is not isotrivial, i.e., the period map of~$\mV$ is not constant.} \leqno{\PeriodCond}
\]
This condition is equivalent to the condition that the VHS~$\mM$ is not isotrivial, which is condition~(b) in Theorem~\ref{MainThm}.

For $s\in S(\mC)$, let $h(s)$ be the $\mQ$-dimension of the space $\mV_s \cap \mV_{s,\mC}^{0,0}$ of Hodge classes in~$\mV_s$. By construction, $h(s) = 0$ for $s \in S^\gen$.

\begin{proposition}
\label{JumpPicard} 
\textup{(\romannumeral1)} The following three conditions are equivalent.
\begin{enumerate}[label=\textup{(\alph*)}]
\item Condition \PeriodCond.
\item The connected algebraic monodromy group $G_\mono(\mV)$ is not the trivial group.
\item The function $h\colon S(\mC) \to \mN$ is not constant.
\end{enumerate}

\textup{(\romannumeral2)} Assume \PeriodCond\ holds. If $E$ is a totally real field then $\rank_E(\mV) \geq 3$, if $E$ is a CM-field then $\rank_E(\mV) \geq 2$.

\textup{(\romannumeral3)} Assume \PeriodCond\ holds. Then the variation~$\mV$ has maximal monodromy, by which we mean that $G_\mono(\cX/S) = G^\der_\Betti(\cX/S)$, except possibly when $E$ is totally real and $\dim_E(V) = 4$.
\end{proposition}

\begin{proof}
That \PeriodCond\ implies (b) follows from the Theorem of the Fixed Part. Conversely, suppose the period map of the variation~$\mV$ is constant. The image of the monodromy representation $r\colon \pi_1(S,\xi) \to \GL(V_\mR)$ is a discrete group, as it is contained in $\GL(V_\mZ)$ with $V_\mZ = V \cap H^2(\cX_\xi,\mZ)(1)$. On the other hand, the assumption that the period map is constant implies that the Weil operator~$C$ on~$V_\mR$ is invariant under the action of $\pi_1(S,\xi)$; hence $\Image(r)$ is contained in $\SO(H_\mR,\Phi)$, where $\Phi \colon V_\mR \times V_\mR \to \mR$ is the form given by $\Phi(x,y) = \phi_\mR(x,Cy)$. By definition of a polarization, $\Phi$ is definite; hence $\Image(r)$ is finite.

That \PeriodCond\ implies (c)  readily follows from Theorem~3.5 in~\cite{Voisin}, taking into account the assumption that $\dim_\mC(\mM_s^{2,0}) = 1$ for all $s \in S$. Conversely it is obvious that (c) implies~\PeriodCond.

For (\romannumeral2) we just have to remark that the generic Mumford-Tate group cannot be abelian, as otherwise the period domain is a point. 

For (\romannumeral3) we use that $G_\mono(\mV)$ is a normal subgroup of~$G^\der_\Betti(\mV)$, which by~(\romannumeral1) is not trivial. Further, if $E$ is a CM-field then $\dim_E(V) \geq 2$ and $G^\der_\Betti(\mV) = \SU_{E/\mQ}(V,\tilde{\phi})$ is a simple algebraic group. If $E$ is totally real then $\dim_E(V) \geq 3$, and $G^\der_\Betti(\mV) = \SO_{E/\mQ}(V,\tilde{\phi})$ can be non-simple only if $\dim_E(V) = 4$.
\end{proof}

\subsection{}
\label{motVConstr}
Let $s \in S(\mC)$. Then $\motM_s$ is a direct factor of $\motH^2(\cX_s)$; so by the Lefschetz theorem on divisor classes, $\mV_s \subset \mM_s(1)$ is the Hodge realization of a submotive $\motV_s \subset \motM_s(1)$. We have $\motM_s(1) = \motV_s \oplus \unitmot^{\oplus \rho}$, where, as in~\ref{SetupContd}, $\rho$~is the dimension of the space of Hodge classes in~$\mM_s(1)$ for a Hodge-generic point~$s$. As in~\ref{AlgTrans}, $\motV_s = \motV_s^\alg \oplus \motV_s^\trans$, and with $h(s)$ as defined in~\ref{SetupContd} we have $\motV_s^\alg \cong \unitmot^{\oplus h(s)}$. By construction, $\motV_s = \motV_s^\trans$ for $s \in S^\gen$.

\begin{proposition}
\label{EMotivated} 
Assume \PeriodCond\ holds. Then for every $s \in S$ the endomorphisms in~$E$, viewed as endomorphisms of~$\mV_s$, are motivated cycles in~$\End(\mV_s)$, i.e., they are the Hodge realizations of endomorphisms of~$\motV_s$.
\end{proposition}

At one step in the proof we shall refer forward to a calculation in Section~\ref{QuatCaseStart}. The reader will have no trouble checking that there is no circularity in the argument. 
\medskip

\begin{proof}
For $s \in S$ and $e\in E$, write $e_s$ for the image of~$e$ in $\End(\mV_s)$. By \cite{YAPour}, Th\'eor\`eme~0.5, if $e_s$ is a motivated cycle for some $s \in S$, the same is true for every $s \in S$. 

The motivic Galois group $\cG_\mot$ of the category~$\Mot_\mC$ acts on $\End(\mV_s)$ by algebra automorphisms. By \cite{YAK3}, Lemma~6.1.1, the subalgebra $\End(\mV_s)^{\pi_1(S,s)} \subset \End(\mV_s)$ of monodromy-equivariant endomorphisms is stable under the action of~$\cG_\mot$.

First assume the monodromy of the variation~$\mV$ is maximal, which by Proposition~\ref{JumpPicard} is automatic unless $E$ is totally real and $\mV$ has rank~$4$ over~$E$. In this case, $G_\mono(\mV) = \SO_{E/\mQ}(V,\tilde{\phi})$ (totally real case) or $G_\mono(\mV) = \SU_{E/\mQ}(V,\tilde{\phi})$ (CM case). In both cases, it follows, taking into account Proposition~\ref{JumpPicard}(\romannumeral2), that $\End(\mV_s)^{\pi_1(S,s)}$ is precisely the image~$E_s$ of~$E$ in $\End(\mV_s)$; hence $\cG_\mot$ acts on~$E_s$ through its group of field automorphisms. This defines an algebra $\motE_s$ in the category~$\Mot_\mC$. If the monodromy is not maximal, $E_s$ is no longer the full subalgebra $\End(\mV_s)^{\pi_1(S,s)}$ but it is the centre of this algebra. See \ref{QuatCaseStart} for more details. So in this case, too, we get an action of $\cG_\mot$ on~$E_s$ by field automorphisms.

Next we show that the action of~$\cG_\mot$ on $E\cong E_s$ is independent of~$s$. By definition of~$E$ we have a constant sub-VHS $\mE \subset \ul\End(\mV)$, purely of type $(0,0)$, with fibres the~$E_s$. Consider the family of motives $\ul\Hom(\motE_\xi,\motE_s)$ over~$S$, which has as Hodge realization the constant VHS $\ul\Hom(E_\xi,\mE)$. At $s=\xi$ the identity $\id_E \colon E_\xi \to E_s$ is clearly a motivated cycle, and it extends to a flat section of $\ul\Hom(E_\xi,\mE)$. By \cite{YAPour}, Corollaire~5.1, it follows that $\id_E \colon E_\xi \to E_s$ is a motivated cycle for all~$s$, and this just means that the $\cG_\mot$-action on~$E_\xi$ is the same as the one on~$E_s$.

Our goal is to show that $\cG_\mot$-acts trivially on~$E$. By Proposition~\ref{JumpPicard}(\romannumeral1), there exists a point $s \in S$ such that $\mV_s$ has non-zero Hodge classes, or, equivalently, $\motV_s^\alg \neq 0$. We then have $\cG_\mot$-equivariant homomorphisms $i^\alg\colon E \hookrightarrow \End(\mV_s^\alg)$ and $i^\trans\colon E \hookrightarrow \End(\mV_s^\trans)$ such that the endomorphism $e_s$ of~$\mV_s$ is given by the diagonal element $\bigl(i^\alg(e),i^\trans(e)\bigr)$. As $\motV_s^\alg \cong \unitmot^{\oplus h(s)}$ and the unit motive~$\unitmot$ corresponds to the trivial representation of~$\cG_\mot$, it follows that $i^\alg(e)$, and hence also~$e_s$, is invariant under the action of~$\cG_\mot$.
\end{proof}

\subsection{}
\label{motVinMotE}
{}From now on, we assume we are in the situation of Theorem~\ref{MainThm}, and we keep all notation introduced earlier in this section. By Proposition~\ref{EMotivated}, $E$ acts on $\motV_s = \motV_s^\alg \oplus \motV_s^\trans$ and the motivic Galois group $G_\mot(\motV_s)$ is an algebraic subgroup of $\OO(\mV_s,\phi) \cap \GL_E(\mV_s) = \OO_{E/\mQ}(\mV_s,\tilde\phi)$. 

In what follows, we view $\motV_s$ as an object of the category~$\Mot_{\mC,(E)}$ of $E$-modules in~$\Mot_\mC$, and we denote by~$\motV_{s,(\mQ)} \in \Mot_\mC$ the object obtained by forgetting the $E$-module structure. (Cf.\ Section~\ref{NotConvNorm}.) The same applies to $\motV_s^\alg$ and~$\motV_s^\trans$ and to realizations; e.g., we view $\mV_s$ as an object of $\QHS_{(E)}$ and denote by $\mV_{s,(\mQ)} \in \QHS$ the underlying $\mQ$-Hodge structure, forgetting the $E$-action.

Note that $\motV_s^\alg \cong \unitmot_E^{\oplus \nu(s)}$ for some $\nu(s) \geq 0$, where $\unitmot_E = E \otimes \unitmot$ denotes the identity object in~$\Mot_{\mC,(E)}$. (The integer~$h(s)$ defined in~\ref{SetupContd} then equals $\nu(s) \cdot [E:\mQ]$.) Further, $G_\mot(\motV_s)$ maps isomorphically to its image in $\OO_{E/\mQ}(\mV^\trans_s,\tilde\phi)$.

\subsection{}
\label{Navigation}
In the proof of Theorem~\ref{MainThm} we distinguish some cases. In the remainder of this section we prove the theorem in the case where the field~$E$ is totally real, under the additional hypothesis that the VHS~$\mV$ has maximal monodromy (which by Proposition~\ref{JumpPicard} is automatic if $\dim_E(V) \neq 4$). In this case, the result will be deduced from Proposition~\ref{MTCDimOdd} that we have proved in the previous section. 

In the next section we treat the case where $E$ is a CM-field. In Section~\ref{Mono=>MTC3}, finally, we deal with the case where $E$ is totally real and the monodromy is not maximal.

\subsection{}
\label{ConditionsRem}
We now assume that $E$ is totally real and that $G_\mono(\mV) = G^\der_\Betti(\mV)$. For $s \in S(\mC)$ we have $G_\mot(\motV_s) \subset \OO_{E/\mQ}(\mV_s,\tilde\phi)$. Therefore, we have a well-defined motive $C^+_{E/\mQ}(\motV_s,\tilde\phi) \in \Mot_\mC$.
 
We retain the notation and conventions of~\ref{NotConvNorm}, with $k=\mQ$. An extension of scalars from~$\mQ$ to~$\Ql$ is denoted by a subscript~``$\ell$''. For instance, $E_\ell = E \otimes_\mQ \Ql$. As in~\ref{NotConvIntro}(c) we can lift the polarization form~$\phi$ to an $E$-bilinear symmetric bilinear form $\tilde\phi \colon \mV \times \mV \to E_S$ such that $\phi = \trace_{E/\mQ} \circ \tilde\phi$. The even Clifford algebra $C^+(\mV,\tilde\phi)$ is then an algebra in the category $\QVHS_{S,(E)}$. Its norm $C^+_{E/\mQ}(\mV,\tilde\phi)$ is an algebra in $\QVHS_S$. 

We shall take the point~$\xi$ as in the formulation of Theorem~\ref{MainThm} as base point. Write $V = \mV_\xi$. We again write $\tilde\phi$ (rather than~$\tilde\phi_\xi$) for the polarization form on~$V$. Let
\[
D = C^+_{E/\mQ}(V,\tilde\phi)\, ,
\]
which is a semisimple $\mQ$-algebra. As in Section~\ref{KugaSat}, we use the notation~$D$ when it appears in its role as algebra, and write $C^+_{E/\mQ}(V,\tilde\phi)$ for the underlying $\mQ$-vector space.

\subsection{}
\label{motu_sDef}
We now apply what was explained in~\ref{KS/S}. The conclusion of this is that, possibly after again passing to a finite \'etale cover of~$S$, there exists an abelian scheme $\pi \colon A\to S$ with multiplication by~$D$ and an isomorphism
\[
u\colon C^+_{E/\mQ}(\mV,\tilde\phi) \isomarrow \ul\End_D(R^1\pi_*\mQ_A)
\] 
of algebras in the category $\QVHS_S$.

The fibre of the isomorphism~$u$ at a point~$s$ is a $\pi_1(S,s)$-equivariant isomorphism 
\[
u_s \colon C^+_{E/\mQ}(\mV_s,\tilde\phi) \isomarrow \ul\End_D\bigl(H^1(A_s)\bigr)
\] 
of algebras in the category $\QHS$. By Lemma~\ref{AnalogDel3.5} and the assumption that the monodromy of the variation~$\mV$ is maximal, $u_s$ is the only such algebra isomorphism that is $\pi_1(S,s)$-equivariant. By the same argument as in \cite{YAK3}, Proposition~6.2.1, it follows that $u_s$ is the Hodge realization of an isomorphism 
\begin{equation}
\label{eq:motu}
\motu_s \colon C^+_{E/\mQ}(\motV_s,\tilde\phi) \isomarrow \ul\End_D\bigl(\motH^1(A_s)\bigr)
\end{equation}
of algebras in the category $\Mot_\mC$.

\subsection{}
\label{EvenCase}
If the variation~$\mV$ introduced in~\ref{SetupContd} has odd rank (and hence $\dim_E(V)$ is odd), Theorem~\ref{MainThm} follows from Proposition~\ref{MTCDimOdd}, taking $s=\xi$ in~\ref{motu_sDef}.

Next assume that the variation~$\mV$ has even rank. In this case we use a trick due to Yves Andr\'e. We start by looking at the variation of Hodge structure $\mV^\sharp = \mV \oplus \mE$, where $\mE$ is the constant variation whose fibres are~$E$ with trivial Hodge structure. Taking the orthogonal sum of the form~$\tilde\phi$ on~$\mV$ and the obvious form on~$\mE$, we obtain an $E$-bilinear polarization form $\tilde\phi^\sharp\colon \mV^\sharp \times \mV^\sharp \to \mE$. Possibly after passing to a finite \'etale cover of~$S$, we have a Kuga-Satake abelian scheme $\pi \colon A^\sharp \to S$ with multiplication by an even Clifford algebra~$D^\sharp$ and an isomorphism
\[
u^\sharp \colon C^+_{E/\mQ}(\mV^\sharp,\tilde\phi^\sharp) \isomarrow \ul\End_{D^\sharp}(R^1\pi_*\mQ_{A^\sharp})\, .
\]
Note that the fibres of $\mV^\sharp$ and $C^+_{E/\mQ}(\mV^\sharp,\tilde\phi^\sharp)$ at a point~$s$ are the Hodge realizations of motives $\motV_s^\sharp = \motV_s \oplus \unitmot_E$ and $C^+_{E/\mQ}(\motV_s^\sharp,\tilde\phi^\sharp)$. 

We claim that there exists a point $s \in S(\mC)$ such that the motive $C^+_{E/\mQ}(\motV_s^\sharp,\tilde\phi^\sharp)$ lies in the subcategory $\Mot(\Ab)_\mC \subset \Mot_\mC$ of abelian motives. Before proving this, let us explain how the theorem follows. 

Suppose that $C^+_{E/\mQ}(\motV_s^\sharp,\tilde\phi^\sharp)$ is an object of $\Mot(\Ab)_\mC$. By \cite{YAPour}, Th\'eor\`eme~0.6.2 (Andr\'e's refinement of Deligne's ``Hodge is absolute Hodge'' for abelian varieties), the fibre of the isomorphism~$u_s^\sharp$ at the point~$s$ is then motivated, i.e., it is the Hodge realization of an isomorphism of motives $\motu_s^\sharp \colon C^+_{E/\mQ}(\motV_s^\sharp,\tilde\phi^\sharp) \isomarrow \ul\End_D\bigl(\motH^1(A^\sharp_s)\bigr)$. By \cite{YAPour}, Th\'eor\`eme~0.5, the same conclusion then holds for all fibres in the family. In particular, we can apply this to the fibre at the point $\xi \in S(\mC)$ as in the statement of the theorem, and we obtain that $u_\xi^\sharp$ is the Hodge realization of an isomorphism $\motu_\xi^\sharp \colon C^+_{E/\mQ}(\motV_\xi^\sharp,\tilde\phi^\sharp) \isomarrow \ul\End_D\bigl(\motH^1(A^\sharp_\xi)\bigr)$. On the other hand, $\cX_\xi \cong X$, and by construction $\motV_{\xi,(\mQ)}$ is a direct factor of $\motH^2(X)\bigl(1\bigr)$. Hence if $X^\sharp$ is the variety obtained from~$X$ by blowing up $[E:\mQ]$ distinct points, $\motV_{\xi,(\mQ)}^\sharp$ is a direct factor of $\motH^2(X^\sharp)\bigl(1\bigr)$. Because $\mV^\sharp$ has odd rank over~$E$ we can apply Proposition~\ref{MTCDimOdd}. This gives the Mumford-Tate conjecture and the Tate conjecture for~$\motV_{\xi,(\mQ)}^\sharp$, from which the same conjectures for~$\motV_{\xi,(\mQ)}$ follow.  

\subsection{}
It remains to be shown that $C^+_{E/\mQ}(\motV_s^\sharp,\tilde\phi^\sharp)$ lies in $\Mot(\Ab)_\mC$ for some $s \in S(\mC)$. By Proposition~\ref{JumpPicard}(\romannumeral1) we can choose for~$s$ a point of~$S(\mC)$ such that the fibre $\mV_{s,\Betti}$ contains non-trivial Hodge classes. By the Lefschetz theorem on divisor classes, the Hodge classes in~$\mV_{s,\Betti}$, which form an $E$-submodule, are algebraic; hence there exists a decomposition $\motV_s = \motV_s^\flat \oplus \unitmot_E$ in $\Mot_{\mC,(E)}$. Write $V^\flat \subset \mV_s$ for the $E$-subspace underlying the Hodge realization of~$\motV_s^\flat$, and let $\tilde\phi^\flat$ be the restriction of the polarization form~$\tilde\phi$ (on the fibre at~$s$) to~$V^\flat$. The motivic Galois group of~$\motV_s$ is then an algebraic subgroup of $\OO_{E/\mQ}(V^\flat,\tilde\phi^\flat)$.

We now use a simple fact from representation theory. If $(U,q)$ is a quadratic space over~$E$ then 
\[
\wedge^\even \bigl(U \perp 1^{\oplus 2}\bigr) \cong \Bigl(\wedge^\even (U \perp 1) \Bigr)^{\oplus 2}
\]
as representations of $\OO(U,q)$. This can be restated as saying that $C^+(U \perp 1^{\oplus 2}) \cong C^+(U \perp 1)^{\oplus 2}$ as $\OO(U,q)$-modules. Applying the norm functor, it follows that 
\begin{equation}
\label{eq:C+(U+1)}
C^+_{E/\mQ}(U \perp 1^{\oplus 2}) \cong C^+_{E/\mQ}(U \perp 1)^{\oplus 2^{[E:\mQ]}}
\end{equation}
as $\OO_{E/\mQ}(U,q)$-modules.

We apply the preceding remark to $U = V^\flat$, viewed as a representation of $G_\mot\bigl(\motV_s^\flat\bigr) = G_\mot\bigl(\motV_s\bigr)$. In this case the representation $U \perp 1$ corresponds to the motive~$\motV_s$ and $U \perp 1^{\oplus 2}$ corresponds to~$\motV_s^\sharp$. We have already seen in~\eqref{eq:motu} that $C^+_{E/\mQ}(\motV_s,\tilde\phi)$ is an object of the category of abelian motives. By \eqref{eq:C+(U+1)} it follows that $C^+_{E/\mQ}(\motV_s^\sharp,\tilde\phi^\sharp)$ lies in the subcategory of abelian motives, too, which is what we wanted to prove. This concludes the proof of Theorem~\ref{MainThm} if the field~$E$ is totally real and the variation~$\mV$ has maximal monodromy.


\section{Monodromy and the Mumford-Tate conjecture: the CM case}
\label{Mono=>MTC2}

\subsection{}
In this section we prove Theorem~\ref{MainThm} in the case where the endomorphism field $E = \End_{\QVHS_S}(\mV)$ (see~\ref{SetupContd}) is a CM-field. Let $\tau \colon E \to \mC$ be the (unique) complex embedding of~$E$ such that $e\in E$ acts on $(\mV \otimes_\mQ \mC)^{1,-1}$ as multiplication by~$\tau(e)$. Choose a CM-type $\Phi \subset \Sigma(E)$ such that $\tau \notin \Phi$. Let $A$ be the complex abelian variety (up to isogeny, as always) of CM-type $(E,\Phi)$. Concretely, $H^1(A,\mQ) = E$ as an $E$-module, and the Hodge decomposition of $H^1(A,\mC) = \oplus_{\sigma \in \Sigma(E)}\, \mC^{(\sigma)}$ is given by
\[
H^{1,0} = \bigoplus_{\sigma\in \Phi}\, \mC^{(\sigma)}\; ,\qquad 
H^{0,1} = \bigoplus_{\sigma\in \ol\Phi}\, \mC^{(\sigma)}\; .
\] 
Let $\psi_A$ be a polarization of $H^1(A,\mQ)$, and denote its unique lifting to a skew-hermitian $E$-valued form by~$\tilde\psi_A$. In what follows we will assume the CM-type~$\Phi$ is chosen to be primitive (not induced from a CM-subfield of~$E$); this is always possible and implies that $\End^0(A) = E$.

Let $a\colon A_S \to S$ denote the constant abelian scheme over~$S$ with fibres~$A$, and write $\mH^1(A_S)$ for the VHS $R^1a_*\mQ$, which is an $E$-module in the category $\QVHS_S$.

Next consider the variation $\mH^1(A_S) \otimes_E \mV$. (This is what van Geemen~\cite{vGeemen} calls a half-twist of~$\mV$.) Because of the way we have chosen~$\Phi$, this is a variation of Hodge structure of type $(0,1) + (1,0)$. Further, it is polarized by the form $\tilde\psi_A \otimes \tilde\phi$, it admits an integral structure, and it comes equipped with an action of~$E$ by endomorphisms. Hence we obtain a polarized abelian scheme $b \colon B \to S$ with multiplication by~$E$ such that there exists an isomorphism $\mH^1(B) = R^1b_* \mQ_B \isomarrow \mH^1(A_S) \otimes_E \mV$ in $\QVHS_{S,(E)}$ that is compatible with polarizations. We shall denote by $\psi_B \colon \mH^1(B) \times_\mQ \mH^1(B) \to \mQ_S$ the polarization form and by $\tilde\psi_B = \tilde\psi_A \otimes \tilde\phi$ its unique lift to an $E_S$-valued skew-hermitian form.

As $\ul\End_E\bigl(\mH^1(A_S)\bigr) = E$, we have an induced isometry
\begin{equation}
\label{eq:uHomVinVHS}
u\colon \ul\Hom_E\bigl(\mH^1(A_S),\mH^1(B)\bigr) \isomarrow \mV  
\end{equation}
in $\QVHS_{S,(E)}$, where we equip the left hand side with the $E$-hermitian form $\tilde\psi = \tilde\psi_A^\vee \otimes \tilde\psi_B$.

\subsection{}
For $s \in S(\mC)$, let $\motV_s \subset \motH^2(\cX_s)\bigl(1\bigr)$ be the submotive as in~\ref{motVConstr}, and write $\motH_s$ for the motive $\ul\Hom_E\bigl(\motH^1(A),\motH^1(B_s)\bigr)$. Both $\motV_s$ and~$\motH_s$ are objects of $\Mot_{\mC,(E)}$. We denote by $V_s$ and~$H_s$ the $E$-vector spaces underlying their Betti realizations. These come equipped with hermitian polarization forms $\tilde\phi \colon V_s \times V_s \to E$ and $\tilde\psi = (\psi_A^\vee \otimes \psi_B) \colon H_s \times H_s \to E$.

The fibre of the isomorphism \eqref{eq:uHomVinVHS} at~$s$ is an $E$-linear and $\pi_1(S,s)$-equivariant isomorphism
\[
u_s \colon H_s \isomarrow V_s\, .
\]
The motivic Galois group $\cG_\mot$ acts on the vector space $\Hom_\mQ(H_{s,(\mQ)},V_{s,(\mQ)})$.

\begin{lemma}
\label{chiGmotUE}
There is a unique (algebraic) character $\chi \colon \cG_\mot \to T_E$ such that $\gamma(u_s) = \chi(\gamma) \cdot u_s$ for all $\gamma \in \cG_\mot$, and this character is independent of~$s$.
\end{lemma}

\begin{proof}
With notation as explained in~\ref{NotConvIntro}(c), the actions of $\cG_\mot$ on $H_s$ and~$V_s$ are given by homomorphisms $\cG_\mot \to \UU_{E/\mQ}(H_s,\tilde\psi)$ and $\cG_\mot \to \UU_{E/\mQ}(V_s,\tilde\phi)$, respectively. Hence the induced action of~$\cG_\mot$ on the space $\Hom_\mQ(H_{s,(\mQ)},V_{s,(\mQ)})$ preserves the subspace $\Hom_E(H_s,V_s)$ and $\cG_\mot$ acts on this subspace by $E$-linear automorphisms. On the other hand, by Lemma 6.1.1 of~\cite{YAK3}, the action of~$\cG_\mot$ also preserves the subspace $\Hom_\mQ(H_{s,(\mQ)},V_{s,(\mQ)})^{\pi_1(S,s)}$ of monodromy invariant elements, which contains~$u_s$. By Proposition~\ref{JumpPicard}, assumption~(b) in Theorem~\ref{MainThm} implies that $G_\mono(\mV) = \SU_{E/\mQ}(V_s,\tilde\phi)$, and therefore $\End_E(V)^{G_\mono(\mV)} = E \cdot \id_V$. Hence
\[
\Hom_E(H_s,V_s)^{\pi_1(S,s)} = E \cdot u_s\, ;
\]
so there is a character $\chi \colon \cG_\mot \to T_E$ such that $\gamma(u_s) = \chi(\gamma) \cdot u_s$ for all $\gamma \in \cG_\mot$.

To see that $\chi$ does not depend on~$s$, we use Corollary~5.1 of~\cite{YAPour}, which implies that, for $t \in S(\mC)$ a second point, the isomorphism 
\[
\Hom_E(H_s,V_s)^{\pi_1(S,s)} \isomarrow \Hom_E(H_t,V_t)^{\pi_1(S,t)}
\]
given by parallel transport is $\cG_\mot$-equivariant. (See also Theorem 10.1.3.1 of~\cite{YAIntro}.) This readily gives the claim.
\end{proof}

\subsection{}
\label{DescendingBis}
The action of $\cG_\mot$ on the vector space~$E$ through the character~$\chi$ of Lemma~\ref{chiGmotUE} defines a motive~$\motU$ with multiplication by~$E$ such that the Hodge realization of~$\motU$ is trivial, and such that for $s \in S(\mC)$, we have an isomorphism $\motu_s\colon \motH_s \isomarrow \motV_s \otimes_E \motU$ whose Hodge realization is~$u_s$.

We shall next prove that the $\ell$-adic realization of~$\motU$ is trivial, too. (Of course, the motive~$\motU$ itself should be trivial; this, however, we are unable to prove.) First we descend, similar to what we did in~\ref{descending}, to a field of finite type over~$\mQ$. This means that, given $s \in S(\mC)$, there exists a subfield $K \subset \mC$ that is finitely generated over~$\mQ$ over which all objects that we are considering are defined. As we do not want to introduce new notation, we shall use the same letters as before for the objects over~$K$, adorning them with a subscript~``$\mC$'' to indicate an extension of scalars to~$\mC$. Thus, we have:
\begin{enumerate}[label=---]
\item a smooth projective variety $\cX_s$ over~$K$ and a submotive $\motV_s \subset \motH^2(\cX_s)\bigl(1\bigr)$, cut out by an algebraic cycle;
\item an action of~$E$ on~$\motV_s$;
\item abelian varieties $A$ and~$B_s$ over~$K$, both with multiplication by~$E$, and the associated motive $\motH_s = \ul\Hom_E\bigl(\motH^1(A),\motH^1(B_s)\bigr)$;
\item a motive $\motU$ with multiplication by~$E$ and an isomorphism $\motu_s\colon \motH_s \isomarrow \motV_s \otimes_E \motU$ in $\Mot_{K,(E)}$.
\end{enumerate}

\noindent
These objects are chosen in such a way that after extension of scalars via $K \hookrightarrow \mC$ we recover the objects considered above. Moreover, we may assume that for all motives involved the associated $\ell$-adic groups~$G_\ell$ are connected.

\begin{lemma}
\label{MlTrivial}
The $\ell$-adic realization of the motive~$\motU$ is trivial, i.e., $G_\ell^0(\motU) = \{1\}$.
\end{lemma}

For our later arguments, it is important to note that this assertion is independent of the choice of a model of~$\motU$ over a finitely generated field; see Proposition~\ref{FieldExtProp}.
\medskip

\begin{proof}
By Proposition~\ref{JumpPicard}(\romannumeral1), there exists a point $s \in S(\mC)$ such that $\motV_s^\alg \neq 0$, and such a point we now choose. As the Hodge realization of~$\motU$ is trivial, the classes in $V_{s,\Betti}^\alg \otimes_E U_\Betti$, viewed as subspace of $\Hom_E\bigl(H^1(A),H^1(B_s)\bigr)$ are algebraic. Hence also the $\ell$-adic realization $V_{s,\ell}^\alg \otimes_{E_\ell} U_\ell$ is trivial. Because $V_{s,\ell}^\alg$ is a trivial Galois representation, the assertion follows. 
\end{proof}

\subsection{}
In the rest of the argument we take for $s$ the point $\xi \in S(\mC)$ such that $X \cong \cX_\xi$. (Note that, even though we have worked with a different point~$s$ in the proof of Lemma~\ref{MlTrivial}, involving different choices in~\ref{DescendingBis}, the conclusion of~\ref{MlTrivial} applies to any form of the motive~$\motU$ over a finitely generated field.) To simplify the notation, we write $\motV$ for~$\motV_\xi$. As usual, $V_\Betti$ denotes the Hodge realization, $V_\ell$ the $\ell$-adic realization, and $V$ is the $E$-vector space underlying~$V_\Betti$. 

As $\motV \otimes_E \motU$ is a submotive of $\ul\Hom\bigl(\motH^1(A),\motH^1(B_\xi)\bigr)$ we can apply the results of Faltings to it. By Lemma~\ref{MlTrivial} it follows that the Galois representation~$V_\ell$ is completely reducible and
\[
\bigl(V_\ell\bigr)^{G_\ell^0(\motV)} \isomarrow \bigl(V_\ell \otimes_{E_\ell} U_\ell\bigr)^{G_\ell^0(\motV\otimes_E \motU)} = \bigl(V_\Betti \otimes_E U_\Betti\bigr)^{G_\Betti(\motV\otimes_E \motU)} \otimes \Ql \isomarrow  \bigl(V_\Betti\bigr)^{G_\Betti(\motV)} \otimes \Ql\, .
\]
The Tate conjecture for divisor classes on~$X$ then follows from the Lefschetz theorem on divisor classes on~$X_\mC$.

\subsection{}
The proof of the Mumford-Tate conjecture is now based on essentially the same argument as in~\ref{PfMTCDimOdd}. We again consider the decomposition $\motV = \motV^\trans \oplus \motV^\alg$ in $\Mot_{K,(E)}$. By the Tate conjecture there are no non-zero Tate classes in~$V_{\ell,(\Ql)}^\trans$. By Pink's theorem~\ref{GenbyWHC}, $G^0_\ell(\motV) \otimes \Qlbar$ is generated by the images of weak Hodge cocharacters. 

Write $\Gamma_K = \Gal(\Kbar/K)$. The motive $\motU$ has trivial Hodge and $\ell$-adic realizations, and for the abelian variety~$A$ we have $\ul\End_E\bigl(\motH^1(A)\bigr) = \unitmot_E$. Using this, we find that 
\[
\End_{\QHS_{(E)}}(V_\Betti) = \End_{\QHS_{(E)}}\bigl(H^1(B_\xi,\mQ)\bigr)\, ,\quad
\End_{E_\ell[\Gamma_K]}(V_\ell) = \End_{E_\ell[\Gamma_K]}\bigl(H^1(B_\xi,\Ql)\bigr)\, .
\]
Again by the results of Faltings it follows that
\begin{equation}
\label{eq:EndE=EndEBis}
\End_{E_\ell[\Gamma_K]}(V_\ell) = \End_{\QHS_{(E)}}(V_\Betti) \otimes_E E_\ell 
\end{equation}
under the comparison isomorphism between Betti and $\ell$-adic cohomology.

We know that $\End_\QHS(V_{\Betti,(\mQ)}^\trans)$ is a field that contains~$E$. It follows that $\End_{\QHS_{(E)}}(V_\Betti^\trans) \cong \End_\QHS(V_{\Betti,(\mQ)}^\trans)$. Similarly, because there are no non-zero Tate classes in~$V_{\ell,(\Ql)}^\trans$, it follows from Theorem~\ref{ZarhK3ell} that $\End_{\Ql[\Gamma_K]}(V_{\ell,(\Ql)}^\trans)$ is a commutative semisimple $\Ql$-algebra that contains~$E_\ell$; hence $\End_{E_\ell[\Gamma_K]}(V_\ell^\trans) = \End_{\Ql[\Gamma_K]}(V_{\ell,(\Ql)}^\trans)$. Then~\eqref{eq:EndE=EndEBis} gives $\End_\QHS(V_{\Betti,(\mQ)}) \otimes \Ql = \End_{\Ql[\Gamma_K]}(V_{\ell,(\Ql)})$, and the Mumford-Tate conjecture $G_\Betti(\motV) \otimes \Ql = G_\ell^0(\motV)$ follows from Corollary~\ref{HowtogetG=G'}. This completes the proof of Theorem~\ref{MainThm} in the case that $E$ is a CM-field.


\section{The case of non-maximal monodromy}
\label{Mono=>MTC3}

\subsection{}
\label{QuatCaseStart}
In this section we complete the proof of Theorem~\ref{MainThm} by considering the case where the monodromy of the variation~$\mV$ (as in~\ref{SetupContd}) is not maximal, which means that $G_\mono(\mV)$ is a proper normal subgroup of~$G_\Betti^\der(\mV)$. As shown in Proposition~\ref{JumpPicard}, this implies that the endomorphism field~$E$ is totally real and that $\rank_E(\mV) = 4$.

We choose a base point $z \in S$, and we let $V = \mV_z$, on which we have a symmetric $E$-bilinear form~$\tilde\phi$. The generic Mumford-Tate group~$G_\Betti(\mV)$ equals $\SO_{E/\mQ}(V,\tilde\phi) = \Res_{E/\mQ}\, \SO(V,\tilde\phi)$.

Our assumption on the non-maximality of the monodromy implies that $\SO(V,\tilde\phi)$ is the almost-direct product of two connected normal subgroups $L_1$ and~$L_2$. (By \cite{BorTits}, Proposition~6.18, if $H$ is a semisimple group over~$E$, any connected normal subgroup of $\Res_{E/\mQ}\, H$ is of the form $\Res_{E/\mQ}\, N$, for some connected normal subgroup $N \triangleleft H$.) The~$L_i$ are $E$-forms of~$\SL_2$ and they commute element-wise. We choose the numbering such that $G_\mono(\mV) = \Res_{E/\mQ}(L_2)$. 

Let $\Delta_i \subset \End_E(V)$ be the $E$-subalgebra generated by the elements of~$L_i(E)$. The~$\Delta_i$ are central simple $E$-algebras of degree~$2$. Denote the canonical involution of~$\Delta_i$ by $x \mapsto 
\bar{x}$, and let $\Trd\colon \Delta_i \to E$ and $\Nrd\colon \Delta_i \to E$ be the reduced trace and norm maps, given by $x \mapsto x+\bar{x}$ and $x \mapsto x\bar{x}$, respectively. We view $\Delta_i^*$ as an algebraic group over~$E$. It is an algebraic subgroup of the group~$\GO(V,\tilde\phi)$ of orthogonal similitudes. The multiplier character is the character $\Nrd\colon \Delta_i^* \to \mG_{\mult,E}$, and $L_i = \Ker(\Nrd) \subset \Delta_i^*$. (Note that any character of~$\Delta_i^*$ is a power of~$\Nrd$, and that the multiplier agrees with the reduced norm on the scalars $E^* \subset \Delta_i^*$; hence the two are equal.)

The space~$V$ is free of rank~$1$ as a module over~$\Delta_i$ ($i=1,2$). We have isomorphisms $\Delta_2 \isomarrow \End_{\Delta_1}(V)$ and $\Delta_1 \isomarrow \End_{\Delta_2}(V)$, and a non-canonical isomorphism $\Delta_2 \cong \Delta_1^\opp$. (In particular, we see that the subalgebra of $G_\mono(\mV)$-invariants in $\End(V_{(\mQ)})$ equals $\Delta_1$ and therefore has centre~$E$, as claimed in the proof of Proposition~\ref{EMotivated}.) There is a unique $\sigma_0 \in \Sigma(E)$ such that $\SO(V,\tilde\phi) \otimes_{E,\sigma_0} \mR$ is non-compact; at this real place the~$\Delta_i$ are split, at all other real places of~$E$ they are non-split. (In particular, the~$\Delta_i$ can be split over~$E$ only if $E=\mQ$.)

There is an isomorphism
\[
\CSpin(V,\tilde\phi) \isomarrow \bigl\{(x_1,x_2) \in \Delta_1^* \times \Delta_2^* \bigm| \Nrd(x_1) \cdot \Nrd(x_2) = 1 \bigr\}
\] 
such that the homomorphism $\CSpin(V,\tilde\phi) \to \SO(V,\tilde\phi)$ sends $(x_1,x_2)$ to the automorphism $x_1x_2$ of~$V$. We can choose this isomorphism such that the scalar multiplication by $z \in E^*$ on the Clifford algebra corresponds to the element $(z\cdot \id_V,z^{-1}\cdot \id_V) \in \Delta_1^* \times \Delta_2^*$. 

Let $D = \FNorm_{E/\mQ}(\Delta_1)$, which is a central simple $\mQ$-algebra of degree~$2^{[E:\mQ]}$, of index at most~$2$ in the Brauer group of~$\mQ$. The $\mQ$-vector space $N(V) = \FNorm_{E/\mQ}(V)$ has a natural structure of a left $D$-module, for which it is free of rank~$1$.

\begin{proposition}
\label{uHomNVQuat}
With assumptions and notation as above, there exists a complex abelian variety~$A$ with $\End^0(A) \cong D^\opp$ and an abelian scheme $b \colon B \to S$ with multiplication by~$D^\opp$ such that we have an isomorphism
\begin{equation}
\label{eq:uQuatCase}
u \colon \ul\Hom_D\bigl(\mH^1(A_S),\mH^1(B)\bigr) \isomarrow \FNorm_{E/\mQ}(\mV) 
\end{equation}
in $\QVHS_S$.
\end{proposition}

The notation is the same as in Section~\ref{Mono=>MTC2}: we write $a \colon A_S \to S$ for the constant abelian scheme with fibres equal to~$A$, and we let $\mH^1(A_S) = R^1a_*\mQ$ and $\mH^1(B) = R^1b_*\mQ$.
\medskip

\begin{proof}
Let $H_1 = D$ and $H_2 = N(V)$, both viewed as left $D$-modules. With $\CSpinbar_{E/\mQ}(V,\tilde\phi) = \CSpin_{E/\mQ}(V,\tilde\phi)/T_E^1$ as in~\ref{CSpins}, we start by defining representations 
\[
\theta_i \colon \CSpinbar_{E/\mQ}(V,\tilde\phi) \to \GL_D(H_i) \subset \GL(H_i)
\] 
($i=1,2$) such that the induced action on $\ul\Hom_D(H_1,H_2)$ factors through $\SO_{E/\mQ}(V,\tilde\phi)$, and such that the map 
\begin{equation}
\label{eq:epsilonisom}
\epsilon \colon \ul\Hom_D(H_1,H_2) \isomarrow N(V)  
\end{equation}
given by $f \mapsto f(1)$ is an isomorphism of $\SO_{E/\mQ}(V,\tilde\phi)$-representations. For this, define $K_i = \Res_{E/\mQ}(\Delta_i^*)$. The universal polynomial map $\Delta_1 \to D = \FNorm_{E/\mQ}(\Delta_1)$ is multiplicative and defines a homomorphism of algebraic groups $K_1 \to D^*$ that factors through $K_1/T_E^1$. Through the composition 
\[
\gamma \colon \CSpin_{E/\mQ}(V,\tilde\phi) \hookrightarrow K_1 \times K_2 \twoheadrightarrow K_1/T_E^1 \hookrightarrow D^*
\]
we obtain an action of $\CSpin_{E/\mQ}(V,\tilde\phi)$ on~$H_1$, letting a group element~$x$ act on $H_1=D$ as right multiplication by $\gamma(x^{-1})$. Clearly this representation factors through $\CSpinbar_{E/\mQ}(V,\tilde\phi)$, and we let $\theta_1$ be the representation thus obtained. Next we note that $K_2 = \Res_{E/\mQ}(\Delta_2^*)$ naturally acts on $N(V) = \FNorm_{E/\mQ}(V)$ by $D$-module automorphisms (cf.\ Section~\ref{NormBasics}), and the corresponding homomorphism $K_2 \to \GL\bigl(N(V)\bigr)$ factors through $K_2/T_E^1$. For~$\theta_2$ we then take the representation of $\CSpinbar_{E/\mQ}(V,\tilde\phi)$ induced by the composition
\[
\CSpin_{E/\mQ}(V,\tilde\phi) \hookrightarrow K_1 \times K_2 \twoheadrightarrow K_2/T_E^1 \hookrightarrow \GL_D(H_2)\, .
\]
We claim that with these definitions the map~$\epsilon$ in~\eqref{eq:epsilonisom} is an isomorphism of representations of~$\SO_{E/\mQ}(V,\tilde\phi)$. To see this, note that~$\epsilon$ is the map obtained by applying the functor~$\FNorm_{E/\mQ}$ to the isomorphism $\ul\Hom_{\Delta_1}(\Delta_1,V) \isomarrow V$ given by $f \mapsto f(1)$. The claim now readily follows from the description of the homomorphism $\CSpin(V,\tilde\phi) \to \SO(V,\tilde\phi)$ given in~\ref{QuatCaseStart}.

The next step is to equip $H_1$ and~$H_2$ with actions of~$\pi_1(S,b)$ in such a way that the map~$\epsilon$ is $\pi_1$-equivariant. For this we simply take the trivial action on~$H_1$ and the given action on $H_2 = N(V)$. (Note that the latter action is obtained by composing the map $\pi_1(S,b) \to G_\mono(\mV)\bigl(\mQ\bigr) \subset K_2(\mQ)$ with the above homomorphism $K_2 \to \GL_D(H_2)$.) In this way we obtain $\mQ$-local systems~$\mH_1$ (constant) and~$\mH_2$ with actions of~$D$ from the left, such that $\epsilon$ is the fibre at~$b$ of an isomorphism of local systems $\ul\Hom_D(\mH_1,\mH_2) \isomarrow \FNorm_{E/\mQ}(\mV)$. For later use, note that the~$\mH_i$ admit an integral structure.

We now equip $\mH_1$ and~$\mH_2$ with the structure of a $\mQ$-VHS. Let $\upsilon \colon \tilde{S} \to S$ be the universal cover, and choose a point $\tilde{z} \in \tilde{S}$ above~$z$. This gives identifications $\upsilon^* \mH_i \cong H_i \times \tilde{S}$ and $\upsilon^* \FNorm_{E/\mQ}(\mV) \cong N(V) \times \tilde{S}$. For $t \in \tilde{S}$, let $h_t \colon \mS \to \SO_{E/\mQ}(V,\tilde\phi)_\mR$ be the homomorphism that defines the Hodge structure on $\FNorm_{E/\mQ}(\mV)_t$, which is of type $(-1,1) + (0,0) + (1,-1)$. By \cite{DelK3}, 4.2, and the discussion in Section~\ref{CSpinbar} above, $h_t$ naturally lifts to a homomorphism $\tilde{h}_t \colon \mS \to \CSpinbar_{E/\mQ}(V,\tilde\phi)_\mR$ such that $\tilde{h}_t \circ w \colon \mG_{\mult,\mR} \to \CSpinbar_{E/\mQ}(V,\tilde\phi)_\mR$ is given by $z \mapsto (z\cdot \id_V,z^{-1}\cdot \id_V) \bmod T_E^1$. As Hodge structure on $(\upsilon^* \mH_i)_t \cong H_i$ we then take the one defined by $\theta_i \circ \tilde{h}_t$. By construction these Hodge structures have weight~$1$, and by \cite{DelK3}, Lemme~2.8, they are polarizable. 

It is immediate from the definitions that the family of Hodge structures on~$\upsilon^* \mH_i$ that we obtain is compatible (in the obvious sense) with the action of $\pi_1(S,b)$; hence we obtain families of polarizable Hodge structures of weight~$1$ over~$S$ with underlying local systems the~$\mH_i$, equipped with a left action of~$D$ by endomorphisms, and such that the isomorphism $\ul\Hom_D(\mH_1,\mH_2) \isomarrow \FNorm_{E/\mQ}(\mV)$ is fibrewise an isomorphism of Hodge structures. 

We claim that for every $t \in \tilde{S}$ the Hodge structures on~$H_i$ ($i=1,2$) given by $\theta_i \circ \tilde{h}_t$ is of type $(0,1) + (1,0)$. To see this, work in the category of $\mZ^2$-graded $\mC$-vector spaces. As $D \otimes_\mQ \mC \cong M_N(\mC)$, with $N = 2^{[E:\mQ]}$, we have $H_i \otimes_\mQ \mC \cong U_i^{\oplus N}$ for some $\mZ^2$-graded spaces~$U_i$ ($i=1,2$), and then $\ul\Hom_{D \otimes_\mQ \mC}(H_{1,\mC},H_{2,\mC}) \cong \ul\Hom_\mC(U_1,U_2)$. As we already know the~$H_i$ to be of weight~$1$, so that $U_i^{p,q} \neq 0$ only if $p+q=1$, we see that $\ul\Hom_\mC(U_1,U_2)$ can be of type $(-1,1) + (0,0) + (1,-1)$ only if $U_1$ and~$U_2$ (and hence also $H_1$ and~$H_2$) are of type $(0,1) + (1,0)$. In particular, the families of Hodge structures we have obtained trivially satisfy Griffiths transversality and are therefore honest polarizable $\mQ$-variations of Hodge structure over~$S$. As they admit an integral structure, there exist abelian schemes (up to isogeny, as always) $a\colon A\to S$ and $b\colon B\to S$, both with multiplication by~$D^\opp$, such that $\mH_1 \cong R^1a_* \mQ_A$ and $\mH_2 \cong R^1b_*\mQ_B$. (The relative dimension of $A/S$ and $B/S$ equals $2^{2[E:\mQ]-1}$.) Further, because $\mH_1$ has trivial monodromy, $A/S$ is a constant abelian scheme; so $A= A_S$ for an abelian variety~$A$ with endomorphisms by~$D^\opp$. By construction, we have an isomorphism of $D$-modules $H^1(A,\mQ) \isomarrow D$, and the fact that the generic Mumford-Tate group of the VHS~$\mV$ equals $\SO_{E/\mQ}(V,\tilde\phi)$ implies that $G_\Betti(A) = K_1/T_E^1$. It follows from this that $\End^0(A) = D^\opp$. 
\end{proof}

\subsection{}
\label{RestrictionsGmot}
We retain the notation and assumptions of~\ref{QuatCaseStart}, and we fix $A$, $B$ and~$u$ as in Proposition~\ref{uHomNVQuat}. Choose a base point $z \in S$, write $\motV = \motV_z$ and $N(\motV) = \FNorm_{E/\mQ}(\motV)$, and let $V = \mV_z$ and $N(V) = \FNorm_{E/\mQ}(V)$ be their Hodge realizations. The $E$-bilinear polarization form~$\tilde\phi$ on~$V$ induces a symmetric $\mQ$-bilinear form~$N(\tilde\phi)$ on~$N(V)$. Further we define $\motH = \ul\Hom_D\bigl(\motH^1(A),\motH^1(B_z)\bigr)$, and we write $H$ for its Hodge realization. Our goal at this point is to deduce some non-trivial information about~$G_\mot\bigl(N(\motV)\bigr)$ by using \cite{YAK3}, Lemma~6.1.1.

Write $\SL_2$ for $\SL_{2,\Qbar}$. Denote by~$\St$ its standard $2$-dimensional representation, on which the determinant gives a symplectic form $\det \colon \St \times \St \to \Qbar$.

As before, for $\sigma \in \Sigma(E)$ we denote the extension of scalars via $\sigma \colon E \to \Qbar$ by a subscript~``$\sigma$''. For each such~$\sigma$, choose isomorphisms $L_{i,\sigma} \isomarrow \SL_2$ ($i=1,2$) and $V_\sigma \cong \St \boxtimes \St$ as representations of $L_{1,\sigma} \times L_{2,\sigma} \cong \SL_2\times \SL_2$, such that $\tilde\phi_\sigma$ corresponds with the orthogonal form $\det \boxtimes \det$. Let $\mu_2 = \bigl\{\pm(\id,\id)\bigr\} \subset \SL_2 \times \SL_2$, and let $r \in \OO\bigl(\St^{\boxtimes 2},\det^{\boxtimes 2}\bigr)$ be given by $r(x_1 \boxtimes x_2) = x_2 \boxtimes x_1$. We have an isomorphism
\[
\bigl((\SL_2 \times \SL_2)/\mu_2\bigr) \rtimes \{1,r\} \isomarrow \OO\bigl(\St^{\boxtimes 2},\det\nolimits^{\boxtimes 2}\bigr) \cong \OO(V_\sigma,\tilde\phi_\sigma)\, ,
\]
that restricts to an isomorphism $(\SL_2 \times \SL_2)/\mu_2  \isomarrow \SO\bigl(\St^{\boxtimes 2},\det\nolimits^{\boxtimes 2}\bigr) \cong \SO(V_\sigma,\tilde\phi_\sigma)$. Note that the reflection~$r$ acts on $(\SL_2 \times \SL_2)/\mu_2$ by exchange of the factors: $r[A_1,A_2] = [A_2,A_1]$.

We have $G_\Betti(\motV) \subseteq G_\mot(\motV) \subseteq \OO_{E/\mQ}(V,\tilde\phi)$ with $G_\Betti(\motV) \subseteq  \SO_{E/\mQ}(V,\tilde\phi)$. If $H$ is an algebraic subgroup of $\OO_{E/\mQ}(V,\tilde\phi)$, denote $H/(H\cap T_E^1)$ by~$\ol{H}$. For the motive~$N(\motV)$ we then have
\[
\begin{matrix}
G_\Betti\bigl(N(\motV)\bigr) &=& \ol{G}_\Betti(\motV) &\subseteq& \ol{\SO}_{E/\mQ}(V,\tilde\phi) &\longhookrightarrow&\SO\bigl(N(V),N(\tilde\phi)\bigr)\cr
\bigcap\!\vert &&\bigcap\!\vert &&\bigcap &&\bigcap\cr
G_\mot\bigl(N(\motV)\bigr) &=& \ol{G}_\mot(\motV) &\subseteq& \ol{\OO}_{E/\mQ}(V,\tilde\phi) &\longhookrightarrow& \OO\bigl(N(V),N(\tilde\phi)\bigr)
\end{matrix}
\]
Extending scalars to~$\Qbar$ we find
\[
\begin{matrix}
G_\Betti\bigl(N(\motV)\bigr)_\Qbar &\subseteq & \ol{\SO}_{E/\mQ}(V,\tilde\phi)_\Qbar &=& \Bigl[\prod_{\sigma\in \Sigma(E)}\; (\SL_2\times \SL_2) \Bigr]\Big/\Theta \cr
\bigcap\!\vert &&\bigcap\!\vert &&\bigcap\cr
G_\mot\bigl(N(\motV)\bigr)_\Qbar &\subseteq & \ol{\OO}_{E/\mQ}(V,\tilde\phi)_\Qbar &=& \Bigl[\prod_{\sigma\in \Sigma(E)}\; (\SL_2\times \SL_2) \rtimes \{1,r\}\Bigr]\Big/\Theta\; , \end{matrix}
\]
where 
\[
\Theta = \Bigl\{(A_1^{(\sigma)},A_2^{(\sigma)})_\sigma \in \prod_{\sigma\in \Sigma(E)}\, \bigl(\{\pm\id\}\times \{\pm\id\}\bigr) \Bigm| \prod A_1^{(\sigma)}A_2^{(\sigma)} = \id \Bigr\}\, .
\]
In this description, the algebraic monodromy group of the variation $N(\mV) = \FNorm_{E/\mQ}(\mV)$ is given by
\[
G_\mono\bigl(N(\mV)\bigr)_\Qbar = \hbox{image of}\quad \prod_{\sigma\in \Sigma(E)}\; (\{\id\} \times \SL_2)\quad \hbox{in}\quad \Bigl[\prod_{\sigma\in \Sigma(E)}\; (\SL_2\times \SL_2) \Bigr]\Big/\Theta\, .
\]

The fibre at~$z$ of the isomorphism~$u$ in~\eqref{eq:uQuatCase} is an isomorphism of Hodge structures $u_z \colon H \isomarrow N(V)$ that we shall take as an identification. As $\motH$ is an abelian motive, we have $G_\mot(\motH) = G_\Betti(\motH) = G_\Betti\bigl(N(\motV)\bigr)$.

The motivic Galois group of the motive $\motM = \ul\Hom\bigl(\motH,N(\motV)\bigr)$ is an algebraic subgroup 
\[
G_\mot(\motM) \subset G_\mot(\motH) \times G_\mot\bigl(N(\motV)\bigr)
\] 
that projects surjectively to the two factors. We shall denote elements of $G_\mot(\motM)_\Qbar$ in the form $\bigl[A_1^{(\sigma)},A_2^{(\sigma)};B_1^{(\sigma)},B_2^{(\sigma)},y^{(\sigma)}\bigr]_{\sigma \in \Sigma(E)}$ with $A_i^{(\sigma)}$, $B_i^{(\sigma)} \in \SL_2$ and $y^{(\sigma)} \in \{1,r\}$, and where the square brackets indicate that we calculate modulo $\Theta \times \Theta$. 

The $\mQ$-vector space underlying the Betti realization of the motive~$\motM$ is the space $M = \ul\Hom\bigl(H,N(V)\bigr)$. Extending scalars to~$\Qbar$ and identifying $H$ with~$N(V)$ via~$u_z$, we have
\[
M \otimes \Qbar \cong \bigotimes_{\sigma\in \Sigma(E)} \End(V_\sigma) \cong \bigotimes_{\sigma\in \Sigma(E)} \bigl(\End(\St) \otimes \End(\St)\bigr)\, . 
\]
For $\sigma \in \Sigma(E)$ fixed, an element $\bigl(A_1^{(\sigma)},A_2^{(\sigma)};B_1^{(\sigma)},B_2^{(\sigma)},1\bigr)$ acts on $\End(\St) \otimes \End(\St)$ by 
\[
f_1 \otimes f_2 \mapsto \Bigl(B_1^{(\sigma)} f_1 \bigl(A_1^{(\sigma)}\bigr)^{-1}\Bigr) \otimes \Bigl(B_2^{(\sigma)} f_2 \bigl(A_2^{(\sigma)}\bigr)^{-1}\Bigr)\, .
\]
An element $\bigl[A_1^{(\sigma)},A_2^{(\sigma)};B_1^{(\sigma)},B_2^{(\sigma)},r\bigr]$ acts on $\End(\St) \otimes \End(\St)$ by 
\[
f_1 \otimes f_2 \mapsto  \Bigl(B_2^{(\sigma)} f_2 \bigl(A_2^{(\sigma)}\bigr)^{-1}\Bigr) \otimes \Bigl(B_1^{(\sigma)} f_1 \bigl(A_1^{(\sigma)}\bigr)^{-1}\Bigr)\, .
\]

Recall that $D = \FNorm_{E/\mQ}(\Delta_1)$ and that $N(V)$ is free of rank~$1$ as a module over~$D$. The subspace $M^{\pi_1(S,b)}$ of monodromy-equivariant homomorphisms $H \to N(V)$ is the $D$-submodule $D \cdot u_z \subset M = \ul\Hom\bigl(H,N(V)\bigr)$ generated by~$u_z$. After extension of scalars this gives
\[
M^{\pi_1(S,b)}_\Qbar = \bigotimes_{\sigma\in \Sigma(E)} \bigl(\End(\St) \otimes (\Qbar \cdot \id)\bigr)\, .
\]
By \cite{YAK3}, Lemma~6.1.1, this subspace is stable under the action of~$G_\mot(\motM)_\Qbar$ on $M_\Qbar$. This implies that, in order for an element 
\[
\bigl[A_1^{(\sigma)},A_2^{(\sigma)};B_1^{(\sigma)},B_2^{(\sigma)},y^{(\sigma)}\bigr]_{\sigma \in \Sigma(E)} \in \ol{\SO}_{E/\mQ}(V,\tilde\phi)\bigl(\Qbar\bigr) \times \ol{\OO}_{E/\mQ}(V,\tilde\phi)\bigl(\Qbar\bigr)
\]
to lie in~$G_\mot(\motM)\bigl(\Qbar\bigr)$, we must have $y^{(\sigma)}= 1$ and $B_2^{(\sigma)} = \pm A_2^{(\sigma)}$ for every $\sigma \in \Sigma(E)$. In particular, it follows that $G_\mot(\motM)$ is an algebraic subgroup of $\ol{\SO}_{E/\mQ}(V,\tilde\phi) \times \ol{\SO}_{E/\mQ}(V,\tilde\phi)$, and, by projection to the second factor, that $G_\mot\bigl(N(\motV)\bigr) \subseteq \ol{\SO}_{E/\mQ}(V,\tilde\phi)$.

\begin{lemma}
\label{GB=SO3}
Suppose the base point $z \in S$ is a point for which $V = \mV_z$ contains non-zero Hodge classes. Then $V^\alg \subset V$ is a $1$-dimensional $E$-subspace, $\dim_E(V^\trans) = 3$, and $G_\Betti(\motV_z) = \SO_{E/\mQ}(V^\trans,\tilde\phi)$. 
\end{lemma}

\begin{proof}
For any $z \in S$ we have $\motV_z^\alg \cong \unitmot_E^{\oplus \nu}$ for some $\nu \geq 0$ (cf.~\ref{motVinMotE}), and then $\dim_E(V^\trans) = 4-\nu$. Further, $G_\Betti(\motV_z) = G_\Betti(\motV_z^\trans)$ is isomorphic to an algebraic subgroup of $\SO_{E/\mQ}(V^\trans,\tilde\phi)$. By \cite{YAMTGps}, Section~6, Proposition~2, our assumption that the variation~$\mV$ has non-maximal monodromy implies that $G_\Betti(\motV_z)$ is not abelian, and therefore $4-\nu \geq 3$. So if $z \in S$ is a point for which $\nu > 0$ then necessarily $\nu = 1$. We then must have $\End_\QHS(\mV_z^\trans) = E$, for if the endomorphism algebra is bigger, $G_\Betti(\motV_z^\trans)$ is abelian. By Zarhin's results (see~\ref{ZarhResult}), $G_\Betti(\motV_z) = \SO_{E/\mQ}(V^\trans,\tilde\phi)$.
\end{proof}

\begin{proposition}
\label{usMotivQuat}
For any $s \in S$ the fibre $u_s \colon \ul\Hom_D\bigl(\mH^1(A),\mH^1(B_s)\bigr) \isomarrow \FNorm_{E/\mQ}(\mV_s)$ of the isomorphism\/~{\rm \eqref{eq:uQuatCase}} at~$s$ is a motivated cycle, i.e., it is the Hodge realization of an isomorphism
\[
\motu_s \colon \ul\Hom_D\bigl(\motH^1(A),\motH^1(B_s)\bigr) \isomarrow \FNorm_{E/\mQ}(\motV_s)\, .
\]
\end{proposition}

\begin{proof}
By \cite{YAPour}, Th\'eor\`eme~0.5, it suffices to prove that~$u_s$ is a motivated cycle for some $s \in S$. We take $s=z$, where $z \in S$ is a base point such that $V = \mV_z$ contains non-zero Hodge classes. (Such points~$z$ exists by Proposition~\ref{JumpPicard}(\romannumeral1).) We retain the notation introduced in~\ref{RestrictionsGmot}.

If $H$ is an algebraic group over~$\mQ$, we denote by $\varpi_0(H) = H/H^0$ the \'etale group scheme of connected components. Note that
\[
\Bigl(\OO_{E/\mQ}(V,\tilde\phi) \cap T_E^1\Bigr)(\Qbar) = \Bigl\{(a_\sigma \cdot \id_V)_\sigma \in \prod_{\sigma\in \Sigma(E)}\, \OO(V_\sigma,\tilde\phi_\sigma) \Bigm| \hbox{$a_\sigma=\pm 1$ for all $\sigma$ and $\prod a_\sigma = 1$}\Bigr\}\, .
\]
As $\dim_E(V) = 4$ is even, it follows that $\OO_{E/\mQ}(V,\tilde\phi) \cap T_E^1 = \SO_{E/\mQ}(V,\tilde\phi) \cap T_E^1$; hence, 
\begin{equation}
\label{eq:varpi0iso}
\varpi_0\bigl(\OO_{E/\mQ}(V,\tilde\phi)\bigr) \isomarrow \varpi_0\bigl(\ol\OO_{E/\mQ}(V,\tilde\phi)\bigr)\, . 
\end{equation}
On the other hand, it is clear that $\bigl(\OO_{E/\mQ}(V^\trans,\tilde\phi) \times  \{\id_{V^\alg}\}\bigr) \cap T_E^1 = \{1\}$, so the inclusion $i \colon \OO_{E/\mQ}(V^\trans,\tilde\phi) \times  \{\id_{V^\alg}\} \longhookrightarrow \OO_{E/\mQ}(V,\tilde\phi)$ induces an injective homomorphism
\[
\ol{i} \colon \OO_{E/\mQ}(V^\trans,\tilde\phi) \times  \{\id_{V^\alg}\} \longhookrightarrow \ol\OO_{E/\mQ}(V,\tilde\phi)\, .
\]
As $i$ induces an isomorphism on component group schemes, it follows from~\eqref{eq:varpi0iso} that the same is true for~$\ol{i}$. So the fact, deduced in~\ref{RestrictionsGmot}, that $G_\mot\bigl(N(\motV)\bigr)$ is contained in $\ol\SO_{E/\mQ}(V,\tilde\phi)$, implies that $G_\mot\bigl(N(\motV)\bigr) \subseteq \SO_{E/\mQ}(V^\trans,\tilde\phi) \times  \{\id_{V^\alg}\}$, which we view as algebraic subgroup of $\SO\bigl(N(V),N(\tilde\phi)\bigr)$. As we always have $G_\Betti \subseteq G_\mot$, it follows from Lemma~\ref{GB=SO3} that
\[
G_\Betti\bigl(N(\motV)\bigr) = G_\mot\bigl(N(\motV)\bigr) = \SO_{E/\mQ}(V^\trans,\tilde\phi) \times  \{\id_{V^\alg}\}\, .
\]

Next we consider the motive $\motM = \ul\Hom\bigl(\motH,N(\motV)\bigr)$. As before, we take the isomorphism $u_z \colon H \isomarrow N(V)$ as an identification. The Mumford-Tate group $G_\Betti(\motM)$ is the diagonal subgroup $\SO_{E/\mQ}(V^\trans,\tilde\phi) \times  \{\id_{V^\alg}\}$ of
\[
G_\Betti(\motH) \times G_\Betti\bigl(N(\motV)\bigr) = \Bigl(\SO_{E/\mQ}(V^\trans,\tilde\phi) \times  \{\id_{V^\alg}\}\Bigr) \times \Bigl(\SO_{E/\mQ}(V^\trans,\tilde\phi) \times  \{\id_{V^\alg}\}\Bigr)\, .
\]
The motivic Galois group $G_\mot(\motM)$ is an algebraic subgroup of 
\[
G_\mot(\motH) \times G_\mot\bigl(N(\motV)\bigr) = G_\Betti(\motH) \times G_\Betti\bigl(N(\motV)\bigr) 
\] 
that contains the diagonal subgroup $G_\Betti(\motM)$. We are done if we can show that $G_\mot(\motM) = G_\Betti(\motM)$.

First we show that the identity component $G_\mot^0(\motM)$ of~$G_\mot(\motM)$ equals $G_\Betti(\motM)$. To see this, we can argue on Lie algebras. We have $\mathfrak{g}_\mot(\motM) \subset \mathfrak{g}_\mot(\motH) \times \mathfrak{g}_\mot\bigl(N(\motV)\bigr)$, the projections of $\mathfrak{g}_\mot(\motM)$ to $\mathfrak{g}_\mot(\motH)$ and to $\mathfrak{g}_\mot\bigl(N(\motV)\bigr)$ are surjective, and $\mathfrak{g}_\mot(\motM)$ contains the diagonal Lie subalgebra.

After extension of scalars to~$\Qbar$, the inclusion $\mathfrak{g}_\mot(\motM) \subset \mathfrak{g}_\mot(\motH) \times \mathfrak{g}_\mot\bigl(N(\motV)\bigr)$ becomes
\begin{equation}
\label{eq:Prodso3s}
\mathfrak{g}_\mot(\motM) \otimes \Qbar \subset \Bigl(\prod_{\sigma\in \Sigma(E)}\, \so_3^{(\sigma)}\Bigr) \times \Bigl(\prod_{\tau\in \Sigma(E)}\, \so_3^{(\tau)}\Bigr) 
\end{equation}
where the superscripts ``$(\sigma)$'' and~``$(\tau)$'' are included only to label the factors. Suppose that $\mathfrak{g}_\mot(\motM)$ is strictly bigger than the diagonal Lie subalgebra of $\mathfrak{g}_\mot(\motH) \times \mathfrak{g}_\mot\bigl(N(\motV)\bigr)$. This means that there exists an embedding $\sigma\in \Sigma(E)$ such that the projection of $\mathfrak{g}_\mot(\motM) \otimes \Qbar$ to $\so_3^{(\sigma)} \times \so_3^{(\sigma)}$ (taking $\tau=\sigma$ in~\eqref{eq:Prodso3s}) is surjective. (Note that $\so_3 \cong \gsl_2$ and that the only semisimple Lie subalgebras of $\gsl_2\times \gsl_2$ that contain the diagonal, are the diagonal itself and the full $\gsl_2\times \gsl_2$.) Because $\mathfrak{g}_\mot(\motM)$ is defined over~$\mQ$ and $\Gal(\Qbar/\mQ)$ acts transitively on~$\Sigma(E)$, it follows that $\mathfrak{g}_\mot(\motM) \otimes \Qbar$ surjects to $\so_3^{(\sigma)} \times \so_3^{(\sigma)}$ for all $\sigma \in \Sigma(E)$. As $\mathfrak{g}_\mot(\motM)$ surjects to $\mathfrak{g}_\mot(\motH)$ and $\mathfrak{g}_\mot\bigl(N(\motV)\bigr)$ and contains the diagonal, it follows that $\mathfrak{g}_\mot(\motM) \otimes \Qbar$ surjects to {\it any\/} product of two factors $\so_3 \times \so_3$ in the right-hand side of~\eqref{eq:Prodso3s}. By \cite{MoZa}, Lemma~2.14(\romannumeral1), this implies that $\mathfrak{g}_\mot(\motM) = \mathfrak{g}_\mot(\motH) \times \mathfrak{g}_\mot\bigl(N(\motV)\bigr)$ but this contradicts what we have found in~\ref{RestrictionsGmot}. So indeed $G_\mot^0(\motM) = G_\Betti(\motM)$ is the diagonal subgroup of $G_\Betti(\motH) \times G_\Betti\bigl(N(\motV)\bigr) \cong \SO_{E/\mQ}(V^\trans,\tilde\phi) \times \SO_{E/\mQ}(V^\trans,\tilde\phi)$.

Finally, $G_\mot(\motM)$ normalizes $G^0_\mot(\motM)$. But $\dim_E(V^\trans) = 3$, so $\SO_{E/\mQ}(V^\trans,\tilde\phi)$ has trivial centre. Hence the diagonal subgroup in $\SO_{E/\mQ}(V^\trans,\tilde\phi) \times \SO_{E/\mQ}(V^\trans,\tilde\phi)$ equals its own normalizer, and we conclude that $G_\mot(\motM) = G_\mot^0(\motM)$, as we wanted to show.
\end{proof}

\subsection{}
\label{TCQuatCase}
The Tate conjecture for~$\motV_\xi$ is now deduced in the same way as in \ref{GellinGB} and~\ref{TateClassesOK}. Taking $s=\xi$ in Proposition~\ref{usMotivQuat}, we find that $N(\motV_\xi) = \FNorm_{E/\mQ}(\motV_\xi)$ is an abelian motive, and by Proposition~\ref{MTCNorms}(\romannumeral1) it follows that $G_\ell^0(\motV_\xi)$ is reductive.

Further, by Deligne's results in~\cite{DelAbsHodge}, $G_\ell^0\bigl(N(\motV_\xi)\bigr) \subseteq G_\Betti\bigl(N(\motV_\xi)\bigr) \otimes \Ql$ as algebraic subgroups of $\GL\bigl(N(V_\xi)\bigr) \otimes \Ql$; hence
\begin{equation}
\label{eq:GlinGBQuat}
G_\ell^0(\motV_\xi) \subseteq G_\Betti(\motV_\xi) \otimes \Ql  
\end{equation}
as algebraic subgroups of $\GL(V_\xi) \otimes \Ql$. 

We have connected algebraic subgroups $H_\Betti \subseteq \SO(V_\xi,\tilde\phi)$ over~$E$ and $H_\ell \subseteq \SO(V_{\xi,\ell},\tilde\phi_\ell)$ over~$E_\ell$ such that $G_\Betti(\motV_\xi) = \Res_{E/\mQ}\, H_\Betti$ and $G_\ell^0(\motV_\xi) = \Res_{E_\ell/\Ql}\, H_\ell$. By the results of Faltings,
\[
\FNorm_{E_\ell/\Ql}\bigl(V_{\xi,\Betti}^{H_\Betti} \otimes_E E_\ell\bigr) = \FNorm_{E/\mQ}\bigl(V_{\xi,\Betti}^{H_\Betti}\bigr) \otimes_\mQ \Ql \isomarrow  \FNorm_{E_\ell/\Ql}\bigl(V_{\xi,\ell}^{H_\ell}\bigr)\, .
\]
Now we can copy the last eight lines of~\ref{TateClassesOK} with $\motV_\xi$ instead of~$\motW$, with as conclusion that all Tate classes in $H^2(X,\Ql)\bigl(1\bigr)$ are algebraic.
\medskip

For the proof of the Mumford-Tate conjecture for~$\motV_\xi$, we start with a lemma.

\begin{lemma}
\label{hinprodsl2}
Let $k$ be an algebraically closed field of characteristic zero, $\Sigma$ a finite index set, $\gh$ a reductive Lie subalgebra of $\prod_\Sigma\, \gsl_2$. Consider the representation of~$\gh$ on $V = \otimes_\Sigma\, M_2(k)$ obtained as the tensor product of the representations of~$\gsl_2$ on~$M_2(k)$ given by left multiplicaton. If $\End(V)^\gh = \otimes_\Sigma\, M_2(k)$ (acting on~$V$ by right multiplication) then $\gh = \prod_\Sigma\, \gsl_2$.
\end{lemma}

\begin{proof}
For $\sigma \in \Sigma$, the projection $\pr_\sigma \colon \gh \to \gsl_2$ is surjective, for otherwise the image is contained in a Cartan subalgebra $\gt \subset \gsl_2$, and since $\End(M_2(k))^\gt \supsetneq M_2(k)$ this contradicts the assumption that $\End(V)^\gh = \otimes_\Sigma\, M_2(k)$.

By \cite{Ribet}, the Lemma on pages 790--791, it now suffices to show that for $\sigma \neq \tau$ the projection $\pr_{\sigma,\tau} \colon \gh \to \gsl_2 \times \gsl_2$ is surjective. Write $\gh^\prime = \pr_{\sigma,\tau}(\gh) \subset \gsl_2\times \gsl_2$. Let $\gk_\sigma = \Ker(\pr_\tau \colon \gh^\prime \to \gsl_2)$ and $\gk_\tau = \Ker(\pr_\sigma \colon \gh^\prime \to \gsl_2)$; these are ideals of~$\gh^\prime$ with $\dim(\gk_\sigma) = \dim(\gk_\tau)$. As $\gh^\prime$ has rank at most~$2$ and the projections $\gh^\prime \to \gsl_2$ are surjective, the ideal $\gk_\sigma \times \gk_\tau$ is either the whole~$\gh^\prime$ or it is zero. In the first case we are done. In the second case, $\gh^\prime \subset \gsl_2\times \gsl_2$ is the graph of an automorphism. As all automorphisms of~$\gsl_2$ are inner, this contradicts the assumption that $\End(V)^\gh = \otimes_\Sigma\, M_2(k)$.
\end{proof}

\subsection{}
We retain the notation introduced in~\ref{QuatCaseStart}. In the argument that follows we shall also use the notation introduced in the proof of Proposition~\ref{uHomNVQuat}. If $Y$ is a complex abelian variety, we write $G_\mot(Y)$ (resp.\ $G_\Betti(Y)$, etc.) for $G_\mot\bigl(\motH^1(Y)\bigr)$ (resp.\ $G_\Betti\bigl(\motH^1(Y)\bigr)$, etc.). Let
\[
G_\Betti^\prime(Y) = \Bigl(G_\Betti(Y) \cap \SL\bigl(H^1(Y,\mQ)\bigr)\Bigr)^0\, ,\qquad
G_\ell^\prime(Y) = \Bigl(G_\ell(Y) \cap \SL\bigl(H^1(Y,\Ql)\bigr)\Bigr)^0\, .
\]
(The group $G_\Betti^\prime(Y)$ is called the Hodge group of~$Y$.) These are connected reductive groups, and $G_\Betti(Y) = \mG_\mult \cdot G_\Betti^\prime(Y)$, resp.\ $G_\ell^0(Y) = \mG_\mult \cdot G_\ell^\prime(Y)$.

Suppose the Mumford-Tate conjecture is true for (the $\motH^1$ of) the abelian variety $A \times B_\xi$. By the remarks in~\ref{MTCforV(n)} together with Proposition~\ref{MTCNorms}(\romannumeral4), the Mumford-Tate conjecture is then also true for the motive
\[
\ul\Hom\bigl(\motH^1(A),\motH^1(B_\xi)\bigr) \cong \motH^1(A)^\vee \otimes \motH^1(B_\xi) \cong \motH^1(A) \otimes \motH^1(B_\xi)\bigl(1\bigr)\, ,
\]
and hence for the motive $\ul\Hom_D\bigl(\motH^1(A),\motH^1(B_\xi)\bigr) \cong \FNorm_{E/\mQ}(\motV_\xi)$. By Proposition~\ref{MTCNorms}(\romannumeral3), this implies the Mumford-Tate conjecture for the motive~$\motV_\xi$.

For the abelian variety~$A$ we know (see the end of the proof of~\ref{uHomNVQuat}) that $\End^0(A) = D^\opp$ and that there is an isomorphism of $D$-modules $H^1(A,\mQ) \isomarrow H_1 = D$ via which $G_\Betti(A) = K_1/T_E^1$. (Recall: $K_1 = \Res_{E/\mQ}(\Delta_1^*)$.) The Mumford-Tate conjecture for~$A$ then follows from Lemma~\ref{hinprodsl2}, applied with $k = \Qlbar$ and $\gh = \mathfrak{g}^\prime_\ell(A) \otimes \Qlbar$.

Next we look at the abelian scheme $B \to S$ constructed in~\ref{uHomNVQuat}. By construction, for every $s \in S$ there is an isomorphism of $D$-modules $H^1(B_s,\mQ) \isomarrow H_2$ via which $G_\Betti(B_s) \subseteq K_2/T_E^1$. Moreover, for Hodge-generic points~$s$ the latter inclusion is an equality. The Shimura variety defined by the algebraic group $K_2/T_E^1 \subset \GL(H_2)$ is $1$-dimensional. (As discussed in~\ref{QuatCaseStart}, there is a unique real place of~$E$ at which the quaternion algebra~$\Delta_2$ splits.) It follows that either $G_\Betti(B_\xi) = K_2/T_E^1$ or $B_\xi$ is an abelian variety of CM-type.

If $B_\xi$ is of CM-type, which means that $G_\Betti(B_\xi)$ is a torus, the Mumford-Tate conjecture for~$B_\xi$ is true. Using the fact that $G_\Betti^\prime(A)$ and $G_\ell^\prime(A)$ are semisimple, it is easily shown that $G_\Betti^\prime(A\times B_\xi) = G_\Betti^\prime(A) \times G_\Betti^\prime(B_\xi)$ and $G_\ell^\prime(A \times B_\xi) = G_\ell^\prime(A) \times G_\ell^\prime(B_\xi)$. In particular, the Mumford-Tate conjecture is true for $A \times B_\xi$.

{}From now on we assume that $G_\Betti(B_\xi) = K_2/T_E^1$, and hence $\End^0(B_\xi) = D^\opp$. The Mumford-Tate conjecture for~$B_\xi$ is true, by the same argument as for~$A$. Note that $D^\opp \cong M_r(Q)$ for some $r \geq 1$ and some division algebra~$Q$; hence $A$ and~$B_\xi$ are both isogenous to the $r$th power of a simple abelian variety. It follows that if $\Hom(A,B_\xi) \neq 0$ then $A$ and~$B_\xi$ are isogenous, in which case the Mumford-Tate conjecture for $A\times B_\xi$ follows from the Mumford-Tate conjecture for~$A$. In what follows we may therefore also assume that $\Hom(A,B_\xi) = 0$.

Under the assumptions we have made, the groups $G_\Betti^\prime(A)$, $G_\ell^\prime(A)$, $G_\Betti^\prime(B_\xi)$ and $G_\ell^\prime(B_\xi)$ are all semisimple. As $N(\motV_\xi) \cong \ul\Hom_D\bigl(\motH^1(A),\motH^1(B_\xi)\bigr)$ this implies that also $G_\Betti\bigl(N(\motV_\xi)\bigr)$ and $G_\ell^0\bigl(N(\motV_\xi)\bigr)$ are semisimple. In \eqref{eq:G/Z=G}, applied with $\motV=\motV_\xi$, the group schemes $Z$ and~$Z_\ell$ are finite; hence also $G_\Betti(\motV_\xi)$ and $G_\ell^0(\motV_\xi)$ are semisimple.

The assumption that $\Hom(A,B_\xi) = 0$ implies that there are no non-zero Hodge (resp.\ Tate) classes in $\FNorm_{E/\mQ}(V_{\xi,\Betti})$ (resp.\ $\FNorm_{E_\ell/\Ql}(V_{\xi,\ell})$). Hence there are no non-zero Hodge (resp.\ Tate) classes in~$V_{\xi,\Betti}$ (resp.\ $V_{\xi,\ell}$). 

On the Hodge-theoretic side, we now know that $\dim_E(V_\xi) = 4$, that there are no non-zero Hodge classes in~$V_{\xi,\Betti}$, and that $G_\Betti(\motV_\xi)$ is semisimple. By Zarhin's results (see~\ref{ZarhResult}) it follows that $\End_{\QHS}(V_{\xi,\Betti}) = E$ and $G_\Betti(\motV_\xi) = \SO_{E/\mQ}(V_\xi,\tilde\phi)$. On the $\ell$-adic side, $E_\ell = \prod_{\lambda | \ell}\, E_\lambda$, and $(V_{\xi,\ell},\tilde\phi_\ell)$ decomposes as an orthogonal sum $\oplus_{\lambda|\ell}\, (V_{\xi,\lambda},\tilde\phi_\lambda)$ with $\dim_{E_\lambda}(V_{\xi,\lambda}) = 4$. Again we know there are no non-zero $G_\ell^0(\motV_\xi)$-invariants, and that  $G_\ell^0(\motV_\xi)$ has no unitary or abelian factors. By Theorem~\ref{ZarhK3ell}, it follows that $G_\ell^0(\motV_\xi) = \prod_{\lambda|\ell}\, \SO_{E_\lambda/\Ql}(V_{\xi,\lambda},\tilde\phi_\lambda) = G_\Betti(\motV_\xi) \otimes \Ql$. 

The proof of our Main Theorem~\ref{MainThm} is now complete.


\section{Applications to algebraic surfaces with \texorpdfstring{$p_g=1$}{pg=1}.}
\label{AlgSurf}

\subsection{}
In this section we work over~$\mC$, and all surfaces we consider are assumed to be complete. We shall mainly be interested in non-singular minimal surfaces of general type with $p_g=1$. Let $\cM = \coprod\, \cM_{K^2,1,q}$ be the moduli stack of such surfaces. If $M$ is an irreducible component of~$\cM$, we say that $M$ satisfies condition~\PeriodCond\ if there exist complex surfaces $X_1$, $X_2$ with $[X_i] \in M$ such that $H^2(X_1,\mQ) \not\cong H^2(X_2,\mQ)$ as $\mQ$-Hodge structures.

\begin{proposition}
\label{SimultRes}
Let $M$ be an irreducible component of~$\cM$ that satisfies condition\/~{\rm \PeriodCond}. If $X$ is a non-singular surface with $[X] \in M$, the Tate Conjecture for divisor classes on~$X$ is true and the Mumford-Tate conjecture for the cohomology in degree~$2$ is true.
\end{proposition}

\begin{proof}
By assumption, there exist complex surfaces $X_1$, $X_2$ with $[X_i] \in M$ and $H^2(X_1,\mQ) \not\cong H^2(X_2,\mQ)$ as $\mQ$-Hodge structures. We may assume $X=X_1$. There exist irreducible $\mC$-schemes~$U_i$ ($i=1,2$) of finite type and smooth morphisms $p_i \colon U_i \to M$ with $[X_i] \in p_i(U_i)$. Let $r_i \colon \tilde{U}_i \to U_{i,\red}$ be a resolution of singularities of the reduced scheme underlying~$U_i$, and let $\pi_i \colon\cY_i \to \tilde{U}_i$ denote the smooth family of surfaces given by $\tilde{U}_i \to M$. By construction, there exist points $u_i \in \tilde{U}_i(\mC)$ such that the fibre of~$\pi_i$ over~$u_i$ is the surface~$X_i$.

Denote by $\mH_i$ the variation of Hodge structure over~$\tilde{U}_i$ given by $R^2\pi_{i,*}\mQ_{\cY_i}$. We claim that $\mH_1$ is not isotrivial. To see this, assume the opposite. By the irreducibility of~$M$ and the fact that the morphisms~$p_i$ are open, $V = p_1(U_1) \cap p_2(U_2)$ is an open dense substack of~$M_\red$, and on $(p_2\circ r_2)^{-1}(V) \subset \tilde{U}_2$ the variation~$\mH_2$ is isotrivial. This implies that $\mH_2$ is isotrivial on all of~$\tilde{U}_2$, but this gives a contradiction with our assumption that $H^2(X_1,\mQ) \not\cong H^2(X_2,\mQ)$. This proves the claim.

The only thing that is left to do is the reduction to a projective family of surfaces. Choose a point $t \in \tilde{U}_1$ such that  $H^2(X_1,\mQ) \not\cong H^2(\cY_{1,t},\mQ)$ as Hodge structures, and choose a morphism $S \to \tilde{U}_1$ from a non-singular irreducible curve~$S$ to~$\tilde{U}_1$ such that $u_1$ and~$t$ are in the image of~$S$. By pull-back this gives a smooth family $f\colon \cX \to S$ such that $X \cong \cX_\xi$ for some $\xi \in S(\mC)$ and such that the period map associated with $R^2f_*\mQ_\cX$ is not constant. Choose an ample divisor~$D$ on~$X$ and let $\cD \subset \cX$ be its flat closure. Over a Zariski-open subset $S^\circ \subset S$ containing~$\xi$ this~$\cD$ is relatively ample, and the proposition follows by applying Theorem~\ref{MainThm} to the restriction of the family~$\cX$ to~$S^\circ$.
\end{proof}

\begin{corollary}
\label{SimultResCor}
Let $M$ be an irreducible component of~$\cM$, and suppose there is a complex surface~$Y$ with $[Y] \in M$ and Picard number $\rho(Y) = h^{1,1}(Y)$. Then for any non-singular surface~$X$ with $[X] \in M$, the Tate Conjecture for divisor classes and the Mumford-Tate conjecture for the cohomology in degree~$2$ are true.
\end{corollary}

\begin{proof}
If the Picard number $\rho(X)$ satisfies $\rho(X) < h^{1,1}(X)$ then $M$ satisfies condition\/~{\rm \PeriodCond} and the result follows from the proposition. If $\rho(X) = h^{1,1,}(X)$ then the motive $\motH = \motH^2(X)\bigl(1\bigr)$ decomposes as $\motH = \unitmot^\rho \oplus \motH^\trans$ with $2$-dimensional transcendental part, on which we have a symmetric bilinear polarization form~$\phi$. If $H^\trans$ denotes the vector space underlying the Betti realization, $G_\Betti(\motH) \cong G_\Betti(\motH^\trans) = \SO(H^\trans,\phi)$ and $G_\ell^0(\motH) \cong G_\ell^0(\motH^\trans)$ is a connected algebraic subgroup of $\SO(H^\trans,\phi) \otimes \Ql \cong \SO_{2,\Ql}$. So it only remains to show that $G_\ell^0(\motH)$ is not trivial, which follows by looking at the Hodge-Tate decomposition at a prime of good reduction.
\end{proof}

\begin{theorem}
\label{MainThmSurf}
Let $X$ be a complex algebraic surface of general type with $p_g(X) =1$. The Tate Conjecture for divisor classes on~$X$ and the Mumford-Tate conjecture for the cohomology in degree~$2$ are true if the minimal model of~$X$ is of one of the following types. 
\begin{enumerate}[label=\textup{(\alph*)}]
\item Surfaces with $q=0$ and $K^2\leq 2$.
\item Surfaces with $q=0$ and $3\leq K^2 \leq 8$ that lie in the same moduli component as a Todorov surface.
\item Surfaces with $q=0$ and $K^2=3$ with torsion (of the Picard group) $\mZ/3\mZ$.
\item Surfaces with $q=1$ and $K^2 = 2$.
\item Surfaces with $q=1$, $K^2 = 3$ and general albanese fibre of genus~$3$.
\item Surfaces with $q=1$ and $K^2 = 4$ in any of the eight moduli components described by Pignatelli in \textup{\cite{Pignatelli}}.
\item Surfaces with $q=1$ and $K^2 = 8$ whose bicanonical map is not birational. 
\end{enumerate}
\end{theorem}

\subsection{}
\label{PfSurfThmStart}
In most cases, the verification that the relevant component of the moduli space satisfies condition~\PeriodCond\ is a matter of quoting some facts from the literature. We treat these cases first. After this we shall turn to cases~(f) and~(g), which require more work.

For Todorov surfaces (which pertains to~(b) and also to surfaces with $q=0$, $K^2=2$ and torsion $\mZ/2\mZ$) the result follows from the Torelli theorem for K3's; cf.\ \cite{Todorpq=10}, Section~4 (where the meaning of the term ``moduli space'' is not the standard one), or \cite{Morrison}, Theorem~7.3.

For case~(a) the assertion follows from the results in \cite{CatK2=1}, \cite{CatPeriod} and~\cite{TodorK2=1}. For case~(c), see~\cite{Murakami}.

Next we turn to minimal surfaces with $q=1$. In this case $2 \leq K^2 \leq 9$. First consider the case $K^2=2$. It was shown by Catanese in \cite{Cat211} that the moduli stack $\cM_{2,1,1}$ is irreducible of dimension~$7$ and that it satisfies condition~\PeriodCond. (This last fact can also be seen from the examples given by Polizzi in \cite{Polizzi}, Section~7.3.) 

Next assume $K^2 =3$. As shown by Catanese and Pignatelli in~\cite{CatPig}, Section~6, the moduli stack $\cM_{3,1,1}$ has four irreducible components, each of dimension~$5$. One of these components parametrizes surfaces whose albanese fibres have genus~$3$; these have been studied in detail by Catanese and Ciliberto; see \cite{CatCil311}, \cite{CatCilJAG}. By \cite{Polizzi311}, Corollary~6.16 and Proposition~6.18, there exist such surfaces~$X$ for which the Picard number equals $h^{1,1}(X)$, and the result follows by Corollary~\ref{SimultResCor}. (For surfaces~$X$ in this component with ample canonical class, the Tate conjecture was proven by Lyons~\cite{Lyons}; when combined with the results in our Section~\ref{Zarhlad}, his methods also give the Mumford-Tate conjecture.)

\subsection{}
Let now $S$ be a surface in one of the eight irreducible components of~$\cM_{4,1,1}$ that are described in~\cite{Pignatelli}. (We switch to the letter~$S$ for the surface we want to study, to facilitate references to the literature.) In everything that follows we assume $S$ to be general in its component of the moduli space; this means that the properties we state are valid for $S$ in a Zariski-open subset. Let $\alpha \colon S \to B$ be the albanese morphism, and define $V_n = \alpha_*(\omega_S^n)$. (Note that $\omega_S = \omega_{S/B}$.) The surfaces that we are considering are characterized by the fact that $V_2$ is a sum of three line bundles. The general albanese fibre has genus~$2$.

The work of Pignatelli (which builds upon the results of Catanese and Pignatelli in~\cite{CatPig}) gives a beautiful geometric description of the surfaces~$S$ in question. The main ingredients for our discussion are summarized in Figure~1. Here $\cC \subset \mP(V_2)$ is a conic bundle that has two $A_1$-singularities lying over two distinct points $P_1$ and~$P_2$ of~$B$. If $\sigma\colon \cCtilde \to \cC$ is the minimal resolution then $\cCtilde$ is a blow-up of~$\mP(V_1)$ in four points, two above each~$P_i$. We denote by $E_{i,1}$ and~$E_{i,2}$ the exceptional fibres of $\beta\colon \cCtilde\to \mP(V_1)$ above~$P_i$. Let $\cE_i \subset \cCtilde$ be the strict transform of the fibre of~$\mP(V_1)$ above~$P_i$; then $\cE_1$ and~$\cE_2$, which are $(-2)$-curves, are the two exceptional fibres of~$\sigma$. The morphism $\cCtilde \to \cC \hookrightarrow \mP(V_2)$, seen as a rational map $\mP(V_1) \ratarrow \mP(V_2)$ is the relative Veronese morphism, and we have a short exact sequence $0 \tto \Sym^2(V_1) \tto V_2 \tto \cO_{\{P_1,P_2\}} \tto 0$.

\begin{figure}
\centering
\begin{tikzpicture}[scale=.8]
\draw (0,6) node[below left] {$\mP(V_1)$} rectangle (4,4);
\draw (0.4,4.1) -- (0.4,5.9);
\draw (1.1,4.1) -- (1.1,5.9);
\fill (1.1,4.5) circle[radius=1pt];
\fill (1.1,5.5) circle[radius=1pt];
\draw (1.8,4.1) -- (1.8,5.9);
\draw (2.2,4.1) -- (2.2,5.9);
\draw (2.9,4.1) -- (2.9,5.9);
\fill (2.9,4.5) circle[radius=1pt];
\fill (2.9,5.5) circle[radius=1pt];
\draw (3.6,4.1) -- (3.6,5.9);
\draw (0,9.5) node[below left] {$\cCtilde$} rectangle (4,7.5);
\draw (0.4,7.6) -- (0.4,9.4);
\draw (0.9,7.8) -- (0.9,9.2) node[pos=.5,right] {$\scriptstyle \!\cE_1$};
\draw (0.8,8.2) -- (1.1,7.6) node[pos=.5,right] {$\scriptstyle E_{1,2}$};
\draw (0.8,8.8) -- (1.1,9.4) node[pos=.5,right] {$\scriptstyle E_{1,1}$};
\draw (1.8,7.6) -- (1.8,9.4);
\draw (2.2,7.6) -- (2.2,9.4);
\draw (2.7,7.8) -- (2.7,9.2) node[pos=.5,right] {$\scriptstyle \!\cE_2$};
\draw (2.6,8.2) -- (2.9,7.6) node[pos=.5,right] {$\scriptstyle E_{2,2}$};
\draw (2.6,8.8) -- (2.9,9.4) node[pos=.5,right] {$\scriptstyle E_{2,1}$};
\draw (3.6,7.6) -- (3.6,9.4);
\draw[->] (2,7.25) -- (2,6.25) node[pos=.5,left] {$\beta$};
\draw (0,13) node[below left] {$\tilde{S}$} rectangle (4,11);
\draw (0.4,11.1) -- (0.4,12.9);
\draw (0.9,11.3) -- (0.9,12.7) node[pos=.5,right] {$\scriptstyle \!\!\cF_1$};
\draw (0.8,11.7) -- (1.1,11.1);
\draw (0.8,12.3) -- (1.1,12.9);
\draw (1.8,11.1) -- (1.8,12.9);
\draw (2.2,11.1) -- (2.2,12.9);
\draw (2.7,11.3) -- (2.7,12.7) node[pos=.5,right] {$\scriptstyle \!\!\cF_2$};
\draw (2.6,11.7) -- (2.9,11.1);
\draw (2.6,12.3) -- (2.9,12.9);
\draw (3.6,11.1) -- (3.6,12.9);
\draw[->] (2,10.75) -- (2,9.75) node[pos=.5,left] {$\tilde\phi$};
\draw (7,9) node[below left] {$\cC$} rectangle (11,7);
\draw (7.4,7.1) -- (7.4,8.9);
\draw (7.8,8.3) -- (8.1,7.1);
\draw (7.8,7.7) -- (8.1,8.9);
\fill (7.88,8) circle[radius=1pt];
\draw (8.8,7.1) -- (8.8,8.9);
\draw (9.2,7.1) -- (9.2,8.9);
\draw (9.6,8.3) -- (9.9,7.1);
\draw (9.6,7.7) -- (9.9,8.9);
\fill (9.68,8) circle[radius=1pt];
\draw (10.6,7.1) -- (10.6,8.9);
\draw (7,12.5) node[below left] {$S$} rectangle (11,10.5);
\draw (7.4,10.6) -- (7.4,12.4);
\draw (7.8,11.8) -- (8.1,10.6);
\draw (7.8,11.2) -- (8.1,12.4);
\fill (7.88,11.5) circle[radius=1pt];
\node at (8.1,11.5) {$\scriptstyle q_1$};
\draw (8.8,10.6) -- (8.8,12.4);
\draw (9.2,10.6) -- (9.2,12.4);
\draw (9.6,11.8) -- (9.9,10.6);
\draw (9.6,11.2) -- (9.9,12.4);
\fill (9.68,11.5) circle[radius=1pt];
\node at (9.9,11.5) {$\scriptstyle q_2$};
\draw (10.6,10.6) -- (10.6,12.4);
\draw[->] (9,10.25) -- (9,9.25) node[pos=.5,left] {$\phi$};
\draw[->] (4.3,12) -- (6.7,11.5) node[pos=.5,below] {$\rho$};
\draw[->] (4.3,8.5) -- (6.7,8) node[pos=.5,below] {$\sigma$};
\draw (3.5,2.5) -- (7.5,2.5) node[right] {$B$};
\fill (4.6,2.5) circle[radius=1pt];
\node at (4.6,2.3) {$\scriptstyle P_1$};
\fill (6.4,2.5) circle[radius=1pt];
\node at (6.4,2.3) {$\scriptstyle P_2$};
\draw[->] (3,3.7) -- (4,3) node[pos=.5,below] {$\pi_1$};
\draw[->] (8,6.7) -- (6.5,3);
\draw[right hook->] (11.3,8) -- (11.8,8) node[right] {$\mP(V_2)$};
\end{tikzpicture}
\caption{The geometry of $S$ as a double cover of a conic bundle.}
\end{figure}
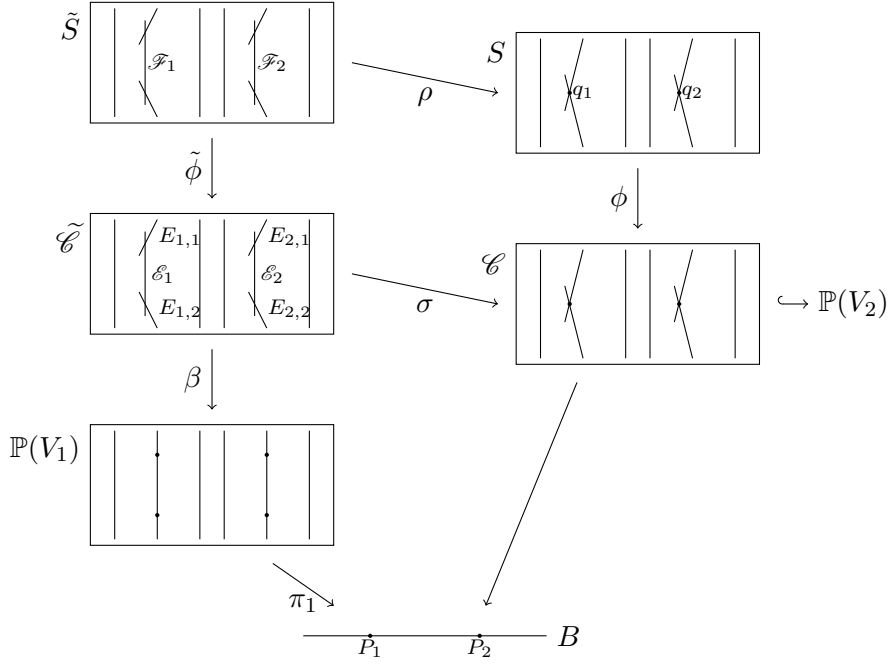

The surface~$S$ is a double cover of~$\cC$ with $\phi \colon S\to \cC$ branched over the two singular points of~$\cC$ and a divisor~$\Delta$ that (for $S$ general) is a non-singular curve of genus~$4$ in~$\cC$, not passing through the singular points. This divisor~$\Delta$ is obtained as the intersection of~$\cC$ with a relative cubic hypersurface $\cG \subset \mP(V_2)$. Let $q_1$, $q_2 \in S$ denote the two points over the singular points of~$\cC$. With $\tilde{S} = \cCtilde \times_\cC S$, the morphism $\rho \colon \tilde{S} \to S$ is the blow-up of the points~$q_i$. We denote by~$\cF_i$ the exceptional fibre above~$q_i$. The morphism $\tilde\phi \colon \tilde{S} \to \cCtilde$ is a double cover branched over $\tilde\Delta + \cE_1 + \cE_2$, where $\tilde\Delta \subset \cCtilde$ denotes the strict transform of~$\Delta$ under~$\sigma$. (Of course $\tilde\Delta \isomarrow \Delta$, as $\Delta$ does not pass through the singular points of~$\cC$.)

The key geometric fact needed for the proof of Theorem~\ref{MainThmSurf}(f) is the following.

\begin{proposition}
\label{GammaCupSigma}
Let $\Gamma$ be the unique effective canonical divisor of~$S$, and let $\Sigma \subset S$ be the critical locus of the albanese morphism $\alpha\colon S \to B$. Then for $S$ general in its moduli component, $\Gamma \cap \Sigma = \{q_1,q_2\}$ (scheme-theoretically).
\end{proposition}

\begin{proof}
The first fact we shall use is that $\Sigma$ consists of the two isolated points $q_1$ and~$q_2$, together with the points of~$S$ lying above the critical points of the morphism $\Delta \to B$. See \cite{Pignatelli}, Section~5.

The other thing we need is a concrete description of~$\Gamma$. For this we start by noting that the relative canonical map $S \ratarrow \mP(V_1)$ is not defined precisely at $q_1$ and~$q_2$, and that the morphism $\beta\circ \tilde\phi \colon \tilde{S} \to \mP(V_1)$ resolves these indeterminacies. This means that $(\beta \circ \tilde\phi)^* \cO_{\mP(V_1)}(1) = \rho^*\omega_S(-\cF_1-\cF_2) = \omega_{\tilde{S}}(-2\cF_1-2\cF_1)$. Hence $\omega_{\tilde{S}} = \tilde\phi^*\bigl(\beta^*\cO_{\mP(V_1)}(1) \otimes \cO_{\cCtilde}(\cE_1+\cE_2)\bigr)$. As $h^0\bigl(\cCtilde,\beta^*\cO_{\mP(V_1)}(1)\bigr) = h^0(B,V_1) = 1$, there is a unique effective divisor~$\Xi$ on~$\cCtilde$ representing $\beta^*\cO_{\mP(V_1)}(1)$. As we shall see, for $S$ general, $\Xi$ intersects $\cE_1$ and~$\cE_2$ transversally and does not meet the exceptional fibres~$E_{i,j}$ of $\beta\colon \cCtilde \to \mP(V_1)$. Let $\tilde\Gamma$ denote the pullback of~$\Xi$ to~$\tilde{S}$, so that $\tilde\Gamma+2\cF_1+2\cF_2$ is an effective canonical divisor of~$\tilde{S}$. The image $\Gamma = \rho(\tilde\Gamma)$ of~$\tilde\Gamma$ in~$S$ passes through the points $q_1$ and~$q_2$ and has multiplicity~$1$ in these points; further, $\tilde\Gamma$ is the strict transform of~$\Gamma$ and $\rho^*(\Gamma) = \tilde\Gamma + \cF_1+\cF_2$, so that indeed $\Gamma$~is the unique effective canonical divisor of~$S$.

To describe~$\Xi$ we have to distinguish two cases. In four of the eight families, $V_1$ is a sum of two line bundles: $V_1 = \cO_B(p) \oplus \cO_B(O-p)$ for some $p \in B$. The projection $V_1 \twoheadrightarrow \cO_B(O-p)$ defines a section of $\mP(V_1) \to B$, which for general~$S$ does not pass through the points that are blown up in~$\cCtilde$. Hence the section lifts to a section $B \to \cCtilde$. If $\Theta \subset \cCtilde$ denotes its image, $\beta^*\cO_{\mP(V_1)}(1)$ is represented by the divisor $\Xi = \Theta + F_p$, where $F_p$ denotes the fibre above~$p$. (Note that $\Theta$ is the relative hyperplane defined by the ``equation'' $\cO_B(p) \hookrightarrow V_1$. Generally, if $L \hookrightarrow V_1$ is the inclusion of a line bundle with locally free quotient then the corresponding divisor of~$\mP(V_1)$ lies in the class $\cO_{\mP(V_1)}(1) \otimes \pi_1^* L^{-1}$.)

In the relative coordinates used by Pignatelli (see \cite{Pignatelli}, Section~2), $\Theta$ is given (on~$\mP(V_1)$) by the equation $x_0 = 0$. Its image in $\cC \subset \mP(V_2)$ is given by the equations $y_2=y_3=0$. Now it is immediate from the equations for $\Delta = \cC \cap \cG$ given in loc.\ cit., Table~$3$, that for a general choice of~$\cG$, the image of~$\Theta$ is disjoint from~$\Delta$. It then remains to consider $\Delta \cap F_p$, where now~$F_p$ denotes the fibre above~$p$ of~$\cC$. As shown by Pignatelli, the critical locus of $\Delta \to B$ is contained in the relative hyperplane of~$\mP(V_2)$ given by $y_3=0$. This hyperplane intersects~$F_p$ in two points, which for a general choice of~$\cG$ do not lie on~$\cG$. This proves the proposition for the families with $V_1$ decomposable.

Next suppose $S$ occurs in one of the other four families. In this case $V_1$ is the unique rank~$2$ bundle on~$B$ with determinant~$\cO_B(O)$, which sits in a non-split short exact sequence $0 \tto \cO_B \tto V_1 \tto \cO_B(O) \tto 0$. The projection $V_1 \twoheadrightarrow \cO_B(O)$ defines a section of $\cCtilde \to B$, whose image is~$\Xi$.

Let $\eta_1$, $\eta_2$, $\eta_3$ be the points of order~$2$ on~$B$. Choose rational functions~$F_i$ with $\div(F_i) = 2\eta_i - 2O$ and such that $F_1+F_2+F_3=0$. We have $\Sym^2(V_1) \cong \cO_B(\eta_1) \oplus \cO_B(\eta_2) \oplus \cO_B(\eta_3)$, which gives relative homogeneous coordinates $(u_1:u_2:u_3)$ on $\mP(\Sym^2(V_1))$, and the image of the relative Veronese map $\mP(V_1) \to \mP(\Sym^2(V_1))$ is given by the equation $F_1^{-1}u_1^2 + F_2^{-1} u_2^2 + F_3^{-1}u_3^2=0$. The image of~$\Xi$ in $\mP(\Sym^2(V_1))$ is given by the section $(F_1:F_2:F_3)$. There are three divisors~$D_i$ of~$B$ such that $V_2 = \cO_B(D_1) \oplus \cO_B(D_2) \oplus \cO_B(D_3)$; these divisors and the multiplication map $\Sym^2(V_1) \to V_2$ are given as in \cite{Pignatelli}, Table~2. In what follows we shall simply write $a$, $b$, $c$,~$d$ for the coefficients that appear in the matrix of the multiplication map and that in loc.\ cit.\ are called $a_j$, $b_j$, $c_j$,~$d_j$ ($j=5,\ldots,8$), with the understanding that $b=c=0$ in the family called~$\cM_{i,3}$. The decomposition of~$V_2$ as a sum of line bundles gives us relative homogeneous coordinates $(y_1:y_2:y_3)$ on~$\mP(V_2)$, and we find that $\cC \subset \mP(V_2)$ is given by 
\[
F_1^{-1}(ay_1+cy_2)^2 + F_2^{-1}(by_1+dy_2)^2 + F_3^{-1}y_3^2 = 0\, .
\]
(This corrects a mistake in~\cite{Pignatelli}; the coordinates~$z_i$ that Pignatelli uses are meaningful only \'etale locally on~$B$. This means that the equation for~$\cC$ in his Table~3 has to be changed, but otherwise this does not affect his results.)

The singular points of~$\cC$ and the critical locus of $\Delta \to B$ are both contained in the relative hyperplane of~$\mP(V_2)$ given by $y_3=0$. The image $\sigma(\Xi)$ of~$\Xi$ is given by a section of~$\mP(V_2)$. As $\sigma(\Xi)$ contains the singular points (because $\Xi$ meets~$\cE_1$ and~$\cE_2$), we are done if we show that $\sigma(\Xi)$ does not meet the hyperplane $y_3=0$ in other points. It is easiest to do the calculation on~$\mP(\Sym^2(V_1))$ and to show that the image of~$\Xi$ there does not intersect the relative hyperplane given by $u_3=0$. (Note that $D_3=\eta_3$ and that the multiplication map $\Sym^2(V_1) \to V_2$ is the identity on the third summands.) Now, $u_3=0$ corresponds to the inclusion $\cO_B(\eta_3) \hookrightarrow \Sym^2(V_1)$, whereas the image of~$\Xi$ in~$\mP(\Sym^2(V_1))$ is given by a surjection $\Sym^2(V_1) \twoheadrightarrow \cO_B(2O)$. The composition $\cO_B(\eta_3) \to \cO_B(2O)$ is given by the inclusion $\cO_B(\eta_3) \subset \cO_B(2\eta_3)$, followed by the isomorphism $F_3^{-1} \colon \cO_B(2\eta_3) \isomarrow \cO_B(2O)$; hence indeed the image of~$\Xi$ is disjoint from the hyperplane $u_3=0$, and we are done.
\end{proof}

\begin{remark}
If $V_1$ is decomposable, $\Delta \cdot F_p = 6$, and it follows that for a general member of the first four families, $\tilde\Gamma$ is the union of two curves of genus~$2$, intersection transversally in two points. (Hence indeed $p_a(\tilde\Gamma) = p_a(\Gamma)=5$.) One component of~$\tilde\Gamma$ is the inverse image of the genus~$1$ curve~$\Theta$, which meets the branch locus only in its intersection points with $\cE_1$ and~$\cE_2$; the other component is the inverse image of the rational curve~$F_p$, which intersects the branch locus in six points. The two components intersect in the points lying over the point $\Theta \cap F_p$.

If $V_1$ is indecomposable, the genus~$1$ curve~$\Xi$ intersects the branch locus of $\tilde{S} \to \cCtilde$ in its intersection points with $\cE_1$ and~$\cE_2$ (the points lying over the singular points of~$\cC$) and six other points. In this case, $\tilde\Gamma \cong \Gamma$ is irreducible of genus~$5$.
\end{remark}

\subsection{}
We now complete the proof of Theorem~\ref{MainThmSurf}(e). Again we assume $S$ is general in its component of the moduli space. As $\cT_S \cong \Omega^1_S(-\Gamma)$, we get an exact sequence
\[
0 \tto H^0(S,\Omega^1_S) \tto H^0(\Gamma,\Omega^1_S|_\Gamma) \tto H^1(S,\cT_S) \tto H^1(S,\Omega^1_S)\, ,
\]
in which the last map sends a class in $H^1(S,\cT_S)$ to its cup-product with a non-zero $2$-form (which is unique up to scalars). Our goal is to show that this map is non-zero, as this implies that the Hodge structure on the~$H^2$ is not constant over the moduli component containing~$S$. Because all families that we are considering have dimension~$4$ or bigger and $h^0(\Omega^1_S)=1$, it suffices to show that $h^0(\Gamma,\Omega^1_S|_\Gamma) \leq 4$. 

Next let $\mu$ be a non-zero global $1$-form on~$B$ and consider the exact sequence
\[
0 \tto \cO_S \mapright{\alpha^*(\mu)} \Omega^1_S \mapright{-\wedge \alpha^*(\mu)} \omega_S \tto \omega_S|_\Sigma \tto 0\, ,
\]
where, as before, $\Sigma$ denotes the critical locus of the albanese map~$\alpha$. (As shown by Pignatelli, $\Sigma$ is finite; counting Euler characteristics, we find that it has length~$8$.) Breaking this up in two short exact sequences and restricting the first to~$\Gamma$ we get $0 \tto \cO_\Gamma \tto \Omega^1_S|_\Gamma \tto \cI_\Sigma\omega_S|_\Gamma \tto 0$, where $\cI_\Sigma \subset \cO_S$ is the ideal sheaf of~$\Sigma$. As $\cO_S(-\Gamma) = \omega_S^{-1}$, Kodaira vanishing gives $h^0(\Gamma,\cO_\Gamma)= 1$. (This is also clear from the concrete description of~$\Gamma$ obtained in the proof of Proposition~\ref{GammaCupSigma}.) It therefore suffices to show that $h^0(\Gamma,\cI_\Sigma\omega_S|_\Gamma) \leq 3$.

{}From the second short exact sequence we get a diagram
\[
\begin{matrix}
&&0&&0&\cr
&&\mapdown{}&&\mapdown{}&\cr
0 & \tto & \cI_\Sigma &\tto & \cI_\Sigma\omega_S & \tto & \cI_\Sigma\omega_S|_\Gamma & \tto & 0\cr
&&\mapdown{} && \mapdown{} && \mapdown{} &\cr
0 & \tto & \cO_S & \tto & \omega_S & \tto & \omega_S|_\Gamma & \tto & 0 &\cr
&&\mapdown{} && \mapdown{} && \mapdown{} &\cr
&&\cO_\Sigma & \tto & \omega_S|_\Sigma & \tto & \omega_S|_{\Gamma\cap\Sigma} & \tto & 0 \cr
&&\mapdown{} && \mapdown{} && \mapdown{} &\cr
&& 0 && 0 && 0 &\cr
\end{matrix}
\]
with exact rows and columns. The map $\cO_\Sigma \to \omega_S|_\Sigma$ is the restriction of the map $\cO_S \to \omega_S$ given by the global $2$-form on~$S$; hence its kernel is~$\cO_{\Gamma\cap\Sigma}$. The middle row, together with $q(S) =1$, gives that $h^0(\Gamma,\omega_S|_\Gamma) \leq 1$ (in fact, it is equal to~$1$). By the snake lemma and Proposition~\ref{GammaCupSigma}, the desired estimate $h^0(\Gamma,\cI_\Sigma\omega_S|_\Gamma) \leq 3$ follows. This settles case~(f) of Theorem~\ref{MainThmSurf}.

\subsection{}
For the surfaces in~(g), which make up three irreducible components of~$\cM_{8,1,1}$, we use the results in \cite{Polizzi811} and~\cite{Borrelli}. The surfaces in question are of the form $S = (C\times F)/G$, where $C$ and~$F$ are curves with a faithful action of a finite group~$G$ such that the diagonal action on $C \times F$ is free. Hence $\motH^2(S) \cong \motH^2(C) \oplus \motH^2(F) \oplus [\motH^1(C) \otimes \motH^1(F)]^G$. 

In all examples, it follows without difficulty from the description given in \cite{Polizzi811}, Section~4, that all simple factors of the Jacobians $J_C$ and~$J_F$ have dimension at most~$2$. By \cite{Lombardo}, Corollary~4.5, it follows that the Mumford-Tate conjecture for $J_C \times J_F$ is true. By Proposition~\ref{MTCNorms}(\romannumeral4) and Remark~\ref{MTCforV(n)}(\romannumeral1), this implies the Mumford-Tate conjecture for  $\motH^2(S)$.

This completes the proof of Theorem~\ref{MainThmSurf}.

{\small

\bigskip

} 

\noindent
\texttt{b.moonen@science.ru.nl}

\noindent
Radboud University Nijmegen, IMAPP, Nijmegen, The Netherlands

\end{document}